\documentclass[final,12pt]{arxiv2025}

\usepackage{graphicx}
\usepackage{epstopdf}

\usepackage{enumitem,amsfonts,amssymb,mathtools,bm,pifont,bbm,listings,threeparttable}
\usepackage{ulem,hyperref,comment,amsmath}

\usepackage{natbib}
\usepackage{multirow}

\usepackage{algorithm}

\newtheorem{assumption}[theorem]{Assumption}

\def\trans{^\mathsf{T}}\def\Rn{\mathbb{R}}  \def\zero{\mathbf{0}}
\def\trans {^\mathsf{T}}
\def\fro {\mathsf{F}}
\def\hG {\hat{G}}

\definecolor{mydarkgreen}{RGB}{39,130,67}
\definecolor{mydarkred}{RGB}{192,25,25}

\usepackage{relsize} 

\DeclareMathOperator{\tr}{tr}

\newcommand{\beq}{\begin{eqnarray}}
\newcommand{\eeq}{\end{eqnarray}}
\newcommand{\beqq}{\begin{equation}}
\newcommand{\eeqq}{\end{equation}}
\newcommand{\bel}{\begin{align}}
\newcommand{\eel}{\end{align}}
\newcommand{\la}{\langle}
\newcommand{\ra}{\rangle}
\newcommand{\noi}{\noindent}
\newcommand{\nn}{\nonumber}
\def\noi{\noindent}
\def\nn{\nonumber}
\def\la{\langle}
\def\ra{\rangle}

\newcommand{\step}[1]{\text{\ding{\numexpr#1+171\relax}}}

\def\epsilons{\boldsymbol{\epsilon}}

\def\zetas{\boldsymbol{\zeta}}

\def\ts{\textstyle}

\def\a{\mathbf{a}}\def\b{\mathbf{b}}\def\q{\mathbf{q}}\def\A{\mathbf{A}}\def\B{\mathbf{B}}\def\D{\mathbf{D}}\def\G{\mathbf{G}}\def\I{\mathbf{I}}\def\J{\mathbf{J}}\def\M{\mathbf{M}}\def\R{\mathbf{R}}\def\S{\mathbf{S}}\def\U{\mathbf{U}}\def\V{\mathbf{V}}\def\W{\mathbf{W}}\def\X{\mathbf{X}}\def\Y{\mathbf{Y}}\def\Z{\mathbf{Z}}

\def\FF{\mathcal{F}}\def\OO{\mathcal{O}}\def\ZZ{\mathcal{Z}}

\def\AAA{\mathbb{A}}\def\EEE{\mathbb{E}}\def\GGG{\mathbb{G}}\def\III{\mathbb{I}}\def\JJJ{\mathbb{J}}\def\KKK{\mathbb{K}}\def\MMM{\mathbb{M}}\def\NNN{\mathbb{N}}\def\RRR{\mathbb{R}}\def\SSS{\mathbb{S}}\def\UUU{\mathbb{U}}\def\VVV{\mathbb{V}}\def\WWW{\mathbb{W}}

\def\NO{\ding{56}}
\def\YES{\ding{52}}

\renewcommand{\X}{\mathbf X}
\renewcommand{\G}{\mathbf G}

\renewcommand{\M}{\mathbf M}
\renewcommand{\V}{\mathbf V}

\newcommand{\Orth}{\operatorname{Orth}}
\renewcommand{\fro}{\mathrm F}
\newcommand{\eps}{\varepsilon}

\usepackage{graphicx}
\usepackage{amssymb} 

\renewcommand{\Orth}{\operatorname{Orth}}

\newcommand{\xxxx}[2]{\if\relax\detokenize\expandafter{\romannumeral-`\q#1}\relax#2\else \overset{\step{#1}}{#2}\fi}

\title[OptMuon]{OptMuon: Closed-Loop Orthogonalized Momentum Methods for Stochastic Optimization with Zero-Noise Optimality}

\usepackage{times}
\coltauthor{
 \Name{Ganzhao Yuan} \Email{yuanganzhao@foxmail.com}\\
 \addr Faculty of Computer Science and Artificial Intelligence\\
 Shenzhen University of Advanced Technology (SUAT), China
}

\begin{document}

\maketitle

\begin{abstract}
Orthogonalized momentum updates, as used in Muon-style optimizers, have recently shown strong empirical stability in large-scale deep learning. However, most current orthogonalized methods are still paired with fixed, externally scheduled, or otherwise open-loop magnitude rules, so their scale is not directly calibrated from the realized optimization trajectory. Motivated by the closed-loop perspective behind Lipschitz-free and noise-adaptive methods, we propose OptMuon, a family of adaptive momentum orthogonalization methods for stochastic nonconvex optimization. OptMuon combines Muon-style polar-factor directions with a trajectory-dependent AdaGrad-Norm-type coefficient schedule, so that the update magnitude is determined by the observed gradient and momentum history rather than by a prescribed Lipschitz-dependent rule. The schedule does not use the smoothness constant, the variance level, or the bounded-gradient constant in parameter selection, and its running-maximum correction prevents isolated gradient spikes from causing excessive coefficient collapse. Under lower-boundedness, unbiased stochastic gradients with bounded variance, smoothness, and an almost-sure bounded stochastic-gradient condition, we prove two complementary expected-stationarity guarantees. OptMuon-A achieves the noise-adaptive rate
\(\tilde{\mathcal O}(T^{-1/2}+\sigma^{1/2}T^{-1/4})\)
under average smoothness, while OptMuon-I achieves
\(\tilde{\mathcal O}(T^{-1/2}+\sigma^{1/3}T^{-1/3})\)
under individual smoothness. In the zero-noise regime, both bounds automatically reduce to a nearly optimal deterministic first-order rate \(\tilde{\mathcal O}(T^{-1/2})\) without manual hyperparameter retuning. These results show that closed-loop scalar adaptation can be combined with Muon-style momentum orthogonalization while retaining noise adaptivity and zero-noise optimality up to logarithmic factors.
\end{abstract}


\section{Introduction}

We study stochastic first-order methods for the matrix-valued nonconvex problem
\begin{equation}\label{eq:main}
\min_{\X\in\mathbb{R}^{m\times n}} f(\X):=\mathbb{E}_{\xi\sim\mathcal{D}}[f(\X;\xi)],
\end{equation}
where $n\leq m$, $\xi$ is sampled independently from an unknown distribution $\mathcal D$. The performance criterion is the standard expected stationarity measure: after $T$ iterations, return a random iterate $\X_{\rm out}$ with small $\EEE\|\nabla f(\X_{\rm out})\|_{\mathrm F}$. The matrix notation is chosen to match Muon-style layerwise updates; vector-valued parameters can be handled by the usual reshaping or blockwise formulation.  The analysis assumes lower boundedness, unbiased stochastic gradients with bounded variance, smoothness, and an explicit almost-sure bounded stochastic-gradient condition.

Stochastic gradient methods remain the computational backbone of modern machine learning \cite{robbins1951stochastic,polyak1964some}. Classical stochastic approximation, SGD, and momentum methods are simple and scalable, whereas adaptive methods such as AdaGrad \citep{duchi2011adaptive}, Adam  \citep{kingma2015adam}, and AdaGrad-Norm  \citep{ward2020adagrad} improve robustness by using observed gradient statistics to calibrate update magnitudes. A persistent difficulty is that the search direction, the step magnitude, and the noise level are often controlled by the same raw stochastic-gradient vector. This coupling makes hyperparameter choices sensitive to the smoothness scale and to changes in variance across training phases.

A complementary line of work modifies the geometry of the update direction. Muon-style methods transform a matrix-valued momentum into an approximately orthogonal direction, typically through a polar-factor computation or a small number of Newton--Schulz iterations \citep{kim2026convergencens,chang2025convergence,huang2025limuon}. This operation separates the orientation of the update from anisotropic singular-value scaling. Existing Muon analyses and implementation-oriented variants often either use a fixed or open-loop global magnitude rule, or focus on the accuracy and efficiency of the orthogonalization map itself. Recent adaptive Muon variants have begun to investigate norm-based, coordinate-wise, trust-region, and distance-aware scaling mechanisms \citep{zhang2025adagradmeetsmuon,si2025adamuon,zhang2026adamimprovesmuon,demidovich2026distance}, which further motivates a systematic study of closed-loop scalar adaptation for orthogonalized momentum.

In a separate line of theory, recursive variance-reduced methods \citep{nguyen2017sarah,nguyen2017nonconvexsarah,fang2018spider,cutkosky2019momentum} and AdaGrad-Norm-type methods \cite{ward2020adagrad,levy2021storm,liu2022meta} adapt the estimator or the update magnitude from the observed trajectory rather than from a prescribed smoothness-dependent schedule. Their most relevant feature for the present paper is noise adaptivity: the stochastic term in the convergence bound improves automatically as the variance parameter $\sigma$ decreases, and the deterministic first-order rate is recovered when $\sigma=0$. The central question addressed here is whether this closed-loop scalar adaptation can be combined with Muon-style orthogonalized momentum without losing the geometric stability of the polar-factor direction.

This paper proposes {OptMuon}, a closed-loop orthogonalized momentum framework for stochastic nonconvex optimization.  The direction is the polar factor $\Orth(\M_t)$ of a momentum matrix $\M_t$, while the magnitude is governed by a scalar coefficient computed from the observed gradient and momentum history.  The coefficient has an AdaGrad-Norm flavor: accumulated stochastic-gradient energy suppresses steps in noisy regimes, whereas the running maximum in the numerator prevents the coefficient from collapsing after isolated gradient spikes.  The method is Lipschitz-free in the algorithmic sense: its parameters do not require the smoothness constant $L$, the variance $\sigma^2$, or the bounded-gradient constant $\hG$, even though these quantities appear in the proof constants.

We analyze two variants. \textsc{OptMuon-A} uses one stochastic gradient per iteration and is analyzed under average smoothness. \textsc{OptMuon-I} uses a STORM-type recursive correction evaluated on the same sample at two consecutive points and is analyzed under individual smoothness.  Under the assumptions in Section~\ref{sect:prel}, the guarantees in Section~\ref{sect:rate} imply
\begin{align}
\textsc{OptMuon-A}:&\quad
\EEE\|\nabla f(\X_{\rm out})\|_{\fro}
\le \tilde{\OO}\!\left(T^{-1/2}+\sigma^{1/2}T^{-1/4}\right),
\qquad \text{Average smoothness},\nonumber\\
\textsc{OptMuon-I}:&\quad
\EEE\|\nabla f(\X_{\rm out})\|_{\fro}
\le \tilde{\OO}\!\left(T^{-1/2}+\sigma^{1/3}T^{-1/3}\right),
\qquad \text{Individual smoothness}.\nonumber
\end{align}

Both bounds reduce automatically to $\tilde{\OO}(T^{-1/2})$ in the zero-noise regime. Thus, the paper extends zero-noise adaptivity from scalar adaptive stepsize rules to polar-factor momentum updates.

\paragraph{Contributions.} The paper makes the following contributions.

\begin{itemize}[leftmargin=1.5em,itemsep=0.35ex]

\item \textbf{Closed-loop Momentum Orthogonalization.} We introduce OptMuon, a stochastic first-order framework in which the update direction is the polar factor $\Orth(\M_t)$ and the update magnitude is determined by a scalar self-normalized schedule.  This design cleanly separates matrix geometry from magnitude calibration.

\item \textbf{Running-Maximum Corrected AdaGrad-Norm Schedule.} The proposed coefficient is driven by accumulated stochastic-gradient energy but corrected by a running maximum. This design preserves the self-normalized suppression effect of AdaGrad-Norm while preventing a single gradient spike from causing excessive coefficient collapse.

\item \textbf{Noise-Adaptive Convergence Rates.} Under lower boundedness, unbiasedness, bounded variance, smoothness, and almost-sure bounded stochastic gradients, OptMuon-A attains the average-smoothness rate $\tilde{\mathcal O}(T^{-1/2}+\sigma^{1/2}T^{-1/4})$, while OptMuon-I attains the individual-smoothness rate $\tilde{\mathcal O}(T^{-1/2}+\sigma^{1/3}T^{-1/3})$.

\item \textbf{Zero-Noise Recovery for Orthogonalized Updates.} When $\sigma=0$, both guarantees reduce without retuning to the nearly deterministic $\tilde{\OO}(T^{-1/2})$ stationarity rate under the stated assumptions.  This shows that closed-loop scalar adaptation can be coupled with Muon-style orthogonalization without sacrificing its zero-noise behavior.

\end{itemize}

\paragraph{Organization.}
Section~\ref{sect:related} discusses related work.  Section~\ref{sect:prel} introduces notation, the orthogonalization operator, and the assumptions.  Section~\ref{sect:proposed} presents the OptMuon algorithms.  Section~\ref{sect:rate} proves the convergence guarantees, and Section~\ref{sect:conc} concludes.

\section{Related Work}
\label{sect:related}

\paragraph{Stochastic Nonconvex Optimization.}
The standard complexity theory of stochastic first-order methods for smooth nonconvex objectives studies expected stationarity measures such as $\EEE\|\nabla f(\X_{\rm out})\|_{\fro}$ or their squared-norm variants \citep{Ghadimi2013,ghadimi2016mini}.  Under unbiased stochastic gradients with bounded variance, SGD-type methods typically require stepsizes tuned to the smoothness and noise level. Adaptive methods reduce this sensitivity by using gradient history to set coordinate-wise or norm-based magnitudes \citep{duchi2011adaptive,kingma2015adam,Reddi2018OnTC,ward2020adagrad}. OptMuon follows this adaptive viewpoint but applies scalar adaptation to an orthogonalized momentum direction rather than to the raw gradient.

\paragraph{Noise-Adaptive and Lipschitz-Free Methods.} Variance-reduced estimators such as SARAH and SPIDER established recursive-gradient templates for nonconvex stochastic optimization, and STORM-type methods use momentum-style recursion to obtain rates with stochastic terms of order $\sigma^{1/3}T^{-1/3}$ under individual smoothness \citep{nguyen2017sarah,nguyen2017nonconvexsarah,fang2018spider,cutkosky2019momentum}. Subsequent Lipschitz-free variants, including STORM$^+$ and META-STORM, further reduce dependence on problem parameters while preserving noise adaptivity \citep{levy2021storm,liu2022meta}. These methods provide closed-loop scalar rules, but they do not use Muon-style polar-factor directions.

\paragraph{Orthogonalized and Normalized Updates.} Normalized-gradient and orthogonalized-gradient methods modify the update direction so that the step length is not dictated directly by the raw gradient norm. For matrix-valued parameters, Muon-style optimizers use a polar-factor or Newton--Schulz approximation of a momentum matrix, which equalizes singular directions and changes the geometry of the search step \citep{kim2026convergencens,riabinin2025gluon}. Existing analyses of orthogonalized methods are largely based on open-loop or constant magnitude choices, whereas OptMuon couples polar-factor momentum with a trajectory-dependent magnitude schedule.

\paragraph{Adaptive Scaling for Muon-style Updates.}
Several recent works have begun to study adaptive scaling mechanisms for orthogonalized updates. AdaGO combines a norm-based AdaGrad-type stepsize with an orthogonalized update direction and establishes convergence guarantees in deterministic and stochastic nonconvex settings \citep{zhang2025adagradmeetsmuon}. AdaMuon combines element-wise adaptivity with orthogonal updates and is mainly motivated by large-scale neural-network training \citep{si2025adamuon}. NAMO and NAMO-D further integrate norm-based Adam-type noise adaptation with orthogonalized momentum and provide convergence guarantees together with GPT pretraining experiments \citep{zhang2026adamimprovesmuon}. Distance-Aware Muon \citep{demidovich2026distance} studies adaptive step scaling for normalized optimization from a trust-region and trajectory-distance perspective. OptMuon is complementary to these works: its scalar schedule is a lagged, running-maximum corrected AdaGrad-Norm coefficient whose non-monotone variation is controlled logarithmically, and the analysis emphasizes explicit noise-adaptive rate decompositions and zero-noise recovery for Muon-style polar-factor momentum.


\paragraph{Relation to Recent Muon Convergence Theory.} This work is complementary to recent convergence analyses of Muon and its variance-reduced variants. In particular, \citet{chang2025convergence} study horizon-free variance-reduced Muon-type methods and show that variance reduction can recover the optimal stochastic nonconvex rate under appropriate smoothness assumptions. Our focus is different: we do not aim to replace the variance-reduction theory in that work. Instead, we ask whether Muon-style polar-factor momentum can be combined with a closed-loop scalar coefficient schedule that is selected from the observed trajectory and automatically separates the deterministic and stochastic terms in the final bound. Thus, the main novelty here is the trajectory-dependent coefficient design and the resulting zero-noise adaptive rate split for orthogonalized momentum updates.

\paragraph{Comparison with Representative Methods.}
Table~\ref{tab:intro} summarizes the comparison most relevant to this paper.  The table focuses on mini-batch first-order methods for finding an $\eps$-stationary point in expectation. The rates reported for OptMuon match the guarantees proved in Section~\ref{sect:rate}; all logarithmic factors are suppressed in $\tilde{\mathcal O}(\cdot)$ notation.

\begin{table*}[t]
\renewcommand{\arraystretch}{1.18}
\centering
\caption{Comparison of representative mini-batch first-order stochastic methods for finding an
\(\epsilon\)-approximate solution \(\X\) satisfying
\(\EEE[\|\nabla f(\X)\|_{\fro}]\le \epsilon\).
The notation \(\tilde{\mathcal O}(\cdot)\) hides polylogarithmic factors.
We distinguish adaptive effective scaling from the closed-loop zero-noise recovery
property studied in this paper. All methods assume lower-boundedness, unbiased
gradients with bounded variance, and smoothness; the column ``Additional Assum.''
lists extra conditions used in the stated guarantees.}
\label{tab:intro}
\scalebox{0.73}{
\begin{tabular}{|p{5.2cm}|p{1.65cm}|p{4.15cm}|p{1.5cm}|p{1.6cm}|p{3.8cm}|}
\hline
Algorithm
& Smoothness Assum.
& Adaptive Scaling
& CLZO Recovery
& Additional Assum.
& Convergence Rate \\
\hline\hline

AdaGO \cite{zhang2025adagradmeetsmuon}
& Average
& \YES\ scalar AdaGrad-type
& \NO
& ---
& \(\mathcal O(T^{-1/4})\) \\
\hline

AdaMuon \cite{si2025adamuon}
& Average
& \YES\ element-wise
& \NO
& ---
& \(\mathcal O(T^{-1/4})\) \\
\hline

NAMO \cite{zhang2026adamimprovesmuon}
& Average
& \YES\ Adam-type scalar
& \NO\({}^{\dagger}\)
& ---
& \(\mathcal O(T^{-1/4})\)\({}^{\ddagger}\) \\
\hline

MuonNS \cite{kim2026convergencens}
& Average
& \NO\ open-loop
& \NO
& ---
& \(\mathcal O(T^{-1/4})\) \\
\hline
Muon-MVR1 \cite{chang2025convergence}
& Average
& \NO\ open-loop / VR
& \NO
& ---
& \(\mathcal O(T^{-1/4})\) \\
\hline
\textbf{OptMuon-A} [ours]
& Average
& \YES\ closed-loop scalar
& \YES
& BG
& \(\tilde{\mathcal O}(T^{-1/2}+\sigma^{1/2}T^{-1/4})\) \\
\hline\hline

Muon-MVR2 \cite{chang2025convergence}
& Individual
& \NO\ open-loop / VR
& \NO
& ---
& \(\mathcal O(T^{-1/3})\) \\
\hline

LiMuon \cite{huang2025limuon}
& Individual
& \NO\ open-loop
& \NO
& ---
& \(\mathcal O(T^{-1/3})\) \\
\hline

\textbf{OptMuon-I} [ours]
& Individual
& \YES\ closed-loop scalar + VR
& \YES
& BG
& \(\tilde{\mathcal O}(T^{-1/2}+\sigma^{1/3}T^{-1/3})\) \\
\hline
\end{tabular}}
\vspace{4pt}

\scalebox{0.96}{
\begin{tabular}{@{}p{\linewidth}@{}}
\footnotesize
\textbf{Note:}
{\color{black}{``Adaptive Scaling'' means that the effective update magnitude uses observed gradient or momentum statistics. ``CLZO Recovery'' means closed-loop zero-noise recovery: the same online scalar rule yields an explicit bound of the form \(\tilde{\mathcal O}(T^{-1/2}+\psi(\sigma,T))\), and hence recovers \(\tilde{\mathcal O}(T^{-1/2})\) when \(\sigma=0\), without retuning to the noise level. BG denotes almost-sure bounded stochastic gradients. VR denotes variance-reduction or recursive-gradient-estimator mechanisms.}}
\({}^{\dagger}\) NAMO has norm-based Adam-type adaptive scaling, but its update
still uses a prescribed learning rate; thus it is not counted as closed-loop
zero-noise recovery in the sense above.
\({}^{\ddagger}\) The displayed NAMO rate is the stochastic mini-batch rate under
the comparison convention; its deterministic guarantee is separate and uses a
different parameter scaling.
\end{tabular}}
\end{table*}

\section{Preliminaries}
\label{sect:prel}

This section introduces the notation, the orthogonalization operator, and the standing assumptions used in the main analysis.

\paragraph{Notation.}
For matrices $\A,\B\in\mathbb R^{m\times n}$, let $\langle \A,\B\rangle:=\tr(\A^\top\B)$ be the Frobenius inner product, and let $\|\A\|_{\fro}$, $\|\A\|_2$, and $\|\A\|_*$ denote the Frobenius, spectral, and nuclear norms, respectively.  We write $\EEE[\cdot]$ for total expectation and $\EEE_t[\cdot]:=\EEE[\cdot\mid\FF_{t-1}]$ for conditional expectation with respect to the natural filtration $\{\FF_t\}_{t\ge0}$.  The stochastic gradient at iteration $t$ is denoted by
\[
    \G_t:=\nabla f(\X_t;\xi_t),
    \qquad
    g_t:=\|\G_t\|_{\fro}.
\]

\subsection{Orthogonalization Operator}

For any matrix $\M\in\mathbb R^{m\times n}$ with $m\geq n$, let $\M=\U\Sigma\W^\top$ be a thin singular value decomposition, where $\U\in \mathbb{R}^{m\times n}$, $\Sigma\in \mathbb{R}^{n\times n}$, and $\W\in\mathbb{R}^{n \times n}$. Here, $\U^\top\U=\I_n$ and $\W^\top\W=\I_n$. We define the polar-factor orthogonalization operator by
\[
    \Orth(\M):=\U\W^\top.
\]
The analysis uses only the following elementary properties:
\begin{enumerate}[label=(P\arabic*),leftmargin=2.0em,itemsep=0.2ex]
\item \textbf{Spectral boundedness:} $\|\Orth(\M)\|_2=1$.
\item \textbf{Frobenius norm:} $\|\Orth(\M)\|_{\fro}^2=n$.
\item \textbf{Alignment:} $\langle \Orth(\M),\M\rangle=\|\M\|_*\ge \|\M\|_{\fro}$.
\end{enumerate}
\noi When \(\M\) is rank-deficient, \(\Orth(\M)\) denotes any measurable choice of a thin-SVD completion satisfying properties (P1)--(P3). In implementation, the exact polar factor can be replaced by a few Newton--Schulz iterations, as in Muon-style optimizers.  The present theory treats exact orthogonalization; inexact orthogonalization would introduce an additional approximation term.

\subsection{Standing Assumptions}

\begin{assumption}[Lower-bounded objective]
\label{ass:F}
There exists $f_*>-\infty$ such that $f(\X)\ge f_*$ for all $\X\in\mathbb R^{m\times n}$.
\end{assumption}

\begin{assumption}[Unbiased stochastic gradients with bounded variance]
\label{ass:sigma}
For every $\X\in\mathbb R^{m\times n}$,
\[
    \EEE_{\xi}\big[\nabla f(\X;\xi)\big]=\nabla f(\X),
    \qquad
    \EEE_{\xi}\big[\|\nabla f(\X;\xi)-\nabla f(\X)\|_{\fro}^2\big]\le \sigma^2 .
\]
\end{assumption}

\begin{assumption}[Smoothness]
\label{ass:f}
At least one of the following conditions holds.
\begin{enumerate}[label=\textbf{(\alph*)},leftmargin=18pt,itemsep=0.2ex]
\item \textbf{Average $L$-smoothness:}
\[
    \|\nabla f(\X)-\nabla f(\Y)\|_{\fro}
    \le L\|\X-\Y\|_{\fro},
    \qquad
    \forall \X,\Y\in\mathbb R^{m\times n}.
\]
\item \textbf{Individual $L$-smoothness:}
\[
    \|\nabla f(\X;\xi)-\nabla f(\Y;\xi)\|_{\fro}
    \le L\|\X-\Y\|_{\fro},
    \qquad
    \forall \X,\Y\in\mathbb R^{m\times n},~\forall \xi.
\]
\end{enumerate}
\end{assumption}

\begin{assumption}[Bounded stochastic gradients]
\label{ass:BG}
There exists $\hG>0$ such that, almost surely,
\[
    \|\nabla f(\X;\xi)\|_{\fro}\le \hG,
    \qquad
    \forall \X\in\mathbb R^{m\times n}.
\]
Consequently, Jensen's inequality gives $\|\nabla f(\X)\|_{\fro}\le \hG$ for all $\X$.
\end{assumption}

\begin{remark}
\ding{182} Assumptions~\ref{ass:F}--\ref{ass:f} are standard in stochastic nonconvex optimization. Assumption~\ref{ass:BG} is stronger and is used to obtain pathwise control of the adaptive normalizers, the estimator tracking error, and the orthogonalized update magnitude. We do not claim a bounded-gradient-free theory. \ding{183} The term ``Lipschitz-free'' is used only in the algorithmic sense: the parameters of OptMuon do not require the smoothness constant \(L\), the variance level \(\sigma^2\), or the bound $\hG$, although these quantities enter the constants in the convergence analysis.


\end{remark}

\section{The OptMuon Algorithms}
\label{sect:proposed}

In this section, we present {OptMuon}, a closed-loop momentum orthogonalization framework. The algorithm preserves the Muon-style update structure while replacing open-loop magnitude control with a trajectory-dependent adaptive schedule. The resulting update is Lipschitz-free in the algorithmic sense: no parameter is tuned using the smoothness constant \(L\).

To accommodate different smoothness regimes, {OptMuon} has two variants. \textbf{Option A} is designed for average smoothness and uses one stochastic gradient per iteration. \textbf{Option I} is designed for individual smoothness and uses a STORM-type recursive momentum estimator. Both variants share the same orthogonalized update template; they differ in the coefficient exponent \(q\), the estimator \(\M_t\), and the adaptive factor \(\gamma_t\).

\paragraph{\texorpdfstring{Steps S1 and S2: Adaptive Parameter Tracking.}{Steps S1 and S2: Adaptive Parameter Tracking.}} At iteration $t$, \textsc{OptMuon} samples a mini-batch $\xi_t$, computes the stochastic gradient
\beq
\G_t:=\nabla f(\X_t;\xi_t),
\qquad
g_t:=\|\G_t\|_{\fro},
\eeq
\noi and updates a scalar trajectory-dependent coefficient. The coefficient used in the $t$-th update is the lagged quantity $\alpha_t=\rho_{t-1}$. Hence $\alpha_t$ is $\mathcal F_{t-1}$-measurable, which ensures that the adaptive scaling is predictable with respect to the current stochastic gradient.

We initialize $\rho_{0}=1$, set $\alpha_t:=\rho_{t-1}$ for all $t\ge1$, and update $\rho_t$ after observing $g_t$ according to
\begin{equation}\label{eq:alpha:update}
\alpha_t=\rho_{t-1},
\qquad \text{with}\quad \rho_t=\left( \frac{1+\max_{1\le i\le t}g_i^2}
{1+\sum_{i=1}^{t}g_i^2}
\right)^q.
\end{equation}
Equivalently, $\alpha_1=1$, and the coefficient used at iteration $t$ depends only on the trajectory observed before sampling $\xi_t$. The tracker $\rho_t$ is updated only after $g_t$ has been observed, so no current-sample information enters $\alpha_t$. The denominator in \eqref{eq:alpha:update} plays the usual self-normalizing role of AdaGrad-Norm-type schedules, reducing the effective update magnitude as the accumulated stochastic-gradient energy grows.  The running maximum in the numerator counteracts excessive shrinkage after isolated gradient spikes.  The exponent $q$ is chosen according to the smoothness regime: $q=1/2$ for the average-smoothness analysis and $q=2/3$ for the individual-smoothness analysis.

The following elementary properties of the coefficient sequence are repeatedly used in the tracking-error estimates.

\begin{lemma}\label{lemma:prop:para}
(Proof in Appendix \ref{app:lemma:prop:para})
Assume $\rho_{-1}=\rho_0=1$ and $\rho_{t}=\left(
\frac{1+\max_{1\le i\le t}g_i^2}{1+\sum_{i=1}^{t}g_i^2}\right)^q$, where $q\in(0,1)$. Define $\NNN_T:=\max_{1\leq t\leq T}g_t^2$ and $\GGG_T:=\sum_{t=1}^T g_t^2$. Then, for all $1\leq t\leq T$, we have

\begin{enumerate}[label=\textbf{(\alph*)},leftmargin=18pt,itemsep=0.2ex]

\item $\frac{1}{\rho_{t}}-\frac{1}{\rho_{t-1}}\le q$, $\frac{\rho_{t-1}}{\rho_t}\le 2^q$, $(1-\rho_t)\sqrt{\tfrac{\rho_{t-1}}{\rho_t}} \le 1- (1-q)\rho_t$, and $\tfrac{\rho_T}{\rho_t}\leq \big(1+ \NNN_T \big)^q$.

\item $\sum_{t=0}^{T-1} \max(0,\rho_t-\rho_{t+1}) \leq 2 \ln(e+\NNN_T)$, and $\sum_{t=0}^{T} |\tfrac{\rho_{t-1}}{\rho_{t}} -1| \leq 4 \ln(1+\GGG_{T})$.

\end{enumerate}

\end{lemma}

A conservative alternative is to replace $\rho_t$ by its monotone envelope, $\rho_t^{\rm mon}:=
\min_{1\le k\le t} \rho_{k}$. This monotone variant can be analyzed with the same order of convergence guarantee, and the proof is somewhat simpler because the coefficient variation is one-sided.  We do not impose this monotonicity in the main algorithm.  Allowing the running-maximum corrected coefficient to be non-monotone makes the adaptive rule more responsive to the realized trajectory, while Lemma \ref{lemma:prop:para} shows that its oscillation remains only logarithmic.  These logarithmic variation bounds are sufficient for the OptMuon tracking analysis and make the monotone envelope unnecessary.

\paragraph{Step S3: Dual-Track Momentum and Step-Size.}
The two variants differ in the construction of the momentum matrix and the scalar factor.

\begin{enumerate}[label=\textbf{(\alph*)}, leftmargin=18pt, itemsep=2pt, topsep=1pt, parsep=0pt, partopsep=0pt]

\item \textbf{Option A (Average Smoothness, \(q=1/2\)):}
The momentum matrix accumulates the current stochastic gradient with an adaptive decay:
\beq
\M_t=\G_t+(1-\alpha_t)\M_{t-1}.\nn
\eeq
The scalar factor is
\beq
\ts \gamma_t = \min\left\{
\alpha_t^2,\,
\alpha_t\left(1+\sum_{j=1}^t \alpha_j\|\M_j\|_{\fro}^2\right)^{-1/2}
\right\}. \nn
\eeq
This choice couples the update magnitude to both the coefficient \(\alpha_t\) and the cumulative momentum energy.

\item \textbf{Option I (Individual Smoothness, \(q=2/3\)):}
The momentum estimator uses a STORM-type recursive correction:
\beq
\ts\M_t=(1-\alpha_t)\big(\M_{t-1}-\nabla f(\X_{t-1};\xi_t)\big)+\G_t.\nn
\eeq
The scalar factor is
\beq
\ts \gamma_t = \min\left\{
\sqrt{\alpha_t},\,
\left(1+\sum_{j=1}^t \alpha_j^{-1/2}\|\M_j\|_{\fro}^2\right)^{-1/2}
\right\}. \nn
\eeq
{{This variant evaluates two stochastic gradients on the same fresh sample or mini-batch, at $\X_t$ and $\X_{t-1}$, and exploits individual smoothness to obtain a sharper stochastic term; the theorem is stated per $T$ iterations, so the gradient-evaluation complexity differs only by a constant factor.}}
\end{enumerate}

\paragraph{Step S4: Orthogonalized Update Rule.}
Unlike standard SGD or Adam, {OptMuon} explicitly orthogonalizes the momentum matrix:
\begin{equation}\label{eq:OptMuon:update}
\X_{t+1}=\X_t-\theta\gamma_t\|\M_t\|_{\fro}\operatorname{Orth}(\M_t).
\end{equation}
The operator \(\operatorname{Orth}(\M_t)\) extracts a geometry-normalized matrix direction, while the scalar factor \(\theta\gamma_t\|\M_t\|_{\fro}\) controls the magnitude. Therefore, the direction and magnitude of the update are separated: the direction is governed by the polar factor of the momentum, and the magnitude is governed by a closed-loop scalar schedule. The polar factor can be computed by SVD in theory and approximated efficiently by a few Newton-Schulz iterations in practice. The present analysis focuses on exact orthogonalization; inexact Newton-Schulz errors can be incorporated as an additional approximation term in future work.

We summarize both configurations in Algorithm \ref{alg:main}.

\begin{algorithm}[!h]
\caption{OptMuon}
\SetKwInOut{Config}{Configuration}
\KwIn{Learning rate \(\theta>0\), initial solution \(\X_1\).}
\textbf{Initialize:} Set $\rho_{-1}=\rho_0=1$, $\alpha_0=\alpha_1=1$, $\M_0=\zero$, and $\X_0=\X_1$.


\noi \textbf{Option A} (Average smoothness): Set \(q=1/2\);

\noi \textbf{Option I} (Individual smoothness): Set \(q=2/3\).

\For{\(t=1,\ldots,T\)}
{
    \textbf{S1)} Sample \(\xi_t\sim\mathcal{D}\), compute \(\G_t=\nabla f(\X_t;\xi_t)\), and set \(g_t=\|\G_t\|_{\fro}\).

    \textbf{S2)} Update $\alpha_{t}=\rho_{t-1}$, where $\rho_t=\left(\frac{1+\max_{1\le i\le t}g_i^2}{1+\sum_{i=1}^t g_i^2}
    \right)^q$ with $\rho_{-1}=\rho_0=1$.

    \textbf{S3)} Select the smoothness regime and update \(\M_t\) and \(\gamma_t\):

    \textbf{Option A}: Set $\M_t=\G_t+(1-\alpha_t)\M_{t-1}$,

    $\quad\quad\quad\quad~~$ and $\gamma_t=\min\left\{\alpha_t^2,\alpha_t\left(1+\sum_{j=1}^t\alpha_j\|\M_j\|_{\fro}^2\right)^{-1/2}\right\}$.

    \textbf{Option I}: Set $\M_t=(1-\alpha_t)(\M_{t-1}-\nabla f(\X_{t-1};\xi_t))+\G_t$,

    $\quad\quad\quad\quad~~$ and $\gamma_t=\min\left\{\sqrt{\alpha_t},\left(1+\sum_{j=1}^t \tfrac{1}{\sqrt{\alpha_j}}\|\M_j\|_{\fro}^2\right)^{-1/2}\right\}$.

    \textbf{S4)} Update $\X_{t+1}=\X_t-\theta\gamma_t\|\M_t\|_{\fro}\operatorname{Orth}(\M_t)$.
}
\label{alg:main}
\end{algorithm}

\section{Convergence Guarantees}
\label{sect:rate}

In this section, we state the convergence guarantees for {OptMuon}. The proofs are deferred to the appendix.

\paragraph{Notation.}
Let \(g_t:=\|\G_t\|_{\fro}\) and \(m_t:=\|\M_t\|_{\fro}\), and $\epsilons_t := \mathbf G_t - \nabla f(\mathbf X_t)$. For {OptMuon-A}, define $\V_t:=\alpha_t\M_t$, $\S_t:=\V_t-\nabla f(\X_t)$. For {OptMuon-I}, define $\S_t:=\M_t-\nabla f(\X_t)$. We use \(\EEE[\cdot]\) for total expectation and
\(\EEE_t[\cdot]:=\EEE[\cdot\mid\FF_{t-1}]\) for conditional expectation. The natural filtration is denoted by \(\{\FF_t\}_{t\ge0}\).

\subsection{Analysis for OptMuon-A}
\label{sect:it:L}

We first consider {OptMuon-A} under average smoothness. Recall that \(q=1/2\). Throughout this subsection, define $\V_t:=\alpha_t\M_t$, $\S_t:=\V_t-\nabla f(\X_t)$, $\bar\gamma_t:=\frac{\gamma_t}{\alpha_t}$. By the definition of \(\gamma_t\), the scalar actually appearing in the scaled \(\V_t\)-update is $\bar\gamma_t = \min\{\alpha_t, \beta_t\}$, where $\beta_t:=\big(1+\sum_{j=1}^t \tfrac{1}{\alpha_j}\|\V_j\|_{\fro}^2\big)^{-1/2}$. Thus the algorithmic update can be equivalently written as $\X_{t+1}=\X_t-\theta\bar\gamma_t\|\V_t\|_{\fro}\Orth(\V_t)$, because \(\Orth(\alpha_t\M_t)=\Orth(\M_t)\) for \(\alpha_t>0\). This notation separates the proof from the raw momentum \(\M_t\) and avoids mixing \(\gamma_t\) with \(\gamma_t/\alpha_t\).

\begin{lemma} [Preliminary Bounds] \label{lemma:OptMuonA:Prel:Bound}
(Proof in Appendix \ref{app:lemma:OptMuonA:Prel:Bound}) Define $\lambda:=1+\hG^2$, and $\varpi:=\sqrt{\lambda}$, and $\bar\gamma_t:=\gamma_t/\alpha_t$. For all $1\leq t\leq T$, it holds that $\tfrac{\alpha_T}{\alpha_t}\leq \varpi$, $\tfrac{\bar{\gamma}_T}{\bar{\gamma}_t}\leq \varpi$, $\|\V_t\|_{\fro}\leq \varpi^2$, $\|\S_t\|_{\fro}\leq 2 \varpi^2$, and $\|\epsilons_t\|_{\fro} \leq 2\hG$ almost surely.
\end{lemma}

\begin{lemma}\label{lemma:OptMuonA:sumV:byG}
(Proof in Appendix \ref{app:lemma:OptMuonA:sumV:byG}) For all \(T\ge1\), $\sum_{t=1}^T\|\V_t\|_{\fro}^2 \le \kappa\sum_{t=1}^T\|\nabla f(\X_t;\xi_t)\|_{\fro}^2$, where $\kappa:=2(1 + \hG)$.
\end{lemma}

\begin{lemma}\label{lemma:VR:bound:A}
(Proof in Appendix \ref{app:lemma:VR:bound:A}) Define $\S_t:=\V_t -\nabla f(\X_t)$, and $\epsilons_t := \mathbf G_t - \nabla f(\mathbf X_t)$. For all $t\geq 1$, we have:
\beq \label{eq:recursive:A}
\ts \|\mathbf S_t\|_{\mathrm F}^2~~\leq~~ \ts  (1-\alpha_{t} ) \|\S_{t-1}\|_{\fro}^2 +  B_t \quad a.s.,
\eeq
\noi where $B_t:=  2 \alpha_{t}^2\|\epsilons_t\|_{\fro}^2 + 2 \alpha_{t} (1-\alpha_{t}) (1+\alpha_{t}) \la \S_{t-1} ,\epsilons_{t}\ra + 4 \varpi^5 |\tfrac{\alpha_{t-1}}{\alpha_{t}}-1| + 3\delta \frac{\gamma_{t-1}^2m_{t-1}^2}{ \alpha_{t-1}} $, and $\delta:= n L^2\theta^2$.

\end{lemma}

\begin{lemma}[One-Step Descent Inequality]\label{lemma:suff:descent:A}
(Proof in Appendix \ref{app:lemma:suff:descent:A})
Define \(c_0:=\theta/2\), \(c_1:=\theta n/2\), and \(c_2:=Ln\theta^2/2\). For all \(t\ge1\), the following inequality holds almost surely:
\beq
f(\X_{t+1})-f(\X_t) + c_0\bar\gamma_t\|\V_t\|_{\fro}^2
\le c_1\bar\gamma_t\|\S_t\|_{\fro}^2 + c_2\bar\gamma_t^2\|\V_t\|_{\fro}^2,
\eeq
\noi where $\bar\gamma_t=\gamma_t/\alpha_t$.
\end{lemma}

\begin{lemma} \label{lemma:bcd:MuonA}
(Proof in Appendix \ref{app:lemma:bcd:MuonA}) Define $\bar\gamma_t:=\gamma_t/\alpha_t$. We have:
\begin{enumerate}[label=\textbf{(\alph*)},leftmargin=18pt,itemsep=2pt,topsep=1pt,parsep=0pt,partopsep=0pt]

\item $\ts\EEE [\sum_{t=1}^T B_t]\le C_b \ln(e+T)$, where $C_b:=  \big(52 \varpi^5 + 9 \delta \ln (e + \varpi \kappa) \big)\cdot \ln(e+\hG^2)$.

\item $\ts \EEE [\sum_{t=1}^T \bar\gamma_t\|\S_t\|_{\fro}^2]\le \EEE [\sum_{t=1}^T \alpha_t\|\S_t\|_{\fro}^2]\le C_c \ln(e+T)$, where $C_c:= 2C_b$.

\item $\ts \sum_{t=1}^T \bar{\gamma}_t^2\|\V_t\|_{\fro}^2\le C_d \ln(e+T)$ almost surely, where $C_d:=3\kappa^3 \ln(e+\kappa) \ln(e+\hG^2)$.


\end{enumerate}
\end{lemma}

\begin{lemma} \label{lemma:U:square:MuonA}
 (Proof in Appendix \ref{app:lemma:U:square:MuonA})    Define $\UUU_T:=\sum_{t=1}^T \bar\gamma_t\|\S_t\|_{\fro}^2$. We have $\EEE[\UUU_T^2] \leq C_u \ln^2(e+T)$, where $C_u:=\big(36\varpi^4\ln(e+\hG^2)+2C_e\big)C_c$, and $C_e:=\big( 88 \varpi^5 + 6 \delta \big) \cdot \ln(e+\hG^2)$.
\end{lemma}

\begin{lemma}\label{R_R}
 (Proof in Appendix \ref{app:R_R}) Let $F_1:=f(\X_1)-f_*$ and $\VVV_T:=\sum_{t=1}^T\|\V_t\|_{\fro}^2$. We have $\EEE[\alpha_T \VVV_T]\le C_v \ln^2(e+T)$, where $C_v:=4\varpi^4 ( C_w + C_x)$, $C_w:=\tfrac{1}{c_0} (F_1 + c_1 C_c + c_2 C_d)$, and $C_x:=\tfrac{3}{c_0^2} \left( F_1^2 + c_1^2 C_u + c_2^2 C_d^2\right)$.
\end{lemma}

We now state the convergence guarantee for {OptMuon-A}.

\begin{theorem}\label{the:it:MuonA}
(Proof in Appendix \ref{app:the:it:MuonA})
Under Assumptions \ref{ass:F}, \ref{ass:sigma}, and \ref{ass:BG} and the average smoothness condition in Assumption \ref{ass:f}, for every $T\ge1$, define $C_z:=\sqrt{2C_v}+\sqrt{2\varpi C_c}$, $C_{A,1}:=C_z\sqrt{2+8C_z^2}$, $C_{A,2}:=2^{3/4}C_z$. Then {OptMuon-A} satisfies the explicit bound
\beq
\EEE \!\left[\frac{1}{T}\sum_{t=1}^T \|\nabla f(\X_t)\|_{\fro} \right]
\le
C_{A,1}\frac{\ln^2(e+T)}{T^{1/2}}
+
C_{A,2}\frac{\sigma^{1/2}\ln(e+T)}{T^{1/4}} .
\eeq
Consequently, if \(R\sim {\rm Unif}\{1,\ldots,T\}\) is independent of the algorithmic randomness and \(\X_{\rm out}:=\X_R\), then the same bound holds for
\(\EEE[\|\nabla f(\X_{\rm out})\|_{\fro}]\).
\end{theorem}

\begin{remark}
\ding{182} Theorem \ref{the:it:MuonA} gives a bound of order $\tilde{\mathcal O}(T^{-1/2}+\sigma^{1/2}T^{-1/4})$ for the averaged gradient norm. The first term is the deterministic optimization contribution and the second term is the variance-dependent contribution. When $\sigma$ is a fixed positive constant, the stochastic component matches the usual $\tilde{\mathcal O}(T^{-1/4})$ gradient-norm behavior for open-loop Muon-style stochastic analyses, whereas in the low-noise regime the bound improves continuously and becomes $\tilde{\mathcal O}(T^{-1/2})$ at $\sigma=0$. \ding{183} Relative to open-loop Muon-style stochastic analyses that yield $\tilde{\mathcal{O}}(T^{-1/4})$-type gradient-norm rates, the bound above makes the low-noise improvement explicit through the term $\sigma^{1/2}T^{-1/4}$ \citep{kim2026convergencens}. \ding{184} The additional polylogarithmic factors in our bound arise naturally from the fully adaptive and horizon-free nature of the algorithm, which requires neither prior knowledge of the total iteration budget $T$ nor the noise variance $\sigma^2$.
\end{remark}

\subsection{Analysis for OptMuon-I}
\label{sect:it:LL}

We now analyze {OptMuon-I} under the individual smoothness condition in Assumption \ref{ass:f}.  This variant uses the STORM-type momentum estimator and the exponent \(q=2/3\).

For notation convenience, we define $m_t:=\|\M_t\|_{\fro}$, $\S_t:=\M_t-\nabla f(\X_t)$, $s_t:=\|\S_t\|_{\fro}$, $\eta_t:=\theta\gamma_t m_t$, $\beta_t:= \big(1+\sum_{j=1}^t \alpha_j^{-1/2}m_j^2\big)^{-1/2}$, $\GGG_T:=\sum_{t=1}^T g_t^2$, $\MMM_T:=\sum_{t=1}^T m_t^2$, $\SSS_T:=\sum_{t=1}^T s_t^2$. Thus the Option-I scalar factor is $\gamma_t=\min\{\sqrt{\alpha_t},\beta_t\}$, and the update is $\X_{t+1}=\X_t-\theta\gamma_t m_t\Orth(\M_t)$. We use the harmless initialization convention \(\X_0=\X_1\), so that the first STORM correction is well-defined and, since \(\alpha_1=1\), \(\M_1=\G_1\).

The following lemma bounds the momentum energy controlled by stochastic gradient energy.

 \begin{lemma} [Preliminary Bounds]\label{lemma:prop:s:2}
(Proof in Appendix \ref{app:lemma:prop:s:2}) Define $\lambda:=1+\hG^2$, and $\varpi:=\lambda^{2/3}$. For all $1\leq t\leq T$, we have $\|\epsilons_t\|_{\fro}\le 2\hG$, $\|\epsilons'_t\|_{\fro}\le 2\hG$, $\tfrac{\alpha_T}{\alpha_t}\leq \varpi$ and $\tfrac{\gamma_T}{\gamma_t}\leq \sqrt{\varpi}$. Moreover, for $t\ge1$, $\tfrac{\alpha_{t-1}}{\alpha_t}\le2$ and $(1-\alpha_t)\sqrt{\alpha_{t-1}/\alpha_t}\le1-\alpha_t/3$.

\end{lemma}

\begin{lemma}\label{lemma:OptMuonI:sumV:byG} (Proof in Appendix \ref{app:lemma:OptMuonI:sumV:byG}) Define $\kappa:= 2\max\big(12 , 96  \delta \varpi^{3/2}   \ln(e +  36  \delta \varpi^{3/2} \lambda ) \big)$ and $\delta:=nL^2\theta^2$. For all \(T\ge1\), $\sum_{t=1}^T\|\M_t\|_{\fro}^2 \le \kappa \ln(e+T)\big(1+\sum_{t=1}^T\|\nabla f(\X_t;\xi_t)\|_{\fro}^2\big)$ a.s..
\end{lemma}

\begin{lemma}[Tracking-error Recursion]\label{lemma:VR:bound:I}
(Proof in Appendix \ref{app:lemma:VR:bound:I}) Define $\S_t:=\M_t-\nabla f(\X_t)$, $\epsilons_t:= \nabla f(\X_{t};\xi_t) - \nabla f(\X_t)$, $\epsilons'_t:= \nabla f(\X_{t-1};\xi_t) - \nabla f(\X_{t-1})$, $m_t:=\|\M_t\|_{\fro}$. Then
\beq \label{eq:recursive:I:pathwise}
\|\S_t\|_{\fro}^2 \le (1-\alpha_t)\|\S_{t-1}\|_{\fro}^2 + B_t
\qquad a.s.,
\eeq
where $B_t:= 2 \alpha_t^2 \|\epsilons_t\|_{\fro}^2 + 8 \delta \gamma_{t-1}^2m_{t-1}^2 + 2 (1-\alpha_t) \la \S_{t-1},\epsilons_t   - (1-\alpha_t) \epsilons'_t\ra$, and $\delta:=n L^2\theta^2$.
\end{lemma}

\begin{lemma}[One-step Descent Inequality]\label{lemma:suff:descent:I}
(Proof in Appendix \ref{app:lemma:suff:descent:I})
Define $c_0:=\frac{\theta}{2}$, $c_1:=\frac{n\theta}{2}$, $c_2:=\frac{Ln\theta^2}{2}$. For all \(t\ge1\), the following inequality holds almost surely:
\beq
f(\X_{t+1})-f(\X_t)+c_0\gamma_t \|\M_t\|_{\fro}^2 \le c_1\gamma_t  \|\S_t\|_{\fro}^2+c_2\gamma_t^2\|\M_t\|_{\fro}^2.
\eeq
\end{lemma}

\begin{lemma} \label{lemma:bcd:MuonI}
(Proof in Appendix \ref{app:lemma:bcd:MuonI}) Define $m_t:=\|\M_t\|_{\fro}$. We have:
\begin{enumerate}[label=\textbf{(\alph*)},leftmargin=18pt,itemsep=2pt,topsep=1pt,parsep=0pt,partopsep=0pt]

\item $\ts\EEE [\sum_{t=1}^T\tfrac{B_t}{\sqrt{\alpha_t}} ]\le C_b \ln(e+T)$, where $C_b:=\big( 64(1+\hG^2)^2 + 64 n L^2\theta^2 \big) \cdot \ln(e+\kappa + \kappa\hG^2)$.

\item $\EEE [\sum_{t=1}^T \gamma_t\|\S_t\|_{\fro}^2]
\le \EEE [\sum_{t=1}^T \sqrt{\alpha_t}\|\S_t\|_{\fro}^2]
\le C_c \ln(e+T)$, where $C_c:=12 (\hG^2 + C_b)$.

\item $\ts \sum_{t=1}^T {\gamma}_t^2\|\M_t\|_{\fro}^2\le C_d \ln(e+T)$ almost surely, where $C_d:=2\ln (e+ \kappa (1+\hG^2) )$.

\end{enumerate}
\end{lemma}

 \begin{lemma} \label{lemma:U:square:MuonI}
 (Proof in Appendix \ref{app:lemma:U:square:MuonI}) Define $\UUU_T:=\sum_{t=1}^T \gamma_t\|\S_t\|_{\fro}^2$. It holds that $\EEE[\UUU_T^2]\le C_u\ln^2(e+T)$, where $C_u:=288\big(\hG^4+C_I^2+C_J^2+C_K\big)$, $C_I := 90 (1+\hG^2)^2\ln(e+\hG^2)$, $C_J:= 64nL^2\theta^2\ln(e+\kappa(1+\hG^2))$, $C_R:=8\sqrt{\varpi}\,\ln(e+\kappa(1+\hG^2))$, $C_w:=64 \hG^2\varpi C_c + 64 \delta C_R\big(\hG^2+{C_I}+{C_J}\big)$, $C_K:=4 (C_w^2+C_w)$.
\end{lemma}

\begin{lemma}\label{lemma:weighted:bridge:MuonI}
 (Proof in Appendix \ref{app:lemma:weighted:bridge:MuonI})
Let $F_1:=f(\X_1)-f_*$ and $\MMM_T:=\sum_{t=1}^T\|\M_t\|_{\fro}^2$. Then $\EEE[\sqrt{\alpha_T}\,\MMM_T]\le C_v\ln^2(e+T)$, where $C_v:=4 \varpi^2 (C_w+C_x)$, $C_w:=\frac{1}{c_0}\big(F_1+c_1C_c+c_2C_d\big)$, and $C_x:=\tfrac{3}{c_0^2} \left( F_1^2+c_1^2C_u+c_2^2C_d^2 \right)$.
\end{lemma}

We now provide the convergence guarantee for {OptMuon-I}.

\begin{theorem}\label{the:it:MuonI}
(Proof in Appendix \ref{app:the:it:MuonI}) Under Assumptions \ref{ass:F}, \ref{ass:sigma}, and \ref{ass:BG} and the individual smoothness condition in Assumption \ref{ass:f}, for every $T\ge1$, define $C_{z}:=\sqrt{2C_v}+ \sqrt{2 C_c} \cdot \varpi^{1/4}$, $C_{A,1}:=4 C_z^{2}$ and $C_{A,2}:=2 C_z$. Then {OptMuon-I} satisfies the explicit bound
\beq
\ts \EEE \left[\frac{1}{T}\sum_{t=1}^T \|\nabla f(\X_t)\|_{\fro} \right]
\le C_{A,1} \frac{\ln^{3/2}(e+T)}{T^{1/2}} +  C_{A,2}\frac{\sigma^{1/3}\ln(e+T)}{T^{1/3}}.
\eeq
Consequently, if \(R\sim {\rm Unif}\{1,\ldots,T\}\) is independent of the
algorithmic randomness and \(\X_{\rm out}:=\X_R\), then the same bound holds for
\(\EEE[\|\nabla f(\X_{\rm out})\|_{\fro}]\).
\end{theorem}

\begin{remark}
Theorem \ref{the:it:MuonI} shows that {OptMuon-I} improves the stochastic term from the average-smoothness rate in Theorem \ref{the:it:MuonA} to
$\tilde{\mathcal{O}}(\frac{1}{T^{1/2}}+\frac{\sigma^{1/3}}{T^{1/3}})$ under individual smoothness, while preserving the zero-noise rate \(\tilde{\mathcal O}(T^{-1/2})\).  Thus the variance-reduced momentum estimator sharpens the noise dependence without changing the Muon-style orthogonalized update.
\end{remark}

\section{Conclusion}
\label{sect:conc}


We have presented {OptMuon}, an adaptive orthogonalized optimization framework that combines Muon-style polar-factor momentum with a lagged, running-maximum corrected AdaGrad-Norm-type scalar schedule. Under the stated lower-boundedness, unbiasedness, smoothness, bounded-variance, and bounded stochastic-gradient assumptions, the proposed variants attain noise-adaptive convergence rates and recover the nearly deterministic \(\tilde{\mathcal O}(T^{-1/2})\) rate in the zero-noise regime without retuning. The analysis clarifies how closed-loop scalar adaptation can be coupled with orthogonalized momentum while preserving the geometric separation between update direction and magnitude.

\paragraph{Limitations and Future Work.} The present theory assumes exact polar-factor orthogonalization and an almost-sure bounded stochastic-gradient condition. It does not provide a bounded-gradient-free analysis, and it treats the polar factor exactly rather than quantifying finite Newton--Schulz approximation errors or the additional constants from approximate orthogonalization. Future work should also examine whether the same closed-loop scalar idea extends to composite, constrained, and blockwise training objectives.
\normalem

\bibliographystyle{plain} 
\bibliography{mybib} 

\clearpage
\appendix

{\huge Appendix}

\noi The appendix section is organized as follows.

\noi Section \ref{app:sect:note} covers notations and relevant lemmas.

\noi Section \ref{app:sect:proposed} contains proofs from Section \ref{sect:proposed}.

\noi Section \ref{app:sect:it:L} contains proofs from Section \ref{sect:it:L}.

\noi Section \ref{app:sect:it:LL} contains proofs from Section \ref{sect:it:LL}.

\section{Notations and Supporting Lemmas}\label{app:sect:note}

\subsection{Notations}

Throughout this paper, we adopt the following notational conventions. Bold lowercase letters denote vectors, and uppercase letters denote real-valued matrices.

\begin{itemize}[leftmargin=12pt,itemsep=0.2ex]

\item $[n]$: The set $\{1,2,...,n\}$.

\item $\|\mathbf{x}\|$: Euclidean norm, defined as $\|\mathbf{x}\|=\|\mathbf{x}\|_2 = \sqrt{\la \mathbf{x},\mathbf{x}\ra}$.

\item $\la \mathbf{a},\mathbf{b}\ra$ : Standard inner product, given by $\la \mathbf{a},\mathbf{b}\ra =\sum_{i}{\mathbf{a}_{i}\mathbf{b}_{i}}$.

\item $\a\leq \alpha$: For $\a\in\Rn^n$ and scalar $\alpha\in\Rn$, this means $\a_i\leq \alpha$ for all $i\in[n]$.

\item $\EEE[X]$: Expectation of the random variable $X$.

\item $\text{Var}(X)$: Variance of the random variable $X$.

\item $\{A_i\}_{i=1}^{\infty}$, $\{B_i\}_{i=1}^{\infty}$: sequences indexed by non-negative integers.

\item $\A \trans$ : the transpose of the matrix $\A$.

\item $\|\A\|_{\fro}$: Frobenius norm of the matrix $\A$.

\item $\|\A\|$: Spectral norm of the matrix $\A$.

\end{itemize}

\subsection{Supporting Lemmas}

This subsection shows a collection of useful facts and lemmas, each of which is independent of context and methodology.

\begin{lemma} \label{lemma:adaptive:important}
Let $b_1,\dots,b_T\ge 0$, and let $q\in(0,1)$. Then

\noindent (a) $\sum_{t=1}^T \frac{b_t}{1+\sum_{i=1}^t b_i} \le
\ln\Bigl(1+\sum_{t=1}^T b_t\Bigr)$.

\noindent (b) $\sum_{t=1}^T \frac{b_t}{\bigl(1+\sum_{i=1}^t b_i\bigr)^q} \le
\frac{1}{1-q}\Bigl(1+\sum_{t=1}^T b_t\Bigr)^{1-q}$.

\begin{proof}
These bounds are standard in the analysis of adaptive optimization algorithms. Part (a) generalizes Lemma 3.2 of \cite{ward2020adagrad}, and Part (b) generalizes Lemma 3 of \cite{levy2021storm}. Both statements follow directly from standard integral bounding techniques and a routine induction argument; the proofs are therefore omitted.
\end{proof}
\end{lemma}

\begin{lemma} \label{lemma:diff}
For any sequences $\{\beta_t\}_{t=1}^{T+1}$ and $\{m_t\}_{t=0}^T$, it holds that:
\beq \label{eq:core:id:0}
\ts \sum_{t=1}^T \beta_t \left( m_{t-1}^2 - m_{t}^2\right) - \sum_{t=1}^T \left( \beta_{t+1} - \beta_t \right) m_{t}^2  =  \beta_1 m_0^2 - \beta_{T+1} m_T^2. \nn
\eeq

\begin{proof}
The identity follows from $\ts \sum_{t=1}^T \beta_t \left( m_{t-1}^2 - m_t^2 \right)
- \sum_{t=1}^T \left(\beta_{t+1}-\beta_t\right)m_t^2=\ts \sum_{t=1}^T \beta_t m_{t-1}^2- \sum_{t=1}^T \beta_{t+1}m_t^2=\ts \beta_1m_0^2-\beta_{T+1}m_T^2$.
\end{proof}
\end{lemma}

\begin{lemma} \label{lemma:sum:by:parts}
Let $\{A_t\}_{t=1}^{T}$ and $\{B_t\}_{t=1}^{T+1}$ be nonnegative sequences. If $\{A_t\}_{t=1}^T$ is nondecreasing, then $\sum_{t=1}^T A_t(B_t-B_{t+1})
\le
A_T \max_{1\le t\le T} B_t$.

\begin{proof}
By summation by parts, $\sum_{t=1}^T A_t(B_t-B_{t+1})=A_1B_1-A_TB_{T+1}+\sum_{t=2}^T (A_t-A_{t-1})B_t$. Since $A_TB_{T+1}\ge 0$ and $A_t-A_{t-1}\ge 0$ for $t\ge2$, it follows that $\sum_{t=1}^T A_t(B_t-B_{t+1})
\le
A_1B_1+\sum_{t=2}^T (A_t-A_{t-1}) \max_{1\le j\le T} B_j$. Hence, $\sum_{t=1}^T A_t(B_t-B_{t+1})
\le
\Bigl(A_1+\sum_{t=2}^T (A_t-A_{t-1})\Bigr)\max_{1\le j\le T} B_j = A_T \max_{1\le j\le T} B_j$.
\end{proof}
\end{lemma}

\begin{lemma}\label{lemma:yc} Assume $y\le \ln(1+ay)$, where $a\geq 0$ and $y\geq 0$. Then $y\le 2\ln(e+a)$.

\begin{proof} The case \(a=0\) is immediate, since then \(y\le0\) and hence \(y=0\). We now assume \(a>0\).

\noi Define $C:=2\ln(e+a)$, and $\phi(s):=s-\ln(1+as)$.

\paragraph{Verifying \(\ln(1+aC)\le C\).} This is equivalent to $1+2a\ln(e+a)\le (e+a)^2$. Let $H(a):=(e+a)^2-1-2a\ln(e+a)$. Then $H'(a)=2\left(e+a-\ln(e+a)-1+\frac{e}{e+a}\right)>0$, because \(e+a\ge e\). Since \(H(0)=e^2-1>0\), we have \(H(a)>0\), and hence $\ln(1+aC)\le C$.

\paragraph{Verifying $\phi(C)\geq 0$ and $\phi(s)$ is Increasing.} Since $C\geq 2$, for all \(s\ge C\), $\phi'(s)=1-\frac{a}{1+as}=\frac{1+a(s-1)}{1+as}>0$. Thus $\phi(\cdot)$ is strictly increasing on \([C,\infty)\). Moreover, $\phi(C)=C-\ln(1+aC)\ge0$.

\paragraph{Finishing the Proof by Contradiction.} Assume, for contradiction, that the conclusion does not hold, i.e., \(y>C\). Then $y-\ln(1+ay)=\phi(y)>\phi(C)\ge0$, which contradicts the assumption \(y\le\ln(1+ay)\). Hence \(y\le C\), and the claim follows.
\end{proof}

\end{lemma}

\begin{lemma}[Weighted Self-Normalized Noise-to-Gradient Comparison]
\label{lemma:weighted:self:normalized}
Let \(\{\FF_t\}_{t\ge0}\) be the natural filtration, and suppose that
\(\X_t\) is \(\FF_{t-1}\)-measurable. Assume
\(\EEE_t[\nabla f(\X_t;\xi_t)]=\nabla f(\X_t)\), and
\(\|\nabla f(\X_t;\xi_t)\|_{\fro}\le \hG\) almost surely. Let
\(W_0,\ldots,W_T\) be nonnegative random variables such that \(W_{t-1}\) is
\(\FF_{t-1}\)-measurable and \(W_{t-1}\le W_T\) almost surely for every
\(t=1,\ldots,T\). Then
\beq
\ts
\EEE\!\left[
\sum_{t=1}^T
W_{t-1}
\frac{\|\nabla f(\X_t;\xi_t)-\nabla f(\X_t)\|_{\fro}^2}
{1+\sum_{i=1}^t\|\nabla f(\X_i;\xi_i)\|_{\fro}^2}
\right]
\le
8(1+\hG^2)\ln(e+\hG^2)\ln(e+T)\EEE[W_T].
\nn
\eeq

\begin{proof}
Let $\G_t:=\nabla f(\X_t;\xi_t)$, $\epsilons_t:=\G_t-\nabla f(\X_t)$, $g_t^2:=\|\G_t\|_{\fro}^2$, and $\GGG^+_t:=1+\sum_{i=1}^t g_i^2$.

\noi Define $R_t:=\EEE_t [\frac{g_t^2}{\GGG^+_t}]$.

\paragraph{Preliminary Bounds.} We write \(\EEE_t[\cdot]=\EEE[\cdot\mid \FF_{t-1}]\). Since \(\EEE_t[\G_t]=\nabla f(\X_t)\) and \(\|\G_t\|_{\fro}\le \hG\) almost surely, Jensen's inequality gives
\beq
\|\nabla f(\X_t)\|_{\fro}=\|\EEE_t[\G_t]\|_{\fro}\le\EEE_t\|\G_t\|_{\fro}\le\hG.\nn
\eeq
This further leads to:
\beq
\|\epsilons_t\|_{\fro}\le\|\G_t\|_{\fro}+\|\nabla f(\X_t)\|_{\fro}\le 2\hG\qquad\text{a.s.} \nn
\eeq
Since $\EEE_t[\epsilons_t]=0$, we derive:
\beq \label{eq:exp:gt2}
\ts \EEE_t[g_t^2] = \|\nabla f(\X_t)\|_{\fro}^2 + \EEE_t\|\epsilons_t\|_{\fro}^2\le \|\nabla f(\X_t)\|_{\fro}^2+4\hG^2.
\eeq
\paragraph{Lower-Bounding the Term \(\EEE_t[\tfrac{g_t^2}{\GGG^+_t}]\).}
Fix \(t\ge1\) and condition on \(\FF_{t-1}\). Then \(\GGG^+_{t-1}\ge1\),
\(\GGG^+_t=\GGG^+_{t-1}+g_t^2\), and \(\nabla f(\X_t)\) is fixed under
\(\EEE_t[\cdot]\). We derive:
\beq
\ts \|\nabla f(\X_t)\|_{\fro}^2 &=& \ts \|\EEE_t [\tfrac{\G_t}{\sqrt{\GGG^+_t}}\sqrt{\GGG^+_t} ] \|_{\fro}^2 ~~\xxxx{1}{\leq}~~ \ts \EEE_t [\frac{\|\G_t\|_{\fro}^2}{\GGG^+_t}]\EEE_t[\GGG^+_t] \nn\\
&=& \ts R_t\big(\GGG^+_{t-1}+\EEE_t[g_t^2]\big) ~~ \xxxx{2}{\leq} ~~ \ts R_t\big(\GGG^+_{t-1}+ \|\nabla f(\X_t)\|_{\fro}^2+4\hG^2 \big),\nn
\eeq
\noi where step \step{1} uses conditional Cauchy--Schwarz applied to $\EEE_t[(\G_t/\sqrt{\GGG_t^+})\sqrt{\GGG_t^+}]$; step \step{2} uses Inequality (\ref{eq:exp:gt2}). This further leads to:
\beq \label{eq:Rt:lower:self:normalized}
\ts R_t:= \EEE_t \left[\frac{g_t^2}{\GGG^+_t}\right] \ge \frac{\|\nabla f(\X_t)\|_{\fro}^2}{\GGG^+_{t-1}+\|\nabla f(\X_t)\|_{\fro}^2+4\hG^2}.
\eeq

\paragraph{Bounding the Term \(\|\nabla f(\X_t)\|_{\fro}^2\EEE_t[\tfrac1{\GGG^+_t}]\).} We derive:
\beq \label{eq:bQt:self:normalized}
&& \ts \|\nabla f(\X_t)\|_{\fro}^2 \EEE_t \left[\frac1{\GGG^+_t}\right] ~~\xxxx{1}{=}~~ \ts \frac{\|\nabla f(\X_t)\|_{\fro}^2 }{\GGG^+_{t-1}} (1-R_t)
~~\xxxx{2}{\le}~~ \frac{\|\nabla f(\X_t)\|_{\fro}^2}{\GGG^+_{t-1}} \cdot
\big(1-\frac{\|\nabla f(\X_t)\|_{\fro}^2}
{\GGG^+_{t-1}+\|\nabla f(\X_t)\|_{\fro}^2+4\hG^2} \big) \nn\\
&=& \ts \tfrac{ \|\nabla f(\X_t)\|_{\fro}^2(\GGG^+_{t-1}+4\hG^2)}{\GGG^+_{t-1}\big(
\GGG^+_{t-1}+\|\nabla f(\X_t)\|_{\fro}^2+4\hG^2\big)} ~~\le~~ (1+4\hG^2)R_t,
\eeq
\noi where step \step{1} uses $\EEE_t \left[\frac1{\GGG^+_t}\right]=\frac1{\GGG^+_{t-1}}\EEE_t \left[1-\frac{g_t^2}{\GGG^+_t}\right]=\frac{1-R_t}{\GGG^+_{t-1}}$, which can be implied by $\GGG^+_t=\GGG^+_{t-1}+g_t^2$; step \step{2} uses Inequality \eqref{eq:Rt:lower:self:normalized}; step \step{3} uses \(\GGG^+_{t-1}\ge1\) and Inequality \eqref{eq:Rt:lower:self:normalized} again.

\paragraph{Final Bound.} Using \(\|\epsilons_t\|_{\fro}^2 \le 2\|\G_t\|_{\fro}^2+2\|\nabla f(\X_t)\|_{\fro}^2\), we obtain
\beq \label{eq:noise:gradient:conditional}
\ts \EEE_t [ \tfrac{\|\epsilons_t\|_{\fro}^2}{\GGG^+_t}] &\le&\ts 2\EEE_t [\frac{g_t^2}{\GGG^+_t} ] + 2\|\nabla f(\X_t)\|_{\fro}^2 \EEE_t [\tfrac{1}{\GGG^+_t}] \nn\\
&\le& \ts \big(4+8\hG^2\big) \EEE_t [\frac{g_t^2}{\GGG^+_t} ] ~~\le~~ 8(1+\hG^2)
\EEE_t [\tfrac{\|\G_t\|_{\fro}^2}{\GGG^+_t}].\nn
\eeq
Since \(W_{t-1}\) is nonnegative and \(\FF_{t-1}\)-measurable, multiplying
\eqref{eq:noise:gradient:conditional} by \(W_{t-1}\), taking total expectation,
and summing over \(t=1,\ldots,T\), gives
\beq \label{eq:weighted:self:normalized:intermediate}
\ts && \ts \EEE \left[
\sum_{t=1}^T W_{t-1}\frac{\|\epsilons_t\|_{\fro}^2}{\GGG^+_t}
\right] ~~\le ~~ 8(1+\hG^2) \EEE \left[ \sum_{t=1}^TW_{t-1}\frac{\|\G_t\|_{\fro}^2}{\GGG^+_t}\right] \nn\\
&\xxxx{1}{\leq} & \ts 8(1+\hG^2) \EEE \left[ W_{T} \sum_{t=1}^T\frac{g_t^2}{ 1 + \sum_{i=1}^t g_i^2}\right] ~~ \xxxx{2}{\leq}~~ 8(1+\hG^2) \EEE \left[ W_{T} \ln(1+ \sum_{t=1}^T g_t^2 ) \right] \nn\\
&\xxxx{3}{\leq} &  8(1+\hG^2) \EEE \left[ W_{T} \ln(1+ T \hG^2 ) \right] ~~\xxxx{4}{\leq} ~~ 8(1+\hG^2) \ln(e+ T) \ln(e+ \hG^2) \EEE \left[ W_{T}  \right],\nn
\eeq
\noi where step \step{1} uses $W_{t-1}\le W_T$ for all $t\leq T$, and $\GGG^+_t = 1 + \sum_{i=1}^t g_i^2$; step \step{2} uses Lemma \ref{lemma:adaptive:important}(a); step \step{3} uses $g_t^2 \leq \hG^2$; step \step{4} uses $\ln(1+ab)\le \ln(e+a)\ln(e+b)$.

\end{proof}
\end{lemma}


\section{Proofs for Section \ref{sect:proposed}}
\label{app:sect:proposed}

\subsection{Proof of Lemma \ref{lemma:prop:para}}
\label{app:lemma:prop:para}

\begin{proof} Define $\rho_t := \left(
\frac{1 + \max_{1\le i\le t} g_i^2}
     {1 + \sum_{i=1}^t g_i^2}
\right)^{q}$ and $q\in(0,1)$.

\noi Define $u_t := \frac{1 + \GGG_t}{1 + \NNN_t}\geq 1$, where $\GGG_t := \sum_{i=1}^t g_i^2$ and $\NNN_t := \max_{1\le i\le t} g_i^2$. We set $\GGG_0=\NNN_0=0$.

\paragraph{Establishing the Inequality $u_{t+1} - u_t\leq 1$.} We distinguish two cases.

\emph{Case 1: $g_{t+1}^2 \le \NNN_t$.} We have:
\beq
u_{t+1} - u_t&=& \frac{1 + \GGG_{t+1}}{1 + \NNN_{t+1}} - \frac{1 + \GGG_t}{1 + \NNN_t}
\leq \frac{1 + \GGG_t + g_{t+1}^2}{1 + \NNN_t} - \frac{1 + \GGG_t}{1 + \NNN_t} \nn\\
&=& \frac{g_{t+1}^2}{1 + \NNN_t} \leq \frac{\NNN_t}{1 + \NNN_t} \leq 1.\nn
\eeq

\emph{Case 2: $g_{t+1}^2 > \NNN_t$.} We derive:
\beq
u_{t+1} - u_t &=& \frac{1 + \GGG_t + g_{t+1}^2}{1 + g_{t+1}^2} - \frac{1 + \GGG_t}{1 + \NNN_t}\nn\\
&=& \frac{(1 + \GGG_t + g_{t+1}^2)(1 + \NNN_t) - (1 + \GGG_t)(1 + g_{t+1}^2)}{(1 + g_{t+1}^2)(1 + \NNN_t)} \nn\\
&=& \frac{ (1 + \GGG_t)(\NNN_t - g_{t+1}^2) + g_{t+1}^2(1 + \NNN_t)}{(1 + g_{t+1}^2)(1 + \NNN_t)} \nn\\
&\leq & \frac{g_{t+1}^2(1 + \NNN_t)}{(1 + g_{t+1}^2)(1 + \NNN_t)} ~~\leq ~~ 1. \nn
\eeq
\noi In both cases we have $u_{t+1} - u_t \le 1$ for all $t\geq 1$.

\paragraph{Part (a-i).} Consider the concave function $f(x) = x^q$ with $x \in [1,\infty)$ and $q\in(0,1]$. Since $f'(x) = q x^{q-1}$, we derive:
\beq
u_{t+1}^{q} - u_t^{q} ~~=~~ f(u_{t+1}) - f(u_t)~~
\leq ~~ q u_{t}^{q-1}\,(u_{t+1} - u_t) ~~\xxxx{1}{\leq}~~ q,\nn
\eeq
\noi where step \step{1} uses $u_{t}^{q-1} \le 1$, which is implied by $u_t \ge 1$ and $q-1 \le 0$. We obtain
\beq
\ts \frac{1}{\rho_{t+1}} - \frac{1}{\rho_t}
= u_{t+1}^{\,q} - u_t^{\,q} \leq q. \label{eq:rho:t1:t}
\eeq

\paragraph{Part (a-ii).} Since $u_{t}\le u_{t-1}+1$ and $u_{t-1}\ge1$, we have $u_t/u_{t-1}\le2$ and therefore $\rho_{t-1}/\rho_t=(u_t/u_{t-1})^q\le2^q$.

\paragraph{Part (a-iii).} We have the following equivalence from Inequality (\ref{eq:rho:t1:t}):
\beq
\ts \rho_{t-1} \le \rho_t(1+q\rho_{t-1})  \iff \rho_{t-1}(1-q\rho_t)
\le \rho_t
\iff  \frac{\rho_{t-1}}{\rho_t} \le \frac{1}{1-q\rho_t}. \label{eq:bound:rho:t1:t}
\eeq
\noi For every \(\rho_{t}\in(0,1]\), we derive:
\beq
\ts (1-\rho_{t})\sqrt{{\rho_{t-1}} / {\rho_{t}}} ~~\xxxx{1}{\leq} ~~\frac{1-\rho_{t}}{\sqrt{1- q\rho_{t}}} ~~\xxxx{2}{\leq} ~~ \frac{1-\rho_{t}}{{1- q\rho_{t}}}  ~~ \xxxx{3}{=} ~~ 1- (1-q)\rho_{t}  - \tfrac{q (1-q) \rho_t^2}{1 - q\rho_t} ~~ \xxxx{ }{\leq}~~ 1- (1-q)\rho_{t},\nn
\eeq
where step \step{1} uses $\frac{\rho_{t-1}}{\rho_t} \le \frac{1}{1-q\rho_t}$, as shown in Inequality (\ref{eq:bound:rho:t1:t}); step \step{2} uses $\sqrt{x}\geq {x}$ for all $x = 1-q \rho_t\in(0,1]$; step \step{3} uses the following identity with $x=\rho_t$:
\beq
1-x = (1-qx) \cdot \left[1-(1-q)x\right] - q (1-q) x^2.\nn
\eeq

\paragraph{Part (a-iv).} For all $1\leq t\leq T$, we derive:
\beq
\ts \tfrac{\rho_T}{\rho_t} = \left(
\frac{1+\max_{1\le i\le T}g_i^2}{1+\sum_{i=1}^{T}g_i^2}\right)^q \cdot \left(
\frac{ 1+\sum_{i=1}^{t}g_i^2}{1+\max_{1\le i\le t}g_i^2}\right)^q \leq \left(
\frac{1+\NNN_T}{1+\NNN_t}\right)^q \leq \left(
1+\NNN_T\right)^q.\nn
\eeq

\paragraph{Part (b-i).} Let $z_t:=\ln \rho_t$ and $a_+:=\max\{a,0\}$. Note that $(\rho_t-\rho_{t+1})_+=(\rho_{t+1}-\rho_t)_+ + \rho_t-\rho_{t+1}$. Summing this inequality over \(t=0,\ldots,T-1\) gives
\beq \label{eq:bound:aim}
\ts && \ts \sum_{t=0}^{T-1}(\rho_t-\rho_{t+1})_+ = \ts \sum_{t=0}^{T-1}(\rho_{t+1}-\rho_t)_+  +  \rho_0-\rho_T~~\xxxx{1}{\leq} 1 + \sum_{t=0}^{T-1}(\rho_{t+1}-\rho_t)_+ \nn\\
&\xxxx{2}{\leq} & \ts 1 + \sum_{t=0}^{T-1}( e^{z_{t+1}} - e^{z_t} )_+ ~~\xxxx{3}{\leq}~~ 1 + \sum_{t=0}^{T-1}( {z_{t+1}} - {z_t} )_+ ~~\xxxx{4}{\leq} ~~ \ts 1 + q\sum_{t=0}^{T-1}\ln\frac{1+\NNN_{t+1}}{1+\NNN_t}\nn\\
&=  & \ts  1 + q\ln(1+\NNN_T) ~~\leq~~ 2 \ln (e+\NNN_T),\nn
\eeq
\noi where step \step{1} uses $\rho_0=1$ and $\rho_T\ge0$; step \step{2} uses the definition of $z_t:=\ln \rho_t$;  step \step{3} uses the fact that $e^x$ is $1$-Lipschitz on $(-\infty,0]$; step \step{4} uses $z_t = q\left(\ln(1+\NNN_t)-\ln(1+\GGG_t)\right)\le0$ and $z_{t+1}-z_t
=q\ln\frac{1+\NNN_{t+1}}{1+\NNN_t}
-q\ln\frac{1+\GGG_{t+1}}{1+\GGG_t}\leq q\ln\frac{1+\NNN_{t+1}}{1+\NNN_t}$, leading to $(z_{t+1}-z_t)_+ \le q\ln\frac{1+\NNN_{t+1}}{1+\NNN_t}$.

\paragraph{Part (b-ii).} Let $y:=\tfrac{u_t}{u_{t-1}}$. We derive:
\beq
&&\ts \sum_{t=0}^{T} |\frac{\rho_{t-1}}{\rho_t}-1 | ~~\xxxx{1}{=}~~ \ts \sum_{t=1}^{T}
|(\tfrac{u_t}{u_{t-1}})^q-1| ~~\xxxx{2}{=}~~ \sum_{t=1}^{T}
|y^q-1| ~~\xxxx{3}{\leq}~~2 \sum_{t=1}^{T}
 |\ln(y)| \nn\\
& = & \ts 2 \sum_{t=1}^{T} |\ln(u_t) - \ln(u_{t-1})|~~=~~2 \sum_{t=1}^{T} |\ln( \tfrac{1+\GGG_t}{1+\GGG_{t-1}})-\ln( \tfrac{1+\NNN_t}{1+\NNN_{t-1}})| \nn\\
& \xxxx{4}{\leq} & \ts 2 \sum_{t=1}^{T} \ln( \tfrac{1+\GGG_t}{1+\GGG_{t-1}})+\ln( \tfrac{1+\NNN_t}{1+\NNN_{t-1}})~~=~~2 \big(\ln(1+\GGG_T) + \ln(1+\NNN_T) \big) \leq 4 \ln(1+\GGG_T),\nn
\eeq
\noi where step \step{1} uses $\rho_t=u_t^{-q}$ for every $t\ge0$, and $\rho_{-1}=\rho_0=1$; step \step{2} uses the definition of $y$; step \step{3} uses the fact that $0<y:=\frac{u_t}{u_{t-1}}\le 2$ (which is implied by $u_{t+1} - u_t\leq 1$), and $|y^q-1|\le 2|\ln y|$ for any \(y\in(0,2]\) and \(q\in(0,1)\); step \step{4} uses $|\ln(a)-\ln(b)|\leq \ln(a) + \ln(b)$ for all $a,b\geq 1$.

\end{proof}

\section{Proofs for Section \ref{sect:it:L}}
\label{app:sect:it:L}

\subsection{Proof of Lemma \ref{lemma:OptMuonA:Prel:Bound}}
\label{app:lemma:OptMuonA:Prel:Bound}

\begin{proof} Define $\V_t:=\alpha_t\M_t$, and $\S_t:=\V_t-\nabla f(\X_t)$. Define $\bar{\gamma}_t:=\gamma_t/\alpha_t$.

\paragraph{Part (a).} For all $1\leq t\leq T$, we derive:
\beq
\ts \tfrac{\alpha_T}{\alpha_t} ~~=~~ \tfrac{\rho_{T-1}}{\rho_{t-1}} ~~=~~ \big(
\frac{1+\max_{1\le i\le T-1}g_i^2}{1+\sum_{i=1}^{T-1}g_i^2}\big)^q \cdot \big(
\frac{ 1+\sum_{i=1}^{t-1}g_i^2}{1+\max_{1\le i\le t-1}g_i^2}\big)^q  ~~\leq~~ \big(
1+\hG^2 \big)^{1/2}~~:=~~\varpi. \label{eq:rho:tk}
\eeq

\paragraph{Part (b).} Define $\bar{\gamma}_t :=\min\big(\alpha_t,\beta_t\big)$, where {$\beta_t:=\big(1 + \sum_{j=1}^t \alpha_j^{-1}\|\V_j\|_{\fro}^2\big)^{-1/2}=\big(1 + \sum_{j=1}^t \alpha_j \|\M_j\|_{\fro}^2\big)^{-1/2}$}. We have:
\beq
\tfrac{\bar{\gamma}_T}{\bar{\gamma}_t} ~~=~~ \tfrac{\min\big(\alpha_T,\beta_T\big)}{\min\big(\alpha_t,\beta_t\big)} ~~\xxxx{1}{\leq}~~ \max\big( \tfrac{\alpha_T}{\alpha_t} , \tfrac{\beta_T}{\beta_t}\big) ~~\xxxx{2}{\leq} ~~ \max\big( \tfrac{\alpha_T}{\alpha_t}, 1 \big) ~~\xxxx{3}{\leq}~~ \varpi.
\eeq
\noi where step \step{1} uses $\tfrac{\min(a,b)}{\min(c,d)}\leq \max(\tfrac{a}{c},\tfrac{b}{d})$ for all $a,b\geq 0$ and $c,d>0$; step \step{2} uses the fact that the sequence $\{\beta_t\}$ is non-increasing; step \step{3} uses $\varpi\geq 1$, and Inequality (\ref{eq:rho:tk}).

\paragraph{Part (c).} Using the update rule $\M_t = \alpha_t \big(\tfrac{\G_t}{\alpha_t}\big) + (1-\alpha_t)\M_{t-1}$ and the convexity of $\|\cdot\|_{\fro}$, we have:
\beq
\|\M_t\|_{\fro} ~~\leq~~ \max \big( \|\tfrac{\G_t}{\alpha_t}\|_{\fro}, \|\M_{t-1}\|_{\fro}\big) ~~\leq~~ \max_{i=1}^t \|\tfrac{\G_i}{\alpha_i}\|_{\fro},\nn
\eeq
\noi where the last inequality uses the standard initialization $\M_0=\zero$. This further leads to:
\beq
\|\V_t\|_{\fro} = \alpha_t\|\M_t\|_{\fro}  ~~\leq~~   \alpha_t \cdot \big(\max_{i=1}^t \|\tfrac{\G_i}{\alpha_i}\|_{\fro}\big) ~~\leq ~~ \varpi \hG~~ \leq ~~\varpi^2.\nn
\eeq

\paragraph{Part (d).} We derive:
\beq
\|\S_t\|_{\fro} ~~\leq ~~  \|\nabla f(\X_t)\|_{\fro} + \|\V_t\|_{\fro} ~~\leq~~\hG + \varpi \hG ~~\leq 2\varpi \hG ~~\leq~~  2\varpi^2.\nn
\eeq

\paragraph{Part (e).} We derive:
\beq
\|\epsilons_t\|_{\fro} ~~= ~~  \| \G_t - \nabla f(\X_t)\|_{\fro}  ~~\leq~~  2\hG.\nn
\eeq

\end{proof}

\subsection{Proof of Lemma \ref{lemma:OptMuonA:sumV:byG}}
\label{app:lemma:OptMuonA:sumV:byG}

\begin{proof} Define \(\V_t:=\alpha_t\M_t\), \(\G_t:=\nabla f(\X_t;\xi_t)\), and set
\(\M_0=\zero\).

\noi Let \(v_t:=\|\V_t\|_{\fro}\), \(g_t:=\|\G_t\|_{\fro}\), and $\kappa := 2(1+\hG)$.

\noi Consider the non-monotone coefficient  $\alpha_{t+1}=\rho_t=
\left(
\frac{1+\max_{1\le i\le t}g_i^2}
     {1+\sum_{i=1}^t g_i^2}
\right)^{q}$, where $q=1/2$,

\paragraph{Expansion of the scaled momentum.} Using the momentum recursion $\M_t=\G_t+(1-\alpha_t)\M_{t-1}$ and the initialization \(\M_0=\zero\), we obtain
\beq
\ts \M_t = \sum_{i=1}^t \left(\prod_{j=i+1}^t(1-\alpha_j)\right)\G_i. \nn
\eeq
Multiplying both sides by \(\alpha_t\) gives
\beq \label{eq:V:weighted:expansion}
\ts \V_t = \sum_{i=1}^t a_{t,i}\G_i, \qquad a_{t,i}:= \alpha_t\prod_{j=i+1}^t(1-\alpha_j) \geq 0.
\eeq

\paragraph{Bounding the Row Sums.} We derive:
\beq \label{eq:row:sum:a}
\ts \sum_{i=1}^t a_{t,i} &=& \ts \sum_{i=1}^t \alpha_t\prod_{j=i+1}^t(1-\alpha_j)~~\xxxx{1}{\leq}~~ \ts (1+\hG) \sum_{i=1}^t \alpha_i\prod_{j=i+1}^t(1-\alpha_j) \nn\\
&=& \ts (1+\hG) \left(1-\prod_{j=1}^t(1-\alpha_j)\right) \le 1+\hG,
\eeq
\noi where step \step{1} uses Lemma \ref{lemma:prop:para} that $\alpha_t\le (1+\hG)\alpha_i$ for all $i$ with $1\leq i\leq t$.

\paragraph{Bounding the Column Sums.} Fix \(i\in\{1,\ldots,T\}\). We derive
\beq \label{eq:column:sum:a}
\ts \sum_{t=i}^T a_{t,i} &=& \ts \alpha_i + \sum_{t=i+1}^T\alpha_t \prod_{j=i+1}^t(1-\alpha_j) ~~=~~ \ts \alpha_i + \sum_{t=i+1}^T \alpha_t(1-\alpha_t) \prod_{j=i+1}^{t-1}(1-\alpha_j)\nn\\
&\le& \ts 1+ \sum_{t=i+1}^T \alpha_t \prod_{j=i+1}^{t-1}(1-\alpha_j) ~~=~~ \ts 1+ \left(1-\prod_{j=i+1}^T(1-\alpha_j)\right) \le 2.
\eeq

\paragraph{Bounding the Term $\sum_{t=1}^T \|\V_t\|_{\fro}^2$.} We derive from \eqref{eq:V:weighted:expansion}:
\beq
\ts \|\V_t\|_{\fro}^2 &=& \ts \left\| \sum_{i=1}^t a_{t,i}\G_i\right\|_{\fro}^2 ~~\xxxx{1}{\leq}~~ \ts \left(\sum_{i=1}^t a_{t,i}\right) \left(\sum_{i=1}^t a_{t,i}\|\G_i\|_{\fro}^2\right) ~~\xxxx{2}{\leq}~~ \ts (1+\hG)\sum_{i=1}^t a_{t,i}g_i^2.
\nn
\eeq
\noi where step \step{1} uses the Cauchy--Schwarz Inequality; step \step{2} uses Inequality \eqref{eq:row:sum:a}. Summing the resulting inequality over \(t=1,\ldots,T\) yields:
\beq
\ts \sum_{t=1}^T\|\V_t\|_{\fro}^2 &\le& \ts (1+\hG)
\sum_{t=1}^T\sum_{i=1}^t a_{t,i}g_i^2 ~~=~~ \ts (1+\hG) \sum_{i=1}^T g_i^2 \sum_{t=i}^T a_{t,i} ~~\xxxx{1}{\le}~~ \ts \kappa\sum_{i=1}^T g_i^2,
\nn
\eeq
\noi where step \step{1} uses Inequality \eqref{eq:column:sum:a}, and $\kappa := 2(1+\hG)$.

\end{proof}

\subsection{Proof of Lemma \ref{lemma:VR:bound:A}}
\label{app:lemma:VR:bound:A}

\begin{proof} Define $\mathbf S_t := \V_t - \nabla f(\mathbf X_t)$, where $\V_t := \alpha_t \mathbf M_t$. Define $m_t:=\|\M_t\|_{\fro}$, and $g_t:=\|\G_t\|_{\fro}$.

\noi Define $\epsilons_t := \mathbf G_t - \nabla f(\mathbf X_t)$, and $\mathbf A_t:=(1-\alpha_{t})\mathbf S_{t-1} +\alpha_{t}\epsilons_t$.

\noi Define $\D_t:=(1-\alpha_t) (\nabla f(\X_{t-1}) - \nabla f(\X_{t}))$, and $\J_t:=(1-\alpha_t) (\alpha_t-\alpha_{t-1})\M_{t-1}$.

\noi Define $z_t:=2 (1-\alpha_{t}) \alpha_{t} (1+\alpha_t) \la \S_{t-1} ,\epsilons_{t}\ra$, and $\delta:= n L^2\theta^2$.

\paragraph{Difference-based representation of $\S_t$.} We derive the following results:
\beq
\mathbf S_t  &\overset{\step{1}}{=}& \alpha_t\mathbf M_t-\nabla f(\mathbf X_t)   \nn\\
&\overset{\step{2}}{=}& \underbrace{\ts (1-\alpha_{t})\mathbf S_{t-1}
+\alpha_{t}\epsilons_t }_{:=\A_t} +\alpha_t(1-\alpha_t)\mathbf M_{t-1}
-\nabla f(\mathbf X_t) - (1-\alpha_t) \S_{t-1}  + \alpha_t \nabla f(\X_t)  \nn \\
&\overset{\step{3}}{=}& \A_t  + \underbrace{\ts (1-\alpha_t)(\alpha_t- \alpha_{t-1})\mathbf M_{t-1}  }_{:=\mathbf J_t} +  \underbrace{\ts (1-\alpha_t) ( \nabla f(\X_{t-1}) - \nabla f(\mathbf X_t)  ) }_{:=\mathbf D_t},\nn
\eeq
\noi where $\step{1}$ uses $\mathbf S_t:=\V_t-\nabla f(\mathbf X_t)$ and $\V_t:=\alpha_t\mathbf M_t$; step \step{2} uses the momentum recursion $\mathbf M_t=\mathbf G_t+(1-\alpha_t)\mathbf M_{t-1}$, and $\mathbf G_t=\nabla f(\mathbf X_t)+\boldsymbol\epsilon_t$; step \step{3} uses the definition of $\S_t$.

\paragraph{Bounding the Term $\|\D_t\|_{\fro}^2$.} By the $L$-smoothness of $f(\cdot)$ and the update rule of $\X_t$, we have
\beq \label{eq:D:bound:scaled}
\ts && \ts \|\D_t\|_{\fro}^2 ~~\leq~~ \|\nabla f(\X_{t-1})-\nabla f(\X_t)\|_{\fro}^2 ~~\overset{\step{1}}{\leq}~~ L^2\|\X_{t-1}-\X_t\|_{\fro}^2 \nn\\
&\overset{\step{2}}{=}& \ts L^2\theta^2\gamma_{t-1}^2m_{t-1}^2\|\Orth(\M_{t-1})\|_{\fro}^2~~\overset{\step{3}}{\leq}~~ \ts n L^2\theta^2 \gamma_{t-1}^2m_{t-1}^2  = \delta \gamma_{t-1}^2m_{t-1}^2,
\eeq
\noi where step \step{1} uses the $L$-smoothness of $f(\cdot)$; step \step{2} uses $\X_{t+1}=\X_{t}-\theta\gamma_{t}m_{t}\Orth(\M_{t})$; step \step{3} uses $\|\Orth(\M_{t-1})\|_{\fro}^2\leq n$.

\paragraph{Bounding the Term $\|\A_t+\D_t\|_{\fro}^2$.} We derive:
\beq \label{eq:bound:ADAD}
&& \|\A_t+\D_t\|_{\fro}^2 ~~\xxxx{1}{\leq}~~ \ts (1+\tfrac{1}{\alpha_t})\|\D_t\|_{\fro}^2 + (1+\alpha_t)\|\A_t\|_{\fro}^2  ~~\leq~~ \tfrac{2}{\alpha_t} \|\D_t\|_{\fro}^2 +  (1+\alpha_t)\|\A_t\|_{\fro}^2  \nn\\
&=& \ts \tfrac{2}{\alpha_t}\|\D_t\|_{\fro}^2 + (1+\alpha_t) \big( (1-\alpha_{t})^2\|\mathbf S_{t-1}\|_{\mathrm F}^2 + \alpha_{t}^2 \|\epsilons_t\|_{\mathrm F}^2 + 2 (1-\alpha_{t}) \alpha_{t}\la \S_{t-1} ,\epsilons_{t}\ra \big)  \nn\\
&\xxxx{2}{\leq}& \ts   \tfrac{2}{\alpha_t}\|\D_t\|_{\fro}^2 +   (1-\alpha_{t}) \|\mathbf S_{t-1}\|_{\mathrm F}^2 + 2 \alpha_{t}^2 \|\epsilons_t\|_{\mathrm F}^2   +z_t   \nn\\
&\xxxx{3}{\leq} & \ts    \tfrac{2}{\alpha_t} n L^2 \theta^2 \gamma_{t-1}^2 m_{t-1}^2 +  (1-\alpha_{t}) \|\mathbf S_{t-1}\|_{\mathrm F}^2 + 2 \alpha_{t}^2 \|\epsilons_t\|_{\mathrm F}^2 + z_t   \nn\\
&\leq & \ts  \tfrac{2}{\alpha_{t-1}} \tfrac{3}{2} \cdot \delta \gamma_{t-1}^2 m_{t-1}^2 +  (1-\alpha_{t}) \|\mathbf S_{t-1}\|_{\mathrm F}^2 + 2 \alpha_{t}^2 \|\epsilons_t\|_{\mathrm F}^2 +z_t,
\eeq
\noi where step \step{1} uses Young's inequality \(\|\A_t+\D_t\|_{\fro}^2\le (1+\alpha_t)\|\A_t\|_{\fro}^2+(1+\alpha_t^{-1})\|\D_t\|_{\fro}^2\); step \step{2} uses $\big(1+ a \big)(1-a) \le 1$ for all $a\in[0,1]$; step \step{3} uses Inequality (\ref{eq:D:bound:scaled}); step \step{4} uses $\alpha_{t-1}\leq (1+\frac{1}{2})\alpha_t$.

\paragraph{Bounding the Term $\|\S_t\|_{\fro}^2$.} We derive:
\beq
\|\S_t\|_{\fro}^2 &=& \ts \la \S_t, \A_t+\D_t \ra + \la \S_t, \J_t \ra ~~\xxxx{1}{\leq} ~~ \tfrac{1}{2}\|\S_t\|_{\fro}^2 + \tfrac{1}{2}\|\A_t+\D_t\|_{\fro}^2 + \|\S_t\|_{\fro}\|\J_t\|_{\fro} \nn\\
&\leq & \ts \|\A_t+\D_t\|_{\fro}^2 + 2 \|\S_t\|_{\fro}\|\J_t\|_{\fro} ~~\xxxx{2}{\leq}~~ \|\A_t+\D_t\|_{\fro}^2 + 4 \varpi^2\|\J_t\|_{\fro}   \nn\\
&\xxxx{3}{\leq} & \ts \|\A_t+\D_t\|_{\fro}^2 + 4 \varpi^2 \tfrac{|\alpha_t-\alpha_{t-1}|}{\alpha_{t-1}} \|\V_{t-1}\|_{\fro}   ~~\xxxx{4}{\leq} ~~ \|\A_t+\D_t\|_{\fro}^2 + 4 \varpi^4 \cdot \tfrac{|\alpha_t-\alpha_{t-1}|}{\alpha_{t-1}} \nn\\
&\xxxx{5}{\leq}& \ts (1-\alpha_{t}) \|\mathbf S_{t-1}\|_{\mathrm F}^2 + \underbrace{\ts z_t + 3 \delta \gamma_{t-1}^2 m_{t-1}^2/\alpha_{t-1} +  2 \alpha_{t}^2 \|\epsilons_t\|_{\mathrm F}^2  + 4 \varpi^5 \cdot \tfrac{|\alpha_t-\alpha_{t-1}|}{\alpha_{t}}}_{:=B_t}, \nn
\eeq
\noi where step \step{1} uses Young's inequality; step \step{2} uses Lemma~\ref{lemma:OptMuonA:Prel:Bound} that $\|\S_t\|_{\fro}\le 2\varpi^2$; step \step{3} uses the definition $\J_t=(1-\alpha_t)(\alpha_t-\alpha_{t-1})\M_{t-1}=(1-\alpha_t)(\alpha_t-\alpha_{t-1})\frac{\V_{t-1}}{\alpha_{t-1}}$, and the bound $1-\alpha_t\le1$; step \step{4} uses $\|\V_{t-1}\|_{\fro}\le \varpi^2$; step \step{5} uses Inequality (\ref{eq:bound:ADAD}).

\end{proof}

\subsection{Proof of Lemma \ref{lemma:suff:descent:A}}
\label{app:lemma:suff:descent:A}

\begin{proof} Let $\bar\gamma_t:=\gamma_t/\alpha_t$, $\V_t=\alpha_t\M_t$, and $\S_t:=\V_t-\nabla f(\X_t)$.

\noi By average \(L\)-smoothness,
\beq
f(\X_{t+1})-f(\X_t) &\xxxx{}{\leq}& \tfrac{L}{2}\|\X_{t+1}-\X_t\|_{\fro}^2+\langle \nabla f(\X_t),\X_{t+1}-\X_t\rangle \nn\\
&\xxxx{1}{=}&\tfrac{L\theta^2}{2}\bar\gamma_t^2\|\V_t\|_{\fro}^2\|\Orth(\V_t)\|_{\fro}^2-\theta\bar\gamma_t\|\V_t\|_{\fro}\langle \Orth(\V_t),\nabla f(\X_t)\rangle
\nn\\
&\xxxx{2}{\leq}&\tfrac{L\theta^2}{2}\bar\gamma_t^2\|\V_t\|_{\fro}^2\cdot n-\theta\bar\gamma_t\|\V_t\|_{\fro}\langle \Orth(\V_t),\V_t-\S_t\rangle\nn\\
&\xxxx{3}{\leq}&\tfrac{Ln\theta^2}{2}\bar\gamma_t^2\|\V_t\|_{\fro}^2-\theta\bar\gamma_t\|\V_t\|_{\fro}^2+\theta\bar\gamma_t\|\V_t\|_{\fro}\|\S_t\|_*
\nn\\
&\xxxx{4}{\leq}&\tfrac{Ln\theta^2}{2}\bar\gamma_t^2\|\V_t\|_{\fro}^2-\tfrac{\theta}{2}\bar\gamma_t\|\V_t\|_{\fro}^2+\tfrac{\theta n}{2}\bar\gamma_t\|\S_t\|_{\fro}^2,
\eeq
\noi where step \step{1} uses $\X_{t+1}=\X_t-\theta\bar\gamma_t\|\V_t\|_{\fro}\Orth(\V_t)$ since $\alpha_t>0$ and $\Orth(\M_t)=\Orth(\V_t)$; step \step{2} uses $\|\Orth(\V_t)\|_{\fro}^2\le n$, and $\S_t:=\V_t-\nabla f(\X_t)$; step \step{3} uses $\langle\Orth(\V_t),\V_t\rangle=\|\V_t\|_*\ge\|\V_t\|_{\fro}$, and $\la \Orth(\V_t),\S_t \ra \leq \|\S_t\|_*$; step \step{4} uses $\|\S_t\|_*^2\le n\|\S_t\|_{\fro}^2$, and $ab\le a^2/2+b^2/2$. Rearranging gives the desired result.
\end{proof}

\subsection{Proof of Lemma \ref{lemma:bcd:MuonA}}
\label{app:lemma:bcd:MuonA}

\begin{proof} Set $\gamma_0=0$, $\V_0=\M_0=\zero$, and $v_0=m_0=0$.

\noi Let $r_t:=\frac{v_t^2}{\alpha_t}$, $H_t:=\sum_{i=1}^t g_i^2$, and $A_t:=1+\sum_{i=1}^t r_i$. Define $\S_t:=\V_t-\nabla f(\X_t)$.

\noi Set $B_t:=  2 \alpha_{t}^2\|\epsilons_t\|_{\fro}^2 + 2 \alpha_{t} (1-\alpha_{t}) (1+\alpha_{t}) \la \S_{t-1} ,\epsilons_{t}\ra + 4 \varpi^5 |\tfrac{\alpha_{t-1}}{\alpha_{t}}-1| + 3 \delta \frac{\gamma_{t-1}^2m_{t-1}^2}{ \alpha_{t-1}}$.

\paragraph{Bounding the Term $\EEE[\sum_{t=1}^T\alpha_t^2 \|\epsilons_t\|_{\fro}^2]$.} We have:
\beq \label{eq:abcd:alpha-g}
\ts \EEE[\sum_{t=1}^T\alpha_t^2 \|\epsilons_t\|_{\fro}^2] &\xxxx{1}{\le}&\ts \EEE[\tfrac{9}{4}\sum_{t=1}^T\alpha_{t+1}^2 \|\epsilons_t\|_{\fro}^2] ~~\xxxx{2}{\le} ~~ \ts  \EEE[\tfrac{9}{4} (1+\hG^2) \sum_{t=1}^T \tfrac{ \|\epsilons_t\|_{\fro}^2}{1+\GGG_t} ] \nn\\
&\xxxx{3}{\leq}& \ts \EEE[\tfrac{9}{4} (1+\hG^2)^2 8\sum_{t=1}^T \tfrac{\|\G_t\|_{\fro}^2}{1+\GGG_t} ] ~~\leq ~~ \ts  \EEE[36 (1+\hG^4) \ln(1+\GGG_T)], ~~~~
\eeq
\noi where step \step{1} uses $\alpha_t \leq \alpha_{t+1} (1+q)$; step \step{2} uses $\alpha_{t+1}^2\leq (1+\hG^2)/(1+\GGG_t)$; step \step{3} uses Lemma \ref{lemma:weighted:self:normalized} with $W_0=1$.

\paragraph{Bounding the Term $\sum_{t=1}^T\alpha_{t-1}^2v_{t-1}^2$.} We have:
\beq \label{eq:alpha2:v2}
&&\ts \sum_{t=1}^T\alpha_{t-1}^2v_{t-1}^2 ~~\xxxx{1}{\leq}~~ \ts \sum_{t=1}^T\alpha_{t}^2v_{t}^2 ~~\xxxx{2}{\leq}~~\ts \tfrac{9}{4}\sum_{t=1}^T\alpha_{t+1}^2v_{t}^2 ~~\xxxx{3}{\leq}~~\ts \tfrac{9}{4} (1+\hG^2) \sum_{t=1}^T \tfrac{v_{t}^2}{1+\GGG_t} \nn\\
&\xxxx{4}{\leq}&\ts \tfrac{9}{4} (1+\hG^2) \cdot \kappa \sum_{t=1}^T \tfrac{v_{t}^2}{1+\VVV_t}  ~~\xxxx{5}{\leq}~~\ts \tfrac{9 \kappa }{4} (1+\hG^2) \cdot  \ln(1+ \VVV_T)\nn\\
&\xxxx{6}{\leq}&\ts \tfrac{9 \kappa}{4} (1+\hG^2) \cdot \ln(1+ \kappa \GGG_T) ~~{\leq}~~ \tfrac{9\kappa}{4} \kappa^2 \ln(e+\kappa)\cdot \ln(e+ \GGG_T) ,
\eeq
\noi where step \step{1} uses $v_0=0$; step \step{2} uses $\alpha_t \leq \alpha_{t+1} (1+q)$; step \step{3} uses $\alpha_{t+1}^2\leq (1+\hG^2)/(1+\GGG_t)$; step \step{4} uses $\VVV_T\leq \kappa \GGG_T$, as established in Lemma~\ref{lemma:OptMuonA:sumV:byG}; step \step{5} uses Lemma~\ref{lemma:adaptive:important}(a); step \step{6} uses $\VVV_T\leq \kappa \GGG_T$.

\paragraph{Bounding the Term $\sum_{t=1}^T\gamma_{t-1}^2m_{t-1}^2/\alpha_{t-1}$.} We have:
\beq \label{eq:gamma:m:alpha}
&& \ts \sum_{t=1}^T \gamma_{t-1}^2m_{t-1}^2/\alpha_{t-1}~~\xxxx{1}{\leq}~~ \ts \sum_{t=1}^T \gamma_{t}^2m_{t}^2/\alpha_{t} ~~\xxxx{2}{\leq}~~\ts  \sum_{t=1}^T \tfrac{\alpha_{t} m_{t}^2}{1 + \sum_{i=1}^t \alpha_{i} m_{i}^2 } \nn\\
&\xxxx{3}{\leq}& \ts \ln (1 + \sum_{t=1}^T \alpha_{t} m_{t}^2 ) ~~\xxxx{4}{\leq}~~ \ts \ln (1 + \varpi \kappa \GGG_T \alpha_{T}^{-1} ) ~~\xxxx{5}{\leq}~~ \ts \ln (1 + \varpi \kappa (1+\GGG_T)^{3/2})\nn\\
&\xxxx{}{\leq}& \ts \ln (2 + \varpi \kappa  \GGG_T^{3/2})
~~\leq ~~ 2 \ln (e + \varpi \kappa)  \ln(e+\GGG_T),
\eeq
\noi where step \step{1} uses $m_0=0$; step \step{2} uses the definition of $\gamma_t$ that $\gamma_t\leq \alpha_t \big(1  + \sum_{j=1}^t \alpha_j m_j^2 \big)^{-1/2}$; step \step{3} uses Lemma~\ref{lemma:adaptive:important}(a); step \step{4} uses $\sum_{t=1}^T \alpha_t m_t^2 = \sum_{t=1}^T \tfrac{1}{\alpha_t} v_t^2 \leq \tfrac{\varpi}{\alpha_T} \sum_{t=1}^T v_t^2 \leq \kappa \varpi \alpha_T^{-1} \GGG_T$, which can be implied by $\VVV_T\leq \kappa \GGG_T$; step \step{5} uses $\alpha_{T}^{-1}\leq (1+\GGG_{T})^{1/2}$; step \step{6} uses $\ln(e+ab)\leq \ln(e+a)\ln(e+b)$ for all $a,b\geq 0$.

\paragraph{Part (a): Bounding the Term $\EEE[ \sum_{t=1}^T B_t]$.} We derive:
\beq
\ts \EEE[ \sum_{t=1}^T B_t] & \xxxx{1}{=}& \ts \sum_{t=1}^T \EEE[2\alpha_t^2\|\epsilons_t\|_{\fro}^2    +  4 \varpi^5 |\tfrac{\alpha_{t-1}}{\alpha_{t}}-1| + 3\delta \frac{\gamma_{t-1}^2m_{t-1}^2}{ \alpha_{t-1}}  ] \nn\\
& \xxxx{2}{\leq}& \ts \left(36 (1+\hG^4) + 16 \varpi^5 + 9\delta \ln (e + \varpi \kappa) \right)\cdot \EEE[\ln(e+\GGG_T)] \nn\\
& \xxxx{3}{\leq}& \ts \underbrace{\ts \left( 52 \varpi^5 + 9\delta \ln (e + \varpi \kappa) \right)\cdot \ln(e+\hG^2) }_{:=C_b}\cdot \ln(e+T), \nn
\eeq
\noi where step \step{1} uses the fact that the martingale term has zero expectation because $\alpha_t$ and $\S_{t-1}$ are $\FF_{t-1}$-measurable and
$\EEE[\epsilons_t\mid\FF_{t-1}]=0$; step \step{2} uses Inequalities (\ref{eq:abcd:alpha-g}) and (\ref{eq:gamma:m:alpha}), and the second inequality in Lemma \ref{lemma:prop:para}(b); step \step{3} uses Assumption \ref{ass:BG} to get $\GGG_T\le T\hG^2$ a.s., leading to $\EEE[\ln(e+\GGG_T)]\le\ln(e+T\hG^2)\le\ln(e+T)\ln(e+\hG^2)$, and the fact that $1+x^4\leq (1+x^2)^{5/2}$ with $x=\hG$.

\paragraph{Part (b). Bounding the Term $ \EEE[\sum_{t=1}^T \alpha_t s_t^2]$.} From Lemma \ref{lemma:VR:bound:A}, we derive: $s_t^2 \le (1-\alpha_t) s_{t-1}^2 + B_t$, leading to the following inequality:
\beq
\ts \alpha_t s_{t}^2 \le (1-\alpha_t) (s_{t-1}^2 - s_t^2) + B_t.\nn
\eeq
\noi Summing over $t=1,\ldots,T$ and taking expectation, we obtain
\beq
\ts \EEE[\sum_{t=1}^T \alpha_t s_{t}^2] &\leq & \ts \EEE[\sum_{t=1}^T (1-\alpha_t) (s_{t-1}^2 - s_t^2) + B_t ] \nn\\
&= & \ts \EEE[\big( \sum_{t=1}^T B_t \big) + (1 - \alpha_1) s_0^2 - (1 - \alpha_T) s_T^2 + \big( \sum_{t=1}^{T-1} (\alpha_t-\alpha_{t+1}) s_{t}^2 \big)  ] \nn\\
&\xxxx{1}{\leq} & \ts C_b \ln(e+T) + \EEE[ \big( \sum_{t=1}^{T-1} (\alpha_t-\alpha_{t+1}) s_{t}^2 \big) ] \nn\\
&\xxxx{}{\leq} & \ts C_b \ln(e+T) + \EEE[ \big( \max_{t=1}^T s_{t}^2 \big)\big( \sum_{t=1}^{T-1} \max(0,\alpha_t-\alpha_{t+1}) \big) ] \nn\\
&\xxxx{2}{\leq}&\ts C_b \ln(e+T) + \EEE[ 4 \varpi^4 \cdot 2 \ln (e + \hG^2)  ] \nn\\
&\xxxx{3}{\leq}& \ts \big( C_b + 8 \varpi^4 \ln (e+\hG^2) \big) \ln(e+T) ~~\leq ~~ \underbrace{2 C_b}_{:=C_c} \cdot \ln(e+T), \nn
\eeq
\noi where step \step{1} uses $s_T^2\geq 0$; step \step{2} uses $s_t^2\leq 4\varpi^4$, and Lemma \ref{lemma:prop:para}(b-i) that $\sum_{t=0}^{T-1} \max(0,\rho_t-\rho_{t+1}) \leq 2 \ln(e+\NNN_T)$; step \step{3} uses $\ln(e+T)\geq 1$. Since $\bar{\gamma}_t:=\tfrac{\gamma_t}{\alpha_t}\leq \alpha_t$, we directly conclude that $\EEE[\sum_{t=1}^T \bar{\gamma}_t s_t^2]\leq \EEE[\sum_{t=1}^T \alpha_t s_t^2]\leq C_c\ln(e+T)$.

\paragraph{Part (c): Bounding the Term $\sum_{t=1}^T \bar{\gamma}_t^2 v_t^2$.} We derive:
\beq
&& \ts \sum_{t=1}^T \bar{\gamma}_t^2 v_t^2 ~~ = ~~ \ts \sum_{t=1}^T \tfrac{\gamma_t^2}{\alpha_t^2} v_t^2 ~~\xxxx{1}{\leq} ~~ \ts \sum_{t=1}^T \alpha_t^2 v_t^2 \nn\\
&\xxxx{2}{\leq}& \tfrac{9\kappa^3}{4} \ln(e+\kappa) \ln(e+T\hG^2) ~~ \xxxx{3}{\leq}~~ \underbrace{\ts 3\kappa^3 \ln(e+\kappa) \ln(e+\hG^2) }_{:=C_d}\cdot \ln(e+T),\nn
\eeq
\noi where step \step{1} uses the definition of $\gamma_t$ that $\gamma_t\leq \alpha_t^2$; step \step{2} uses Inequality (\ref{eq:alpha2:v2}); step \step{3} uses $\ln(1+ab)\leq \ln(e+a)\ln(e+b)$ for all $a,b\geq 0$.

\end{proof}

\subsection{Proof of Lemma \ref{lemma:U:square:MuonA}} \label{app:lemma:U:square:MuonA}

\begin{proof}
Define $s_t:=\|\S_t\|_{\fro}$, $w_t:=\bar\gamma_t s_t^2$, $\UUU_t:=\sum_{i=1}^t w_i$, $\UUU_0:=0$.

\noi Let $D_t:=2\alpha_t^2\|\epsilons_t\|_{\fro}^2 + 4 \varpi^5 |\tfrac{\alpha_{t-1}}{\alpha_t}-1|+ 3 \delta {\gamma_{t-1}^2m_{t-1}^2}/{\alpha_{t-1}}$.

 \noi Let $K_t:=2\alpha_t(1-\alpha_t)(1+\alpha_t) \langle \S_{t-1},\epsilons_t\rangle$, and $B_t= D_t + K_t$.

\paragraph{Bounding the Term $\EEE[\sum_{t=1}^T\UUU_{t-1}2\alpha_t^2\|\epsilons_t\|_{\fro}^2]$.} Note that $\UUU_{t-1}$ is $\FF_{t-1}$-measurable and $\UUU_{t-1}\le \UUU_T$ almost surely. Applying Lemma \ref{lemma:weighted:self:normalized} with $W_{t-1}=\UUU_{t-1}$ gives the following estimate:
\beq
&& \ts \EEE[\sum_{t=1}^T \UUU_{t-1}\,2\alpha_t^2\|\epsilons_t\|_{\fro}^2 ] ~~\xxxx{1}{\le}~~ \frac{9}{2}(1+\hG^2) \EEE[\sum_{t=1}^T \UUU_{t-1} \frac{\|\epsilons_t\|_{\fro}^2}{1+\sum_{i=1}^t g_i^2}] \nn \\
&\xxxx{2}{\le}& 36(1+\hG^2)^2 \EEE[\ln(1+\GGG_T) \UUU_T]~~\xxxx{3}{\le}~~ 36(1+\hG^2)^2 \ln(e+\hG^2) \cdot C_c \ln^2(e+T),
\eeq
\noi where step \step{1} uses $\frac{\alpha_t}{\alpha_{t+1}}\leq \tfrac{3}{2}$; step \step{2} uses Lemma \ref{lemma:weighted:self:normalized} with the predictable weight $W_{t-1}=\UUU_{t-1}$; step \step{3} uses $\GGG_T\leq T\hG^2$, and Lemma \ref{lemma:bcd:MuonA}(b).

\paragraph{Bounding the Term $\EEE[\sum_{t=1}^T\UUU_{t-1} 2\alpha_{t-1}^2v_{t-1}^2]$.} We derive:
\beq
\ts \EEE[\sum_{t=1}^T\UUU_{t-1}\,2\alpha_{t-1}^2v_{t-1}^2] &\xxxx{1}{\leq}& \EEE[ 18 (1+\hG^2)\ln(e+\GGG_T) \UUU_T]\nn\\
&\xxxx{2}{\leq}& 18(1+\hG^2) \ln(e+\hG^2) C_c \ln^2(e+T),
\eeq
\noi where step \step{1} uses $\UUU_{t-1}\le\UUU_T$ and the same pathwise estimate used in Lemma \ref{lemma:bcd:MuonA} to control
$\EEE[\sum_{t=1}^T\alpha_{t-1}^2v_{t-1}^2]$; step \step{2} uses Lemma \ref{lemma:bcd:MuonA}(b).

\paragraph{Bounding the Term $\EEE[\sum_{t=1}^T\UUU_{t-1} 3\delta \frac{\gamma_{t-1}^2m_{t-1}^2}{\alpha_{t-1}}]$.} We obtain
\beq
& &\ts \EEE [\sum_{t=1}^T\UUU_{t-1} 3\delta \frac{\gamma_{t-1}^2m_{t-1}^2}{\alpha_{t-1}} ]~~\xxxx{1}{\le}~~ 6 \delta \ln(e+\hG^2) \ln(e + \varpi \kappa) C_c \ln^2(e+T),
\eeq
\noi where step \step{1} uses the pathwise estimate used in Inequality (\ref{eq:gamma:m:alpha}) that $\EEE[\sum_{t=1}^T\frac{\gamma_{t-1}^2m_{t-1}^2}{\alpha_{t-1}}]\le \EEE[2 \ln(e + \varpi \kappa)\ln(e+\GGG_T)]$.

\paragraph{Bounding the Term $\EEE[\sum_{t=1}^T \UUU_{t-1}D_t]$.} We derive:
\beq \label{eq:CwCw}
&& \ts \EEE[\sum_{t=1}^T \UUU_{t-1}D_t] \nn\\
&= & \ts \sum_{t=1}^T \EEE[ \UUU_{t-1} \cdot 2\alpha_t^2\|\epsilons_t\|_{\fro}^2 + \UUU_{t-1} \cdot 4 \varpi^5 | \tfrac{\alpha_{t-1}}{\alpha_t} -1| + \UUU_{t-1} \cdot 3\delta {\gamma_{t-1}^2m_{t-1}^2}/{\alpha_{t-1}} ] \nn\\
&\leq &  \big\{ 36 (1+\hG^2)^2 \cdot C_c + 16 \varpi^5  + 6\delta C_c \big\} \ln(e + \varpi \kappa) \cdot \ln(e+\hG^2) \ln^2 (e+T)\nn\\
&\leq &  \underbrace{\ts \big( 88 \varpi^5 + 6 \delta \big) \cdot \ln(e+\hG^2) }_{:=C_e} \cdot C_c \ln^2 (e+T),
\eeq

\paragraph{Bounding the Term $\EEE\left[\sum_{t=1}^T\UUU_{t-1}w_t\right]$.} Using the fact that $\alpha_t s_t^2 = s_t^2-(1-\alpha_t)s_t^2$, we have from the recursion:
\beq
 \alpha_t s_t^2 \le (1-\alpha_t)(s_{t-1}^2-s_t^2)+B_t ~~\xxxx{1}{\Rightarrow}~~ w_t \le (1-\alpha_t)(s_{t-1}^2-s_t^2)+B_t. \nn
\eeq
\noi where step \step{1} uses $w_t=\bar\gamma_t s_t^2\le \alpha_t s_t^2$, which can be implied by $\bar\gamma_t\le\alpha_t$. Therefore,
\beq
&& \ts \sum_{t=1}^T\UUU_{t-1}w_t \nn\\
&\le& \ts\sum_{t=1}^T\UUU_{t-1}B_t + \sum_{t=1}^T \UUU_{t-1}(1-\alpha_t)(s_{t-1}^2-s_t^2)   \nn\\
& = & \ts \sum_{t=1}^T\UUU_{t-1}B_t + \UUU_0 (1 - \alpha_1) s_0^2 - \UUU_{T-1} (1 - \alpha_T) s_T^2 + \sum_{t=1}^{T-1} [\UUU_{t} (1-\alpha_{t+1}) -\UUU_{t-1} (1-\alpha_{t}) ] s_{t}^2  \nn\\
& \xxxx{1}{\leq} & \ts \sum_{t=1}^T\UUU_{t-1}B_t + \sum_{t=1}^{T-1} [\UUU_{t} (1-\alpha_{t+1}) -\UUU_{t-1} (1-\alpha_{t}) ] s_{t}^2  \nn\\
& \xxxx{ }{=} & \ts \sum_{t=1}^T\UUU_{t-1}B_t + \sum_{t=1}^{T-1} [(\UUU_t - \UUU_{t-1}) (1-\alpha_{t+1}) + \UUU_{t-1} (\alpha_t - \alpha_{t+1}) ] s_{t}^2  \nn\\
& \xxxx{ }{\leq } & \ts \sum_{t=1}^T\UUU_{t-1}B_t + \sum_{t=1}^{T-1} [(\UUU_t - \UUU_{t-1}) + \UUU_{t-1} \max(0,\alpha_t - \alpha_{t+1}) ] s_{t}^2  \nn\\
& \xxxx{ }{\leq } & \ts \sum_{t=1}^T\UUU_{t-1}B_t + (\max_{0\leq t\leq T} s_t^2) \UUU_T + (\max_{0\leq t\leq T} s_t^2) \UUU_T  \big( \sum_{t=1}^{T-1} \max(0,\alpha_t - \alpha_{t+1}) \big)  \nn\\
&\xxxx{2}{\leq} & \ts \sum_{t=1}^T\UUU_{t-1}(D_t + K_t) + \UUU_T \cdot 4\varpi^4 \cdot 4 \ln (e+\hG^2), \nn
\eeq
\noi where step \step{1} uses $\alpha_1=1$, and $\UUU_{T-1}\geq0$; step \step{2} uses $s_t\leq 2\varpi^2$.

Note that $\EEE[\UUU_{t-1}K_t] =\EEE[\UUU_{t-1}\EEE_t[K_t]]=0$, which holds since $\UUU_{t-1}$ is $\FF_{t-1}$-measurable and $\EEE_t[\epsilons_t]=\zero$. Taking expectations, we get
\beq \label{eq:weighted:cross:MuonA:new}
\ts \EEE\left[\sum_{t=1}^T\UUU_{t-1}w_t\right] &\le& \ts 16 \varpi^4 \ln(e+\hG^2) \EEE[\UUU_T] + \EEE[\sum_{t=1}^T\UUU_{t-1}D_t ] \nn \\
&\xxxx{1}{\leq}& 16\varpi^4 \ln(e+\hG^2) C_c \ln(e+T) + C_e C_c \ln^2(e+T) \nn \\
&\le& \big( 16 \varpi^4 \ln(e+\hG^2) + C_e\big)\cdot C_c \ln^2(e+T),
\eeq
\noi step \step{1} uses Inequality (\ref{eq:CwCw}).

\paragraph{Bounding \(\EEE[\UUU_T^2]\).}
Using $\UUU_T^2=\left(\sum_{t=1}^T w_t\right)^2=\sum_{t=1}^T w_t^2+2\sum_{t=1}^T\UUU_{t-1}w_t$, and $w_t^2=\bar\gamma_t^2s_t^4\le 4\varpi^4\bar\gamma_t s_t^2=4\varpi^4 w_t$, we obtain
\beq
\ts \EEE[\UUU_T^2] &\le& \ts 4\varpi^4\EEE[\UUU_T] + 2\EEE[\sum_{t=1}^T\UUU_{t-1}w_t] \nn\\
&\le& \ts 4\varpi^4C_c \ln(e+T) + 2( 16 \varpi^4 \ln(e+\hG^2) + C_e) C_c \ln^2(e+T) \nn \\
&\le& {\ts \big(4\varpi^4+32\varpi^4\ln(e+\hG^2)+2C_e\big)C_c} \ln^2(e+T)\nn \\
&\le& \underbrace{\ts \big(36\varpi^4\ln(e+\hG^2)+2C_e\big)C_c}_{:=C_u}\ln^2(e+T).\nn
\eeq

\end{proof}

\subsection{Proof of Lemma \ref{R_R}}
\label{app:R_R}

\begin{proof} Let $F_t:=f(\X_{t}) - f_*$, and $\bar\gamma_t=\gamma_t/\alpha_t$. Let $\bar\gamma_t = \min\{ \alpha_t,\,
\big(1+\sum_{j=1}^t\frac{\|\V_j\|_{\fro}^2}{\alpha_j}\big)^{-1/2}\}$.

\noi Define $\UUU_T:=\sum_{t=1}^T \bar\gamma_t\|\S_t\|_{\fro}^2$, $\WWW_T:=\sum_{t=1}^T\bar\gamma_t\|\V_t\|_{\fro}^2$, and $\VVV_T:=\sum_{t=1}^T\|\V_t\|_{\fro}^2$.

\paragraph{Bounding the term $\EEE[\WWW_T]$.} We have from Lemma \ref{lemma:suff:descent:A} that:
\beq
\bar\gamma_t\|\V_t\|_{\fro}^2
\le \tfrac{1}{c_0} \left( f(\X_{t})  - f(\X_{t+1}) + c_1 \bar\gamma_t\|\S_t\|_{\fro}^2 + c_2 \bar\gamma_t^2\|\V_t\|_{\fro}^2\right).\nn
\eeq
\noi Summing this inequality from $t=1$ to $T$:
\beq \label{eq:pathwise:ineq}
\ts \sum_{t=1}^T\bar\gamma_t\|\V_t\|_{\fro}^2 &\leq& \ts \tfrac{1}{c_0}\left( F_1 +  c_1 \sum_{t=1}^T\bar\gamma_t\|\S_t\|_{\fro}^2 + c_2 \sum_{t=1}^T\bar\gamma_t^2\|\V_t\|_{\fro}^2\right).
\eeq
\noi Using the definition of $\WWW_T:=\sum_{t=1}^T\bar\gamma_t\|\V_t\|_{\fro}^2$, and taking the expectation yields
\beq \label{eq:bound:WT1}
\EEE[\WWW_T] &\le & \ts \tfrac{1}{c_0}\left( F_1 + c_1\EEE\left[\sum_{t=1}^T\bar\gamma_t\|\S_t\|_{\fro}^2\right] + c_2\EEE\left[\sum_{t=1}^T\bar\gamma_t^2\|\V_t\|_{\fro}^2\right]\right) \nn\\
&\xxxx{1}{\leq} & \underbrace{\ts \tfrac{1}{c_0} (F_1 + c_1 C_c + c_2 C_d)}_{:=C_w} \cdot \ln (e+T),
\eeq
\noi where step \step{1} uses $\EEE [\sum_{t=1}^T \bar\gamma_t\|\S_t\|_{\fro}^2]\le C_c \ln(e+T)$, and $\ts \sum_{t=1}^T\bar{\gamma}_t^2\|\M_t\|_{\fro}^2\le C_d \ln(e+T)$ a.s..

\paragraph{Bounding the term $\EEE[\WWW_T^2]$.} Squaring the pathwise inequality in (\ref{eq:pathwise:ineq}) yields:
\beq
\WWW_T^2 &\leq& \ts \tfrac{3}{c_0^2} \left( F_1^2 + c_1^2 (\sum_{t=1}^T\bar\gamma_t\|\S_t\|_{\fro}^2)^2 + c_2^2 (\sum_{t=1}^T\bar\gamma_t^2\|\V_t\|_{\fro}^2)^2 \right) \nn\\
& \xxxx{1}{\leq} &  \ts \tfrac{3}{c_0^2} \left( F_1^2 + c_1^2  (\sum_{t=1}^T\bar\gamma_t\|\S_t\|_{\fro}^2)^2 + c_2^2 C_d^2 \ln^2(e+T) \right), \nn
\eeq
\noi where step \step{1} uses Lemma \ref{lemma:bcd:MuonA} that $\ts \sum_{t=1}^T \bar{\gamma}_t^2\|\V_t\|_{\fro}^2\le C_d \ln(e+T)$ holds almost surely. Taking expectations on both sides and using Lemma \ref{lemma:U:square:MuonA} that $\EEE[\UUU_T^2] \leq C_u \ln^2(e+T)$ with $\UUU_T:=\sum_{t=1}^T \bar\gamma_t\|\S_t\|_{\fro}^2$, we have:
\beq \label{eq:bound:WT2}
\EEE[\WWW_T^2] &\leq&  \ts \tfrac{3}{c_0^2} \left( F_1^2 + c_1^2 C_u \ln^2(e+T) + c_2^2 C_d^2 \ln^2(e+T) \right) \nn\\
&= &  \ts \underbrace{\ts \tfrac{3}{c_0^2} \left( F_1^2 + c_1^2 C_u + c_2^2 C_d^2 \right) }_{C_x} \cdot \ln^2(e+T).
\eeq

\paragraph{Bounding the term $\VVV_T$.} We derive:
\beq
\VVV_T & \xxxx{1}{\leq} & \ts \varpi \big( \sum_{t=1}^T \bar{\gamma}_t \|\V_t\|_{\fro}^2\big) \cdot \tfrac{1}{\bar\gamma_T} ~~ \xxxx{2}{=} ~~ \varpi \WWW_T \cdot \tfrac{1}{\bar\gamma_T}\nn\\
&\xxxx{3}{=} & \ts \varpi \WWW_T \cdot  \max\left(\frac1{\alpha_T}, \left(1+\sum_{j=1}^T\frac{\|\V_j\|_{\fro}^2}{\alpha_j}\right)^{1/2}  \right)  \nn\\
&\leq& \ts \varpi \WWW_T \cdot \left\{ \frac1{\alpha_T} + \left(1+\sum_{j=1}^T\frac{\|\V_j\|_{\fro}^2}{\alpha_j}\right)^{1/2} \right\} \nn\\
&\xxxx{4}{\leq}& \ts \varpi^2 \WWW_T \cdot \left\{ \frac1{\alpha_T} +  \left(1+ \tfrac{1}{\alpha_T} \VVV_T \right)^{1/2} \right\},  \label{eq:VVVT}
\eeq
\noi where step \step{1} uses $\bar\gamma_T\leq \varpi \bar \gamma_t$ for all $t\leq T$; step \step{2} uses the definition of $\WWW_T$; step \step{3} uses the definition of $\bar\gamma_t = \min \big\{\alpha_t,\big(1+\sum_{j=1}^t\frac{\|\V_j\|_{\fro}^2}{\alpha_j}\big)^{-1/2}\big\}$; step \step{4} uses $\alpha_T\leq\alpha_t \varpi$ for all $t\leq T$.

\paragraph{Bounding the term $\EEE[\alpha_T \VVV_T]$.} Let $Y_T:=\alpha_T \VVV_T$. Multiplying both sides of Inequality (\ref{eq:VVVT}) by $\alpha_T$ yields:
\beq
Y_T &\le& \ts \varpi^{2} \WWW_T \left( 1 + \sqrt{ \alpha_T^2 + Y_T}\right) ~~\xxxx{1}{\leq} ~~ 2\varpi^{2} \WWW_T \max\left( 2, \sqrt{ Y_T} \right) \nn\\
&\leq& \ts \max(4\varpi^2 \WWW_T, 4\varpi^4 \WWW_T^2 ) ~~  \leq ~~  4\varpi^4 (\WWW_T + \WWW_T^2),\nn
\eeq
\noi where step \step{1} uses $\sqrt{a+b}\leq \sqrt{a}+\sqrt{b}$ and $a+b\leq 2 \max(a,b)$ for all $a,b\geq 0$.

Taking expectations gives
\beq
\EEE[\alpha_T \VVV_T] &\le& 4\varpi^4 \EEE[ \WWW_T+\WWW_T^2] \nn\\
&\xxxx{1}{\leq}& 4\varpi^4 C_w \ln (e+T) + 4\varpi^4 C_x \ln^2(e+T) \nn\\
&\leq & 4\varpi^4 ( C_w + C_x) \ln^2(e+T), \nn
\eeq
\noi where step \step{1} uses Inequalities (\ref{eq:bound:WT1}) and (\ref{eq:bound:WT2}).

\end{proof}

\subsection{Proof of Theorem \ref{the:it:MuonA}}
\label{app:the:it:MuonA}

\begin{proof}
Define $\VVV_T:=\sum_{t=1}^T\|\V_t\|_{\fro}^2$, $\SSS_T:=\sum_{t=1}^T\|\S_t\|_{\fro}^2$, and $\GGG_T:=\sum_{t=1}^T\|\G_t\|_{\fro}^2$.

\noi Let $\epsilons_t = \G_t-\nabla f(\X_t)$, $\S_t=\V_t-\nabla f(\X_t)$, and $\ZZ_T:=\EEE\!\left[\sqrt{1+\GGG_T}\right]$.

\paragraph{Bounding the term $\EEE [\frac1{\alpha_T}]$.} We derive:
\beq
\ts \frac1{\alpha_T} ~~=~~ \left(\frac{1+\sum_{i=1}^{T-1} g_i^2}{1+\max_{1\le i\le {T-1}}g_i^2}
\right)^{1/2} ~~\le~~ \sqrt{1+\sum_{i=1}^T g_i^2} ~~=~~\sqrt{1+\GGG_T}.
\eeq
\noi This further leads to $\EEE [\frac1{\alpha_T}]\le \ZZ_T$.

\paragraph{Bounding the Scaled Momentum Sum.} We derive:
\beq \label{eq:bound:sqrt:V}
\ts \EEE[\sqrt{\VVV_T}] ~~\xxxx{1}{\leq}~~ \sqrt{\EEE[\alpha_T\VVV_T]}\,
\sqrt{\EEE\!\left[\frac1{\alpha_T}\right]} ~~\xxxx{2}{\leq}~~ \sqrt{ C_v \ln^2(e+T) \ZZ_T},
\eeq
\noi where step \step{1} uses $\EEE[\sqrt{XY}]\leq \sqrt{\EEE[X]\EEE[Y]}$; step \step{2} uses $\EEE [\frac1{\alpha_T}]\le \ZZ_T$, and Lemma \ref{R_R} that $\EEE[\alpha_T \VVV_T]\le C_v \ln^2(e+T)$.

\paragraph{Bounding the Tracking-Error Sum.} Since $\alpha_T\leq \alpha_t \varpi$, leading to $\SSS_T =\sum_{t=1}^T\|\S_t\|_{\fro}^2 \le \frac{\varpi}{\alpha_T}
\sum_{t=1}^T \alpha_t\|\S_t\|_{\fro}^2$. Taking square roots and using Cauchy--Schwarz gives
\beq \label{eq:bound:sqrt:S}
\ts \EEE[\sqrt{\SSS_T}] ~~\le ~~ \ts \sqrt{\varpi} \sqrt{\EEE [\frac1{\alpha_T}]\cdot \EEE[
\sum_{t=1}^T \alpha_t\|\S_t\|_{\fro}^2]} ~~\le~~ \ts \sqrt{\varpi} \sqrt{C_c \ln(e+T)\ZZ_T},
\eeq
\noi where step \step{1} uses $\EEE[\sqrt{XY}]\leq \sqrt{\EEE[X]\EEE[Y]}$; step \step{2} uses $\EEE [\frac1{\alpha_T}]\le \ZZ_T$, and Lemma \ref{lemma:bcd:MuonA} that $\EEE[ \sum_{t=1}^T \alpha_t \|\S_t\|_{\fro}^2]\le C_c \ln(e+T)$.

\paragraph{Bounding the terms $\EEE[\sqrt{ \sum_{t=1}^T\|\nabla f(\X_t)\|_{\fro}^2}]$ and $\ZZ_T$.} We derive:
\beq
\ts \EEE[\sqrt{ \sum_{t=1}^T\|\nabla f(\X_t)\|_{\fro}^2}] &\xxxx{1}{\leq}& \ts \EEE[\sqrt{2}\sqrt{\VVV_T} + \sqrt{2}\sqrt{\SSS_T}] \nn\\
&\xxxx{2}{\leq}&  \ts \sqrt{ 2 C_v \ln^2(e+T) \ZZ_T} + \sqrt{2 \varpi C_c \ln(e+T)\ZZ_T} \nn\\
&\xxxx{}{\leq}&  \ts \underbrace{ \ts \sqrt{ 2 C_v} + \sqrt{2 \varpi C_c}  }_{ \ts :=C_z} \sqrt{\ZZ_T} \ln (e+T), \label{eq:ggg}
\eeq
\noi where step \step{1} uses $\|\a+\b\|^2\leq 2\|\a\|^2+2\|\b\|^2$; step \step{2} uses Inequalities (\ref{eq:bound:sqrt:V}) and (\ref{eq:bound:sqrt:S}).

Additionally, we have
\beq \label{eq:Z:split}
\ZZ_T &= & \ts \EEE [\sqrt{1+\sum_{t=1}^T\|\G_t\|_{\fro}^2} ] \notag\\
&\xxxx{1}{\leq}& \ts 1 + \sqrt{2} \EEE [\sqrt{\sum_{t=1}^T\|\nabla f(\X_t)\|_{\fro}^2}] + \sqrt{2} \EEE [
\sqrt{\sum_{t=1}^T\|\epsilons_t\|_{\fro}^2}] \nn\\
&\xxxx{2}{\leq}& \ts 1 + \sqrt{2} C_z \sqrt{\ZZ_T} \ln (e+T)  + \sqrt{2}\sigma \sqrt{T} \nn\\
&\xxxx{3}{\leq}& \ts 2\max(\sqrt{2} C_z \sqrt{\ZZ_T} \ln (e+T) ,  1 + \sqrt{2}\sigma \sqrt{T} ) \nn\\
&\le& \ts \max\big( 4(\sqrt{2} C_z  \ln (e+T) )^2,  2 + 2\sqrt{2}\sigma \sqrt{T}\big) \nn\\
&\le& \ts (2 + 8 C_z^2) \ln^2(e+T)  + 2\sqrt{2}\sigma \sqrt{T}, \label{eq:MuonA:zzz}
\eeq
\noi where step \step{1} uses $\|\a+\b\|^2\leq 2\|\a\|^2+2\|\b\|^2$; step \step{2} uses Inequality (\ref{eq:ggg}); step \step{3} uses $a+b\leq2\max(a,b)$, and the corrected self-consistency step uses $Z\le2\max(A\sqrt Z,B)\Rightarrow Z\le\max(4A^2,2B)$.

\paragraph{Bounding the term $\EEE [\frac{1}{T}\sum_{t=1}^T \|\nabla f(\X_t)\|_{\fro} ]$.} By Cauchy--Schwarz,
\beq
\ts \EEE \left[\frac1T\sum_{t=1}^T\|\nabla f(\X_t)\|_{\fro}\right] &\le& \ts \frac1{\sqrt T} \EEE\left[\sqrt{\sum_{t=1}^T\|\nabla f(\X_t)\|_{\fro}^2}
\right] \nn\\
&\xxxx{1}{\leq}& \ts \frac1{\sqrt T} C_z \sqrt{\ZZ_T} \cdot \ln(e+T) \nn\\
&\xxxx{2}{\leq}& \ts \frac1{\sqrt T} C_z \sqrt{(2 + 8 C_z^2) \ln^2(e+T) + 2\sqrt{2}\sigma \sqrt{T}} \cdot \ln(e+T) \nn\\
& \leq & \ts C_z\sqrt{2+8C_z^2}\frac{\ln^2(e+T)}{T^{1/2}}
+2^{3/4}C_z\frac{\sigma^{1/2}\ln(e+T)}{T^{1/4}}. \nn
\eeq
\noi where step \step{1} uses Inequality (\ref{eq:ggg}); and step \step{2} uses Inequality (\ref{eq:MuonA:zzz}).

\noi If \(R\sim {\rm Unif}\{1,\ldots,T\}\) is independent of the algorithmic randomness, then
\[
\EEE\|\nabla f(\X_R)\|_{\fro}
=
\EEE\left[\frac1T\sum_{t=1}^T\|\nabla f(\X_t)\|_{\fro}\right],
\]
which proves the output guarantee.

\end{proof}

\section{Proofs for Section \ref{sect:it:LL}}
\label{app:sect:it:LL}

\subsection{Proof of Lemma \ref{lemma:prop:s:2}}
\label{app:lemma:prop:s:2}

\begin{proof}

By Assumption \ref{ass:BG},
\beq
\ts \|\epsilons_t\|_{\fro}\le 2\hG,\qquad
    \|\epsilons'_t\|_{\fro}\le 2\hG
    \quad\text{a.s.}.
\eeq

Recall that \(\alpha_t=\rho_{t-1}\), where
\[
\rho_t=\left(\frac{1+\NNN_t}{1+\GGG_t}\right)^{2/3},
\qquad
\NNN_t:=\max_{1\le i\le t}g_i^2,\quad
\GGG_t:=\sum_{i=1}^t g_i^2 .
\]
For \(1\le k\le t\le T\), Lemma \ref{lemma:prop:para} gives
\[
    \frac{\alpha_t}{\alpha_k}
    =
    \frac{\rho_{t-1}}{\rho_{k-1}}
    \le
    (1+\NNN_{t-1})^{2/3}
    \le (1+\hG^2)^{2/3}
    =\varpi.
\]
Let
\[
    \beta_t:=\left(1+\sum_{j=1}^t\alpha_j^{-1/2}\|\M_j\|_{\fro}^2\right)^{-1/2}.
\]
Then \(\beta_t\) is nonincreasing and
\(\gamma_t=\min\{\sqrt{\alpha_t},\beta_t\}\). Hence
\[
\frac{\gamma_t}{\gamma_k}\le \max\left\{\sqrt{\frac{\alpha_t}{\alpha_k}},\frac{\beta_t}{\beta_k}\right\} \le \sqrt{\varpi}.
\]
The consecutive ratio follows from Lemma \ref{lemma:prop:para}:
\[
    \frac{\alpha_{t-1}}{\alpha_t}
    =
    \frac{\rho_{t-2}}{\rho_{t-1}}
    \le 2^{2/3}\le 2 .
\]
Finally, applying Lemma \ref{lemma:prop:para} with \(q=2/3\) gives
\[
    (1-\alpha_t)\sqrt{\frac{\alpha_{t-1}}{\alpha_t}}
    \le
    1-(1-q)\alpha_t
    =
    1-\frac{\alpha_t}{3}.
\]
\end{proof}

\subsection{Proof of Lemma \ref{lemma:OptMuonI:sumV:byG}}
\label{app:lemma:OptMuonI:sumV:byG}

\begin{proof} Define $\MMM := \sum_{i=1}^T m_i^2$, and $\GGG^+ := 1 + \sum_{i=1}^T g_i^2$.

\noi Define $\Z_t:= \tfrac{1-\alpha_t}{\alpha_t}\left(\nabla f(\X_t;\xi_t)  - \nabla f(\X_{t-1};\xi_t)\right)  +  \nabla f(\X_t;\xi_t)$, $z_t:= \|\Z_t\|_{\fro}$, and $\delta:=nL^2\theta^2$.

\paragraph{Bounding $\sum_{t=1}^T m_t^2$ using $\sum_{t=1}^T z_t^2$.} The update rule for $\M_t$ can be rewritten as $\M_t = (1-\alpha_t) \M_{t-1} + \alpha_t \Z_t$, and applying the convexity of $\|\cdot \|$ yields:
\beq
\ts m_t^2 \leq (1-\alpha_t) m_{t-1}^2 + \alpha_t z_t^2 &\overset{}{\Rightarrow}& \ts \frac{m_t^2}{\alpha_t} \leq \frac{m_{t-1}^2}{\alpha_t} - m_{t-1}^2 + z_t^2\nn\\
&\overset{}{\Rightarrow}& \ts m_{t}^2 \leq (\frac{1}{\alpha_t}-1) (m_{t-1}^2 - m_{t}^2) + z_t^2. \label{eq:bound:m:1}
\eeq
\noi Summing Inequality (\ref{eq:bound:m:1}) over $t = 1,2,\ldots,T$ yields:
\beq \label{eq:sum:m}
\ts \sum_{t=1}^T m_t^2 &\overset{}{\leq}& \ts \sum_{t=1}^T z_t^2 + \sum_{t=1}^T (\frac{1}{\alpha_t}-1) (m_{t-1}^2 - m_{t}^2) \nn\\
&\overset{\step{1}}{=}& \ts \sum_{t=1}^T z_t^2 + \left(\tfrac{1}{\alpha_1} - 1\right) m_0^2 - \left(\tfrac{1}{\alpha_{T+1}} - 1 \right) m_T^2 + \sum_{t=1}^T (\tfrac{1}{\alpha_{t+1}} - \tfrac{1}{\alpha_t}  ) m_{t}^2  \nn\\
&\overset{\step{2}}{\leq}& \ts \sum_{t=1}^T z_t^2 + q \sum_{t=1}^T m_{t}^2~~\overset{}{\leq}~~ \ts \frac{1}{1-q}\sum_{t=1}^T z_t^2 ~~=~~ 3 \sum_{t=1}^T z_t^2, \nn
\eeq
\noi where step \step{1} uses Lemma \ref{lemma:diff}; step \step{2} uses $\left(\tfrac{1}{\alpha_1} - 1\right) m_0^2=0$, $-\left(\tfrac{1}{\alpha_{T+1}} - 1 \right) m_T^2\leq 0$, and the fact that $\tfrac{1}{\alpha_{t+1}}-\tfrac{1}{\alpha_{t}}\leq q$.

\paragraph{Bounding the term $\sum_{t=1}^T z_t^2$.} We derive the following results:
\beq
\ts \sum_{t=1}^T z_t^2 &\overset{}{=}& \ts \sum_{t=1}^T \| \tfrac{1-\alpha_t}{\alpha_t}\left(\nabla f(\X_t;\xi_t)  - \nabla f(\X_{t-1};\xi_t)\right)  +  \nabla f(\X_t;\xi_t)\|_{\fro}^2  \nn\\
&\overset{\step{1}}{\leq}& \ts \bigl(  2 \sum_{t=1}^T g_t^2\bigr) +  \bigl( 2L^2 \sum_{t=1}^T  \frac{1}{\alpha_t^2} \| \X_t -\X_{t-1} \|_{\fro}^2  \bigr)  \nn\\
&\overset{\step{2}}{=}& \ts \bigl(  2\sum_{t=1}^T g_t^2 \bigr) + \bigl( 2L^2 \sum_{t=1}^{T-1}  \frac{1}{\alpha_{t+1}^2} \| \X_{t+1} -\X_{t} \|_{\fro}^2  \bigr) \nn\\
&\overset{\step{3}}{\leq}& \ts \bigl(  2\sum_{t=1}^T g_t^2 \bigr) + 6  n L^2\theta^2    \sum_{t=1}^T  \frac{\gamma_{t}^2 m_{t}^2}{\alpha_{t}^2} =  \bigl(  2\sum_{t=1}^T g_t^2 \bigr) + 6 \delta \sum_{t=1}^T  \frac{\gamma_{t}^2 m_{t}^2}{\alpha_{t}^2}   , \label{eq:bound:sum:z}
\eeq
\noi where step \step{1} uses $\|\nabla f(\X_t;\xi_t)  - \nabla f(\X_{t-1};\xi_t) \|_{\fro}\leq L\|\X_t-\X_{t-1}\|_{\fro}$; step \step{2} uses $\|\X_{1}-\X_{0}\|_{\fro}=0$; step \step{3} uses $\|\X_{t+1} - \X_{t}\|_{\fro}^2  \leq n \eta_{t}^2 = n \theta^2 \gamma_{t}^2 m_{t}^2$ and $\tfrac{\alpha_t^2}{\alpha_{t+1}^2} \leq (1+q)^2 = (1 + 2/3)^2<3$.

\paragraph{Bounding the term $\sum_{t=1}^T  \frac{\gamma_{t}^2 m_{t}^2}{\alpha_{t}^2}$.} We further derive:
\beq \label{eq:MGMG}
\ts \sum_{t=1}^T \frac{\gamma_{t}^2 m_{t}^2}{\alpha_{t}^2}     &\overset{\step{1}}{\leq}& \ts \sum_{t=1}^{T}   \frac{  m_t^2 / \alpha_{t}^2 }{  1 + \sum_{i=1}^t m_i^2 / \sqrt{\alpha_i} } ~~\overset{\step{2}}{\leq}~~ \ts  \frac{\varpi^{3/2}}{\alpha_{T}^{3/2}} \cdot   \ln\left(1 + \sum_{t=1}^T \frac{m_t^2}{\sqrt{\alpha_t}}  \right)   \nn\\
&\overset{\step{3}}{\leq}&  \ts \varpi^{3/2} \GGG^+ \cdot  \ln \left( 1 + \MMM \cdot(\GGG^+)^{1/3}   \right),
\eeq
\noi where step \step{1} uses $\gamma_t^2 \leq \tfrac{1}{1 + \sum_{i=1}^t m_i^2/\sqrt{\alpha_i}}$; step \step{2} uses Lemma \ref{lemma:adaptive:important}(a), and $\alpha_t^{-3/2} \leq\alpha_T^{-3/2} \varpi^{3/2}$ (which can be implied by $\tfrac{\alpha_T}{\alpha_t}\leq \varpi$); step \step{3} uses
\beq
\ts \frac{1}{\alpha_{T}^{3/2}} ~~=~~ \frac{1 + \sum_{i=1}^{T-1} g_i^2}{ 1 + \max_{i=1}^{T-1} g_i^2} ~~\leq ~~ 1 + \sum_{i=1}^T g_i^2 ~~=~~ \GGG_T^+.\nn
\eeq

\noi We have from Inequalities (\ref{eq:sum:m}), (\ref{eq:bound:sum:z}), and (\ref{eq:MGMG}):
\beq
\ts \MMM &:=& \ts \sum_{t=1}^T m_t^2 ~~\leq~~  6 \GGG + 18  \delta  \cdot \varpi^{3/2} \GGG^+ \cdot \ln(1 + \MMM (\GGG^+)^{1/3} )  \label{eq:hhh}\\
&\overset{\step{1}}{\leq}& \ts  \max(12 \GGG, 36   \delta \cdot \varpi^{3/2} \GGG^+ \cdot \ln(1 + \MMM (\GGG^+)^{1/3}) ) \nn\\
&\overset{\step{2}}{\leq}& \ts  \max\left(12 \GGG, 72   \delta \cdot \varpi^{3/2} \GGG^+ \cdot \ln(e +  36   \delta  \varpi^{3/2}\cdot (\GGG^+)^{4/3}) \right) \nn\\
&\overset{\step{3}}{\leq}& \ts  \max\left(12 \GGG, 72   \delta \cdot \varpi^{3/2} \GGG^+ \cdot \tfrac{4}{3}\ln(e +  36  \delta \varpi^{3/2} (1+T\hG^2) ) \right) \nn\\
&\overset{}{\leq}& \ts \underbrace{\ts  2\max\big(12 , 96 \delta \varpi^{3/2} \ln(e +  36  \delta \varpi^{3/2} (\hG^2+1) ) \big) }_{:=\kappa}\cdot \ln(e+T) \GGG^+ \nn
\eeq
\noi where step \step{1} uses $a+b\leq 2 \max(a,b)$ for all $a,b\geq 0$; step \step{2} uses Inequality (\ref{eq:hhh}) and Lemma \ref{lemma:yc}; step \step{3} uses $\ln(1+ab)\leq \ln(e+a)\ln(e+b)$ for all $a,b\geq 0$.






\end{proof}

\subsection{Proof of Lemma \ref{lemma:VR:bound:I}}
\label{app:lemma:VR:bound:I}

\begin{proof} Define $\mathbf S_t := \M_t - \nabla f(\mathbf X_t)$, and $\epsilons_t := \G_t - \nabla f(\mathbf X_t)$.

\noi Define $m_t:=\|\M_t\|_{\fro}$, and $g_t:=\|\G_t\|_{\fro}$.

\noi Define $\G_t:=\nabla f(\X_{t};\xi_t)$, and $\G'_t:=\nabla f(\X_{t-1};\xi_t)$.

\noi Define $\Delta_t := (\G_t-\G'_t)-(\nabla f (\X_t)- \nabla f (\X_{t-1}))$.

\noi Define $\epsilons_t:= \nabla f(\X_{t};\xi_t) - \nabla f(\X_t)$ and $\epsilons'_t:= \nabla f(\X_{t-1};\xi_t) - \nabla f(\X_{t-1})$.

\paragraph{Difference-based representation of $\S_t$.} We derive the following equalities:
\beq \label{eq:eez}
\S_t &:= & \M_t - \nabla f(\X_t) \nn\\
&\xxxx{1}{=}&  (1-\alpha_t) (\M_{t-1}  -  \nabla f(\X_{t-1};\xi_t) ) +  \G_t  - \nabla f(\X_t) \nn\\
&\xxxx{}{=}& (1-\alpha_t) (\M_{t-1} - \nabla f(\X_{t-1})) + (1-\alpha_t) ( \nabla f(\X_{t-1})   -   \nabla f(\X_{t-1};\xi_t) ) +  \epsilons_t   \nn\\
&\xxxx{2}{=}& \ts (1-\alpha_t)\S_{t-1}    + \ts (\G_t - \nabla f(\X_t))  - (1-\alpha_t) (\G'_t -  \nabla f(\X_{t-1}) ) \nn\\
&\xxxx{2}{=}&  \underbrace{\ts (1-\alpha_t)\S_{t-1}   }_{:=\A_t} + \underbrace{\ts \epsilons_t   - (1-\alpha_t) \epsilons'_t }_{:=\R_t},
\eeq
\noi where step \step{1} uses the update rule of OptMuon-I that $\ts\M_{t} = (1-\alpha_{t}) (\M_{t-1}  - \nabla f(\X_{t-1};\xi_t) )+\G_t$; step \step{2} uses $\S_t := \M_t - \nabla f(\X_t)$.

\paragraph{Bounding the term $\|\R_t\|_{\fro}^2$.} We have:
\beq \label{eq:z:exp}
\|\R_t\|_{\fro}^2 &\xxxx{1}{=}&  \| \epsilons_t - (1-\alpha_t) \left( \G'_t   - \nabla f(\X_{t-1})  \right)  \|_{\fro}^2 \nn\\
&\xxxx{}{=}& \| \alpha_t \epsilons_t + (1-\alpha_t) \left(  \G_t -  \nabla f(\X_{t}) +  \nabla f(\X_{t-1})   -   \nabla f(\X_{t-1};\xi_t)  \right) \|_{\fro}^2 \nn\\
&\xxxx{2}{\leq}& 2 \alpha_t^2 \|\epsilons_t\|_{\fro}^2 + 2 (1-\alpha_t)^2 \| \G_t -  \nabla f(\X_t) + \nabla f(\X_{t-1})   -  \nabla f(\X_{t-1};\xi_t) \|_{\fro}^2 \nn\\
&\xxxx{}{\leq}& 2 \alpha_t^2 \|\epsilons_t\|_{\fro}^2  + 4 \| \G_t -  \nabla f(\X_{t-1};\xi_t) \|_{\fro}^2  + 4 \|\nabla f(\X_{t}) - \nabla f(\X_{t-1})    \|_{\fro}^2 \nn\\
&\xxxx{3}{\leq}& 2 \alpha_t^2 \|\epsilons_t\|_{\fro}^2 + (4+4) L^2 \| \X_{t} -  \X_{t-1} \|_{\fro}^2 \nn\\
&\xxxx{4}{\leq}& 2 \alpha_t^2 \|\epsilons_t\|_{\fro}^2 + 8 n L^2\theta^2 \gamma_{t-1}^2m_{t-1}^2,
\eeq
\noi where step \step{1} uses the definition of $\R_t$; step \step{2} uses $\|\A+\B\|_{\fro}^2\leq 2\|\A\|_{\fro}^2+2\|\B\|_{\fro}^2$ for matrices $\A,\B$; step \step{3} uses $\|\nabla f(\X_{t-1};\xi_t) - \nabla f(\X_{t};\xi_t)\| \leq L\|\X_{t}-\X_{t-1}\|_{\fro}$, and $\|\nabla f(\X_{t-1}) - \nabla f(\X_{t})\|\leq L\|\X_{t}-\X_{t-1}\|_{\fro}$; step \step{4} uses $\X_{t+1}=\X_{t}-\theta\gamma_{t}m_{t}\Orth(\M_{t})$, and $\|\Orth(\M_{t-1})\|_{\fro}^2\leq n$.

\paragraph{Bounding the term $\|\S_t\|_{\fro}^2$.} Given $\S_t=\A_t+\R_t$, we derive:
\beq
\| \S_t\|_{\fro}^2 & = &  \ts \|\A_t\|_{\fro}^2 + \|\R_t\|_{\fro}^2 + 2 \la \A_t,\R_t \ra \nn\\
& \xxxx{1}{\leq} & \ts (1-\alpha_t) \|\S_{t-1}\|_{\fro}^2 + 2 \alpha_t^2 \|\epsilons_t\|_{\fro}^2 + 8 n L^2\theta^2 \gamma_{t-1}^2m_{t-1}^2 \nn\\
&& \ts + 2 (1-\alpha_t) \la \S_{t-1},\epsilons_t   - (1-\alpha_t) \epsilons'_t \ra, \nn
\eeq
\noi where step \step{1} uses $\alpha_t\in(0,1]$, leading to $(1-\alpha_t)^2\leq 1-\alpha_t$.
\end{proof}

 \subsection{Proof of Lemma \ref{lemma:suff:descent:I}}
\label{app:lemma:suff:descent:I}

\begin{proof} Let $\mathbf{X}_{t+1} = \mathbf{X}_t - \tilde{\gamma}_t \operatorname{Orth}(\mathbf{M}_t)$, where $\tilde{\gamma}_t:=\theta \gamma_t \|\M_t\|_{\fro}$.

For any $t\leq T$, we have the following inequalities:
\beq
&& \ts f(\X_{t+1}) - f(\X_{t}) \nn\\
&\overset{\step{1}}{\leq}& \ts \tfrac{L}{2}\|\X_{t+1}-\X_t\|_{\fro}^2 + \la \X_{t+1}-\X_t,\nabla f(\X_t) \ra  \nn\\
&\overset{\step{2}}{=}& \ts \tfrac{L}{2}\|-\tilde{\gamma}_t\Orth(\M_t)\|_{\fro}^2 - \la \tilde{\gamma}_t\Orth(\M_t),\M_t \ra + \la -\tilde{\gamma}_t\Orth(\M_t),\nabla f(\X_t) - \M_t \ra \nn\\
&\overset{\step{3}}{\leq}& \ts \tfrac{Ln}{2} \tilde{\gamma}_t^2  - \tilde{\gamma}_t \|\M_t\|_{\fro} + \tilde{\gamma}_t \|\nabla f(\X_t) - \M_t \|_*     \nn\\
&\overset{\step{4}}{\leq}& \ts \tfrac{L n}{2} \theta^2 \gamma_t^2 \|\M_t\|_{\fro}^2  + \theta \gamma_t \big( -  \|\M_t\|_{\fro}^2 +  \|\M_t\|_{\fro} \|\nabla f(\X_t) - \M_t \|_* \big) \nn\\
&\overset{\step{5}}{\leq}& \ts \tfrac{L n}{2} \theta^2 \gamma_t^2 \|\M_t\|_{\fro}^2  + \theta \gamma_t \big( -  \tfrac{1}{2}\|\M_t\|_{\fro}^2 +  \tfrac{1}{2}\|\nabla f(\X_t) - \M_t \|_*^2 \big) \nn\\
&\overset{\step{6}}{\leq}& \ts \tfrac{L n}{2} \theta^2 \gamma_t^2 \|\M_t\|_{\fro}^2  + \theta \gamma_t \big( -  \tfrac{1}{2}\|\M_t\|_{\fro}^2 +  \tfrac{n}{2}\|\nabla f(\X_t) - \M_t \|_{\fro}^2 \big), \nn
\eeq
\noi where step \step{1} uses the $L$-smoothness of $f(\cdot)$; step \step{2} uses $\X_{t+1}-\X_t=-\tilde{\gamma}_t\Orth(\M_t)$; step \step{3} uses $\la \X,\Y \ra\leq \|\X\| \|\Y\|_{*}$, and the fact that:
\beq
\la \Orth(\M_t),\M_t \ra = \|\M_t\|_* \geq \|\M_t\|_{\fro};\nn
\eeq
\noi step \step{4} uses the definition of $\tilde{\gamma}_t$; step \step{5} uses $-a^2+ab\le -\tfrac12a^2+\tfrac12b^2$ with $a=\|\M_t\|_{\fro}$ and $b=\|\nabla f(\X_t)-\M_t\|_*$, and step \step{6} uses $\|\nabla f(\X_t)-\M_t\|_*^2\le n\|\nabla f(\X_t)-\M_t\|_{\fro}^2$.

\end{proof}

\subsection{Proof of Lemma \ref{lemma:bcd:MuonI}}
\label{app:lemma:bcd:MuonI}

\begin{proof} Let $s_t:=\|\S_t\|_{\fro}$, $m_t:=\|\M_t\|_{\fro}$, $R_t:=\sum_{i=1}^t g_i^2$, $H_t:=1+\sum_{j=1}^t \alpha_j^{-1/2}m_j^2$.

\noi We set $\gamma_0=0$, $m_0=0$, and $\M_0=\zero$.

\noi Define $\S_t:=\M_t-\nabla f(\X_t)$, $\epsilons_t:= \nabla f(\X_{t};\xi_t) - \nabla f(\X_t)$, $\epsilons'_t:= \nabla f(\X_{t-1};\xi_t) - \nabla f(\X_{t-1})$.

\noi Recall that for OptMuon-I, $q=2/3$ and $\gamma_t=\min\{\sqrt{\alpha_t},H_t^{-1/2}\}$.

\noi Let $B_t:= 2 \alpha_t^2 \|\epsilons_t\|_{\fro}^2 + 8 \delta  \gamma_{t-1}^2m_{t-1}^2 + 2 (1-\alpha_t) \la \S_{t-1},\epsilons_t   - (1-\alpha_t) \epsilons'_t\ra$, where $B_t$ can be negative.

\noi Define $\gamma_t=\min\left\{\sqrt{\alpha_t},\left(1+\sum_{j=1}^t\alpha_j^{-1/2}\|\M_j\|_{\fro}^2\right)^{-1/2}\right\}$. Let $\GGG^+ := 1 + \sum_{i=1}^T g_i^2$.

\paragraph{Bounding the Term $\alpha_t^{3/2}$.} By the definition of \(\alpha_t\), for all \(t\ge2\),
\beq
\alpha_t^{3/2}\le \frac{1+\max_{1\le i\le t-1}g_i^2}{1+\sum_{i=1}^{t-1}g_i^2}\le\frac{1+\hG^2}{1+R_{t-1}}.
\eeq
\noi The inequality above also holds for the case $t=1$.

\paragraph{Bounding the Term $\EEE[\sum_{t=1}^T\alpha_{t}^{3/2}\|\epsilons_t\|_{\fro}^2]$.} We derive:
\beq \label{eq:alpha:32:epsilon}
&& \ts \EEE[\sum_{t=1}^T\alpha_{t}^{3/2}\|\epsilons_t\|_{\fro}^2] ~~\xxxx{1}{\leq}~~ 4\EEE[\sum_{t=1}^T\alpha_{t+1}^{3/2}\|\epsilons_t\|_{\fro}^2] ~~\xxxx{2}{\leq}~~ \ts  \ts  4(1+\hG^2) \EEE[\sum_{t=1}^T  \tfrac{\|\epsilons_t\|_{\fro}^2}{1+\GGG_t} ]  \nn\\
&\xxxx{3}{\leq}& \ts 32 (1+\hG^2)^2 \EEE[\sum_{t=1}^T  \tfrac{\|\G_t\|_{\fro}^2}{1+\GGG_t} ]  ~~\xxxx{4}{\leq}~~\ts 32 (1+\hG^2)^2 \EEE[\ln(1 + \GGG_T) ] ,
\eeq
\noi where step \step{1} uses $\tfrac{\alpha_t}{\alpha_{t+1}} \leq 2$; step \step{2} uses $\alpha_{t+1}^{3/2} \leq \tfrac{1+\hG^2}{1+\GGG_t}$; step \step{3} uses Lemma \ref{lemma:weighted:self:normalized} with $W_t=1$; step \step{4} uses Lemma \ref{lemma:adaptive:important}(a).

\paragraph{Bounding the Term $\sum_{t=1}^T\gamma_{t-1}^2m_{t-1}^2/\sqrt{\alpha_t}$.} We derive:
\beq \label{eq:gamma:2:m:2:sqrt:alpha}
&& \ts \sum_{t=1}^T\frac{\gamma_{t-1}^2m_{t-1}^2}{\sqrt{\alpha_t}}
~~\xxxx{1}{\leq}~~
2\sum_{t=1}^T\frac{\gamma_t^2m_t^2}{\sqrt{\alpha_t}}
~~\xxxx{2}{\leq}~~
2\sum_{t=1}^T \frac{m_t^2/\sqrt{\alpha_t}}{1+\sum_{i=1}^t m_i^2/\sqrt{\alpha_i}} \nn\\
&\xxxx{} {\leq}&  \ts 2\ln(1 + \MMM_T /\sqrt{\alpha_T})
~~\xxxx{3}{\leq} ~~2\ln(1 + \kappa \ln (e+T) \GGG^+ /\sqrt{\alpha_T})
~~\xxxx{4}{\leq}~~ 2\ln(1 + \kappa \ln (e+T) (\GGG^+)^{4/3})  \nn\\
&\xxxx{5}{\leq}& 4\ln(1 + \kappa \ln (e+T) \GGG^+)
~~\leq~~ 4\ln(1 + \kappa \ln (e+T) (1 + \hG^2T))  \nn\\
&\xxxx{6}{\leq}& 4\ln(e+\kappa (1+\hG^2))\ln(1+\ln(e+T)T)
~~\xxxx{7}{\leq}~~ 8\ln(e+\kappa (1+\hG^2))\ln(e+T).
\eeq
\noi where step \step{1} uses $m_0=0$ and Lemma~\ref{lemma:prop:s:2}, which gives $\alpha_{t-1}/\alpha_t\le2$; step \step{2} uses the definition of $\gamma_t$; step \step{3} uses Lemma~\ref{lemma:OptMuonI:sumV:byG}; step \step{4} uses $\tfrac{1}{\sqrt{\alpha_T}}\leq \big(1+\sum_{t=1}^T g_t^2\big)^{1/3}=(\GGG^+)^{1/3}$; step \step{5} uses $\ln(a+b^{4/3})\leq \tfrac{4}{3}\ln(a+b)$ for all $a\geq 1$ and $b\geq 0$; step \step{6} uses $\ln(1+ab)\leq \ln(e+a)\ln(e+b)$; step \step{7} uses $\ln(1+\ln(e+T)T)\leq2\ln(e+T)$.

\paragraph{Bounding the Term $\EEE[\sum_{t=1}^T \tfrac{B_t}{\sqrt{\alpha_t}} ]$.} We derive:
\beq
\ts \EEE[\sum_{t=1}^T \tfrac{B_t}{\sqrt{\alpha_t}} ] & = & \ts \EEE[ \sum_{t=1}^T  2 \alpha_t^{3/2} \|\epsilons_t\|_{\fro}^2 + 8 \delta \tfrac{1}{\sqrt{\alpha_t}}  \gamma_{t-1}^2m_{t-1}^2 + 2 \tfrac{1-\alpha_t}{\sqrt{\alpha_t}} \la \S_{t-1},\epsilons_t   - (1-\alpha_t) \epsilons'_t\ra] \nn\\
&\xxxx{1}{\leq}& \ts \EEE[ \sum_{t=1}^T  2 \alpha_t^{3/2} \|\epsilons_t\|_{\fro}^2 + 16 \delta \tfrac{1}{\sqrt{\alpha_{t-1}}}   \gamma_{t-1}^2m_{t-1}^2 ] \nn\\
&\xxxx{2}{\leq}& \ts 64 (1+\hG^2)^2 \EEE[\ln(1 + \GGG_T) ] +  64 \delta \ln(e+\kappa + \kappa\hG^2) \ln (e+T) \nn\\
&\xxxx{}{\leq}& \ts \underbrace{\ts \big( 64(1+\hG^2)^2 + 64 \delta \big) \cdot \ln(e+\kappa + \kappa\hG^2) }_{:=C_b}\cdot \ln(e+T),  \nn
\eeq
\noi where step \step{1} uses the conditional-zero identity $\EEE_t[\langle\S_{t-1},\epsilons_t-(1-\alpha_t)\epsilons'_t\rangle]=0$ and $\tfrac{1}{\sqrt{\alpha_t}}\leq 2\tfrac{1}{\sqrt{\alpha_{t-1}}}$; step \step{2} uses Inequalities (\ref{eq:alpha:32:epsilon}) and (\ref{eq:gamma:2:m:2:sqrt:alpha}).

\paragraph{Bounding the Term $\sum_{t=1}^T \sqrt{\alpha_t} s_t^2$.} Let $P_t:=\frac{s_t^2}{\sqrt{\alpha_t}}$. By Lemma \ref{lemma:VR:bound:I}, we have $\ts s_t^2\le (1-\alpha_t)s_{t-1}^2+B_t$. Dividing both sides by $\sqrt{\alpha_t}$ gives
\beq
\ts P_t ~~\le~~ \ts (1-\alpha_t)\sqrt{\frac{\alpha_{t-1}}{\alpha_t}}P_{t-1} + \frac{B_t}{\sqrt{\alpha_t}} ~~\xxxx{1}{\leq}~~ \ts (1- \frac{\alpha_t}{3}) P_{t-1} + \frac{B_t}{\sqrt{\alpha_t}},  \label{eq:to:be:cited}
\eeq
\noi where step \step{1} uses Lemma~\ref{lemma:prop:s:2}. Equivalently, $\ts \frac13\alpha_tP_{t-1}\le P_{t-1}-P_t+\frac{B_t}{\sqrt{\alpha_t}}$. Summing over $t=1,\ldots,T$ yields:
\beq \label{eq:bound:alpha:P:sum}
\ts \sum_{t=1}^T \alpha_tP_{t-1} &\le& \ts 3P_0 + 3 \sum_{t=1}^T\frac{B_t}{\sqrt{\alpha_t}}.
\eeq
This further leads to the following inequalities:
\beq \label{eq:pathwise:alpha:sqrt:st2}
\ts \sum_{t=1}^T \sqrt{\alpha_t} s_t^2 &= & \ts  \sum_{t=2}^{T+1} \sqrt{\alpha_{t-1}} s_{t-1}^2 ~~ \xxxx{1}{=} ~~ \sum_{t=2}^{T+1} \alpha_{t-1} P_{t-1} \nn\\
&\xxxx{2}{\leq}& \ts 2 \sum_{t=2}^{T+1} \alpha_{t} P_{t-1} ~~\leq ~~ 6 P_0 + 6 \sum_{t=1}^{T+1}\tfrac{B_t}{\sqrt{\alpha_t}},
\eeq
\noi where step \step{1} uses the definition of $P_t$; step \step{2} uses $\frac{\alpha_{t-1}}{\alpha_t}\leq 1+q\leq 2$; step \step{3} uses Inequality (\ref{eq:bound:alpha:P:sum}). Taking expectations, we obtain
\beq \label{eq:pathwise:alpha:sqrt:st2:exp}
&& \ts \EEE\left[\sum_{t=1}^T \sqrt{\alpha_t} s_t^2 \right] ~~\xxxx{1}{\leq} ~~ 6 \hG^2 + 6 C_b \ln(e+T+1) ~~\xxxx{2}{\leq} ~~ \underbrace{\ts 12 (\hG^2 + C_b)}_{:=C_c}\cdot \ln (e+T),
\eeq
\noi where step \step{1} uses $\EEE[\sum_{t=1}^{T+1}\tfrac{B_t}{\sqrt{\alpha_t}}]\leq C_b\ln(e+T+1)$; step \step{2} uses $P_0=s_0^2=\|\nabla f(\X_0)\|_{\fro}^2\leq \hG^2$, $\ln(e+T)\ge1$, and $\ln(e+T+1)\le2\ln(e+T)$ for all $T\ge1$.

\paragraph{Part (c): Bounding the Term $\sum_{t=1}^T \gamma_t^2 m_t^2$.} We derive:
\beq
\ts \sum_{t=1}^T \gamma_t^2 m_t^2 & \xxxx{1}{\leq} & \ts \sum_{t=1}^T \tfrac{m_t^2}{ 1+\sum_{j=1}^t m_j^2 }  ~~\xxxx{}{\leq} ~~ \ln (1+ \MMM_T) ~~\xxxx{2}{\leq} ~~ \ts \ln (1+ \kappa \ln(e+T) (1+\GGG_T) ) \nn\\
&\xxxx{3}{\leq} & \ln (1+ \kappa \ln(e+T) (1+T \hG^2 ) ) ~~\leq~~  \ln (1+ \kappa (1+\hG^2) \ln(e+T) T ) \nn\\
&\xxxx{4}{\leq} & \ln (e+ \kappa (1+\hG^2) ) \ln (e + \ln(e+T) T  ) ~~\xxxx{5}{\leq} ~~ \underbrace{\ts 2\ln (e+ \kappa (1+\hG^2) ) }_{:=C_d}\ln(e+T) ,\nn
\eeq
\noi where step \step{1} uses $\gamma_t\leq\big(1+\sum_{j=1}^t\alpha_j^{-1/2}\|\M_j\|_{\fro}^2\big)^{-1/2} \leq\big(1+\sum_{j=1}^t m_j^2\big)^{-1/2}$; step \step{2} uses Lemma~\ref{lemma:OptMuonI:sumV:byG}; step \step{3} uses $\GGG_T\leq T\hG^2$; step \step{4} uses $\ln(1+ab)\leq \ln(e+a)\ln(e+b)$; step \step{5} uses $\ln (e + \ln(e+T) T  )\leq 2 \ln(e+T)$.

\end{proof}

\subsection{Proof of Lemma \ref{lemma:U:square:MuonI}}
\label{app:lemma:U:square:MuonI}

\begin{proof} Define $s_t:=\|\S_t\|_{\fro}$, and $g_t:=\|\G_t\|_{\fro}$.

\noi Define $\UUU_t:=\sum_{i=1}^t \gamma_i s_i^2$ and $\tilde{\UUU}_T :=\sum_{t=1}^{T}\sqrt{\alpha_t}s_t^2$, where $\UUU_0=\tilde{\UUU}_0=0$.

\noi Define $B_t:=2\alpha_t^2\|\epsilons_t\|_{\fro}^2+8 \delta \gamma_{t-1}^2m_{t-1}^2+2(1-\alpha_t)\left\langle \S_{t-1},\zetas_t\right\rangle$, where $\zetas_t:=\epsilons_t-(1-\alpha_t)\epsilons'_t$.

\noi Define $I_t:=2\alpha_t^{3/2}\|\epsilons_t\|_{\fro}^2$, $J_t:=8 \delta \frac{\gamma_{t-1}^2m_{t-1}^2}{\sqrt{\alpha_t}}$, $K_t:=\frac{2(1-\alpha_t)}{\sqrt{\alpha_t}}\left\langle\S_{t-1},\zetas_t\right\rangle$, and $\delta:=n L^2\theta^2$.

\noi Let $\III_T:=\sum_{t=1}^{T}I_t$, $\JJJ_T:=\sum_{t=1}^{T}J_t$, and $\KKK_T:=\sum_{t=1}^{T}K_t$, where $\III_0=\JJJ_0=\KKK_0=0$.

\noi Let $\frac{B_t}{\sqrt{\alpha_t}}=I_t+J_t+K_t$, where $B_t$ can be negative. Define $P_t:=\frac{s_t^2}{\sqrt{\alpha_t}}$.

\noi Define $R_t:=\frac{\gamma_{t-1}^2m_{t-1}^2}{\sqrt{\alpha_{t-1}}}$, and $P_{t-1}:=\tfrac{s_{t-1}^2}{\sqrt{\alpha_{t-1}}}$. Define $\RRR_T:=\sum_{t=1}^TR_t$.

Since \(\alpha_t\), \(\S_{t-1}\), and \(\X_t,\X_{t-1}\) are
\(\FF_{t-1}\)-measurable, and $\EEE_t[\epsilons_t]=0$, $\EEE_t[\epsilons'_t]=0$, we have $\EEE_t[K_t]=0$. Thus \(\{K_t\}\) is a martingale-difference sequence.

\paragraph{Bounding the Term $\tilde{\UUU}_{T-1}:=\sum_{t=1}^{T-1} \sqrt{\alpha_t}s_t^2$.} We have from Inequality (\ref{eq:pathwise:alpha:sqrt:st2}) that:
\beq
\ts \sum_{t=1}^{T-1} \sqrt{\alpha_t} s_t^2 ~~\leq ~~ \ts 6 P_0 + 6 \sum_{t=1}^{T}\tfrac{B_t}{\sqrt{\alpha_t}} ~~\xxxx{1}{=}~~ \ts 6 P_0 + 6\III_T + 6 \JJJ_T + 6 \KKK_T,\nn
\eeq
\noi where step \step{1} uses the definition of $B_t$, $\III_T$, $\JJJ_T$, and $\KKK_T$. Using $(a+b+c+d)^2\le 4(a^2+b^2+c^2+d^2)$, we have:
\beq
\tilde{\UUU}_{T-1}^2 \le 144 \left(P_0^2+\III_T^2 + \JJJ_T^2 + \KKK_T^2 \right).
\eeq

\paragraph{Bounding the Term $\EEE[\III_T^2]$.} Define $\III_T:=2\sum_{t=1}^{T}\alpha_t^{3/2}\|\epsilons_t\|_{\fro}^2$, and $Y_T:=\sum_{t=1}^{T}\frac{\|\epsilons_t\|_{\fro}^2}{1+\sum_{i=1}^t g_i^2}$. Applying Lemma \ref{lemma:weighted:self:normalized} with $W_{t-1}\equiv1$, we have:
\beq \label{eq:YT:III}
\EEE[Y_T]\le 8 (1+\hG^2) \ln(e+\hG^2) \ln(e+T).
\eeq
\noi Define $y_t:=\frac{\|\epsilons_t\|_{\fro}^2}{1+\sum_{i=1}^t g_i^2}$. We notice the following identity holds: $\ts Y_T^2= \sum_{t=1}^{T}y_t^2+2\sum_{t=1}^{T}Y_{t-1}y_t$. We derive:
\beq
\EEE[Y_T^2] &= & \ts \EEE[\sum_{t=1}^{T}y_t^2]+ 2\EEE[\sum_{t=1}^{T}Y_{t-1}y_t ] ~~\xxxx{1}{\leq} ~~ \ts \ts 4\hG^2 \EEE[\sum_{t=1}^{T}y_t] + 2\EEE[\sum_{t=1}^{T}Y_{t-1}y_t ] \nn\\
&\xxxx{2}{\leq}& \ts \ts 32 (1+\hG^2)^2 \ln(e+\hG^2) \ln(e+T) + 16 (1+\hG^2) \ln(e+\hG^2)\ln(e+T) \EEE[Y_{T-1}] \nn\\
&\xxxx{3}{\leq}& \ts \ts 32 (1+\hG^2)^2 \ln(e+\hG^2) \ln(e+T) + 16 (1+\hG^2)^2 \ln^2(e+\hG^2)\ln^2(e+T) \cdot 8  \nn\\
&\leq & \ts \ts 15^2 (1+\hG^2)^2 \ln^2(e+\hG^2)\ln^2(e+T) ,\nn
\eeq
\noi where step \step{1} uses $y_t \leq \|\epsilons_t\|_{\fro}^2\leq 4 \hG^2$; step \step{2} uses Lemma \ref{lemma:weighted:self:normalized} again with
\(W_{t-1}=Y_{t-1}\), and Inequality (\ref{eq:YT:III}); step \step{3} uses Inequality (\ref{eq:YT:III}) again.

We derive the following inequality:
\beq
\ts \III_T &=&\ts 2\sum_{t=1}^{T}\alpha_t^{3/2}\|\epsilons_t\|_{\fro}^2 ~~\xxxx{1}{\leq} ~~ \ts 6 \sum_{t=1}^{T}\alpha_{t+1}^{3/2}\|\epsilons_t\|_{\fro}^2 \nn\\
&\xxxx{2}{\leq} & \ts 6 (1+\hG^2) \sum_{t=1}^{T} \tfrac{\|\epsilons_t\|_{\fro}^2}{ 1 + \GGG_t} ~~=~~ 6 (1+\hG^2) Y_{T}, \nn
\eeq
\noi where step \step{1} uses $\tfrac{\alpha_t}{\alpha_{t+1}} \leq  1+q = \tfrac{5}{3}$; step \step{2} uses $\alpha_{t+1}^{3/2}\le \frac{1+\hG^2}{1+\sum_{i=1}^t g_i^2}$. Taking the expectation yields:
\beq \label{eq:II}
\EEE[\III_T^2] = 36 (1+\hG^2)^2\EEE[ Y_T^2] \leq \underbrace{\ts 36 (1+\hG^2)^2 \cdot 15^2 (1+\hG^2)^2 \ln^2(e+\hG^2)}_{:=C_I^2} \cdot \ln^2(e+T). \nn
\eeq

\paragraph{Bounding the Term $\EEE[\JJJ_T^2]$.} Define $\JJJ:=8 \delta \sum_{t=1}^{T}\frac{\gamma_{t-1}^2m_{t-1}^2}{\sqrt{\alpha_t}}$. We have from Inequality (\ref{eq:gamma:2:m:2:sqrt:alpha}):
\beq
\ts \sum_{t=1}^T\gamma_{t-1}^2m_{t-1}^2/\sqrt{\alpha_t} ~~ {\leq} ~~ 8 \ln(e+\kappa (1+\hG^2)) \ln(e+T)\quad a.s.
\eeq
Taking the expectation yields:
\beq \label{eq:JJ}
\EEE[\JJJ_T^2] = \underbrace{\ts (64 \delta)^2 \ln^2(e+\kappa (1+\hG^2))}_{:=C_J^2}\cdot \ln^2(e+T).
\eeq

\paragraph{Bounding the Term $\|\zetas_t\|_{\fro}^2$.} By individual smoothness and the update rule,
\beq
\|\epsilons_t-\epsilons'_t\|_{\fro}
&\le& \|\nabla f(\X_t;\xi_t)-\nabla f(\X_{t-1};\xi_t)\|_{\fro} + \|\nabla f(\X_t)-\nabla f(\X_{t-1})\|_{\fro} \nn\\
&\le& 2L\|\X_t-\X_{t-1}\|_{\fro} ~~\le~~ 2L\theta\gamma_{t-1}m_{t-1}\|\Orth(\M_{t-1})\|_{\fro} ~~\le~~ 2 \sqrt{\delta} \gamma_{t-1}m_{t-1}.\nn
\eeq
\noi Furthermore, we derive:
\beq \label{eq:zeta:bound:MuonI:new}
\|\zetas_t\|_{\fro}^2  &\xxxx{1}{=} & \ts \| \alpha_t\epsilons'_t+(\epsilons_t-\epsilons'_t) \|_{\fro}^2 ~~{\leq} ~~ \ts 2\alpha_t^2\| \epsilons'_t \|^2 + 2 \|\epsilons_t-\epsilons'_t\|_{\fro}^2 \nn\\
&\xxxx{2}{\leq} & 2 \alpha_t^2 (2\hG)^2 + 2 \|\epsilons_t-\epsilons'_t\|_{\fro}^2  ~~\leq ~~ 8\hG^2\alpha_t^2 + 8 \delta \gamma_{t-1}^2m_{t-1}^2.
\eeq
\noi where step \step{1} uses $\zetas_t=\epsilons_t-(1-\alpha_t)\epsilons'_t=\alpha_t\epsilons'_t+(\epsilons_t-\epsilons'_t)$; step \step{2} uses $\|\epsilons'_t\|_{\fro}\le 2\hG$, which is implied by Assumption \ref{ass:BG}.

\paragraph{Bounding the Term $\sum_{t=1}^T \EEE[K_t^2]$.} We derive:
\beq
\ts K_t^2 &\xxxx{1}{\le}& \ts \frac{4}{\alpha_t}s_{t-1}^2\|\zetas_t\|_{\fro}^2  ~~\xxxx{2}{\le}~~ 32\hG^2\alpha_t s_{t-1}^2 + 32 \delta \frac{\gamma_{t-1}^2m_{t-1}^2}{\alpha_t} s_{t-1}^2  \nn \\
& = &  \ts 32\hG^2\alpha_t s_{t-1}^2 + 32 \delta  \frac{\gamma_{t-1}^2m_{t-1}^2}{\sqrt{\alpha_{t-1}}} {\tfrac{\alpha_{t-1}}{\alpha_{t}}} \cdot \tfrac{s_{t-1}^2}{\sqrt{\alpha_{t-1}}}  \nn\\
& \xxxx{3}{\leq} &  \ts 32\hG^2 \varpi \alpha_{t-1} s_{t-1}^2 + 64 \delta R_t P_{t-1}, \nn
\eeq
\noi where step \step{1} uses $1-\alpha_t\le1$; step \step{2} uses Inequality (\ref{eq:zeta:bound:MuonI:new}); step \step{3} uses $\alpha_{t}\leq 1$, and $\tfrac{\alpha_{t-1}}{\alpha_t}\leq 2$. Taking the expectation and summing this inequality over $t$ from $1$ to $T$ yields:
\beq \label{eq:EKKK}
\ts \sum_{t=1}^T \EEE[K_t^2] &\leq & \ts \sum_{t=1}^T \EEE[32\hG^2 \varpi \alpha_{t-1} s_{t-1}^2 + 64 \delta R_t P_{t-1} ] \nn\\
&\leq & \ts 32\hG^2 \varpi s_{0}^2 + 32\hG^2 \varpi  \EEE[ \sum_{t=1}^T \alpha_{t} s_{t}^2] + 64 \delta \EEE\big[\sum_{t=1}^T   R_t P_{t-1} \big] \nn\\
&\leq & \ts 64\hG^2 \varpi C_c \ln(e+T) + 64 \delta  \EEE\big[ \sum_{t=1}^T   R_t P_{t-1} \big] . \nn
\eeq

\paragraph{Bounding the Term $\RRR_T$.} We have the following pathwise bound:
\beq
\ts \RRR_T
=\sum_{t=1}^T\frac{\gamma_{t-1}^2m_{t-1}^2}{\sqrt{\alpha_{t-1}}}
\xxxx{1}{\leq}
\sqrt{\varpi}\sum_{t=1}^T\frac{\gamma_{t-1}^2m_{t-1}^2}{\sqrt{\alpha_t}}
\xxxx{2}{\leq}
\underbrace{\ts 8\sqrt{\varpi}\ln(e+\kappa (1+\hG^2))}_{:=C_R}\cdot \ln(e+T)
\quad a.s.,
\eeq
\noi where step \step{1} uses $\alpha_t/\alpha_{t-1}\le\varpi$, and step \step{2} uses Inequality (\ref{eq:gamma:2:m:2:sqrt:alpha}).

\paragraph{Bounding the Term $\EEE\big[\sum_{t=1}^T R_tP_{t-1}\big]$.} Inequality (\ref{eq:to:be:cited}) implies that $P_t\le P_{t-1}+\frac{B_t}{\sqrt{\alpha_t}}$, where $\frac{B_t}{\sqrt{\alpha_t}}=I_t+J_t+K_t$. This leads to the following results:
\beq
\ts P_{t-1} \le P_0+\sum_{i=1}^{t-1}(I_i+J_i)+\sum_{i=1}^{t-1}K_i \le P_0+\III_T+\JJJ_T+\KKK_{t-1}.\nn
\eeq
\noi Multiplying by \(R_t\ge0\), taking expectations, and summing the resulting inequality over $t$ from $1$ to $T$ yields:
\beq
\label{eq:R:P:decomp:MuonI:new}
\ts \EEE[\sum_{t=1}^TR_tP_{t-1} ] &\le& \ts \EEE[\big(\sum_{t=1}^T R_t\big) (P_0+\III_T+\JJJ_T) + \big( \sum_{t=1}^TR_t\KKK_{t-1}\big)] \nn \\
&\xxxx{1}{\leq} & \ts   \EEE[  C_R \ln(e+T) \big(P_0 +\III_T+\JJJ_T \big) + \big( \sum_{t=1}^TR_t\KKK_{t-1}\big)  ) ]  \nn \\
&= & \ts   C_R \ln(e+T) \big(P_0 + \EEE[\III_T + \JJJ_T ]\big) + \EEE[ \sum_{t=1}^TR_t\KKK_{t-1} ] \nn \\
& \xxxx{2}{\leq} & \ts C_R \big(P_0 + {C_I}+{C_J} \big)\ln^2(e+T) + \EEE[ \sum_{t=1}^TR_t\KKK_{t-1} ],
\eeq
\noi where step \step{1} uses the pathwise bound $\sum_{t=1}^T R_t \leq C_R \ln(e+T)$; step \step{2} uses $\EEE[\III_T] \leq \sqrt{ \EEE[ (\III_T)^2] }$, and $\EEE[\JJJ_T] \leq \sqrt{ \EEE[ (\JJJ_T)^2] }$.

\paragraph{Bounding the Term $\EEE[ \sum_{t=1}^TR_t\KKK_{t-1} ]$.} Let $\RRR_t:=\sum_{i=1}^tR_i$. Since $\KKK_t=\KKK_{t-1}+K_t$ and $\RRR_t=\RRR_{t-1}+R_t$, we have $\RRR_T\KKK_T=\sum_{t=1}^T\RRR_tK_t+\sum_{t=1}^TR_t\KKK_{t-1}$, leading to:
\beq
\ts \sum_{t=1}^TR_t\KKK_{t-1} ~~=~~ {\RRR_T\KKK_T-\sum_{t=1}^T\RRR_tK_t}.\nn
\eeq
\noi Taking expectations, and using that \(\RRR_t\) is \(\FF_{t-1}\)-measurable and \(\EEE_t[K_t]=0\), we get
\beq \label{eq:AT:MT:bound:MuonI:new}
&&\ts \EEE\left[\sum_{t=1}^TR_t\KKK_{t-1}\right] = \EEE[ \RRR_T\KKK_T]~~\le~~
\EEE[ \RRR_T\cdot |\KKK_T|]\nn\\
&\le& C_R \ln(e+T)\EEE[|\KKK_T|] ~~\le~~ C_R\ln(e+T)\sqrt{\EEE[\KKK_T^2]} .
\eeq

\paragraph{Final Bound.} Combining \eqref{eq:R:P:decomp:MuonI:new} and
\eqref{eq:AT:MT:bound:MuonI:new}, we obtain
\beq \label{eq:R:P:expect:MuonI:new}
\ts \EEE\left[\sum_{t=1}^TR_tP_{t-1}\right] \le C_R (P_0 + {C_I}+ {C_J}) \ln^2(e+T) + C_R \ln(e+T) \sqrt{\EEE[\KKK_T^2]}.
\eeq

Define $C_w:= 64 \hG^2\varpi C_c + 64 \delta C_R\big(\hG^2+{C_I}+{C_J}\big)$. Since $\EEE[\KKK_T^2] =\sum_{t=1}^T \EEE[K_t^2]$, we derive from Inequality (\ref{eq:EKKK}):
\beq
\ts \EEE[\KKK_T^2] &\leq& \ts 64\hG^2 \varpi C_c \ln(e+T) + 64 \delta  \EEE\big[ \sum_{t=1}^T   R_t P_{t-1} \big] \nn\\
&\xxxx{1}{\leq} & \ts C_w \ln^2(e+T) + 64 \delta C_R \ln(e+T) \sqrt{\EEE[\KKK_T^2]} \nn\\
&\xxxx{}{\leq} & \ts C_w \ln^2(e+T) + C_w \ln(e+T) \sqrt{\EEE[\KKK_T^2]} \nn\\
&\xxxx{}{\leq} & \ts \max\left(2 C_w \ln^2(e+T), 2 C_w \ln(e+T) \sqrt{\EEE[\KKK_T^2]} \right) \nn\\
&\xxxx{}{\leq} & \ts \underbrace{\ts 4 (C_w^2+C_w) }_{:=C_K}\ln^2(e+T), \nn
\eeq
\noi where step \step{1} uses Inequality (\ref{eq:R:P:expect:MuonI:new}).

Since the preceding argument is valid for any integer horizon, applying it with
\(T+1\) in place of \(T\) and using $\ln^2(e+T+1)\le 2\ln^2(e+T)$ for $T\geq 1$,
we obtain
\beq
&&\EEE[\III_{T+1}^2]\le 2C_I^2\ln^2(e+T),~~\EEE[\JJJ_{T+1}^2]\le 2C_J^2\ln^2(e+T),~~\EEE[\KKK_{T+1}^2]\le 2C_K\ln^2(e+T).\nn
\eeq

Finally, we have
\beq
\EEE[\UUU_T^2] &\le& \EEE[\tilde{\UUU}_T^2] ~~\le~~ 144\left(P_0^2+\EEE[\III_{T+1}^2] +\EEE[\JJJ_{T+1}^2] +\EEE[\KKK_{T+1}^2]\right) \nn\\
&\le& \underbrace{288\left(\hG^4+C_I^2+C_J^2+C_K\right)}_{:=C_u}
\cdot \ln^2(e+T).\nn
\eeq

\end{proof}

\subsection{Proof of Lemma \ref{lemma:weighted:bridge:MuonI}}
\label{app:lemma:weighted:bridge:MuonI}

\begin{proof}
Let $F_t:=f(\X_t)-f_*$, $\UUU_T:=\sum_{t=1}^T\gamma_t s_t^2$, $\WWW_T:=\sum_{t=1}^T\gamma_t m_t^2$, $\MMM_T:=\sum_{t=1}^T m_t^2$.

\noi Let $m_t=\|\M_t\|_{\fro}$ and $s_t=\|\S_t\|_{\fro}$.

\paragraph{Bounding the Term $\EEE[\WWW_T]$.} By Lemma \ref{lemma:suff:descent:I},
\beq
\ts f(\X_{t+1})-f(\X_t)+c_0\gamma_t m_t^2 \le c_1\gamma_t s_t^2+c_2\gamma_t^2m_t^2.
\eeq
Rearranging and summing over \(t=1,\ldots,T\), we get
\beq \label{eq:MuonI:pathwise:W}
\ts \WWW_T \le  \frac{1}{c_0}
\left(F_1 + c_1\UUU_T + c_2\sum_{t=1}^T\gamma_t^2m_t^2 \right).
\eeq
Taking expectations and using Lemma \ref{lemma:bcd:MuonI}(b) and Lemma \ref{lemma:bcd:MuonI}(c) gives
\beq
\EEE[\WWW_T] &\le& \tfrac{1}{c_0}\left( F_1+c_1C_c\ln(e+T)+c_2C_d\ln(e+T)\right)\nn\\
&\xxxx{1}{\leq} & \underbrace{ \ts \frac{1}{c_0}\big(F_1+c_1C_c+c_2C_d\big)
}_{:=C_w} \cdot \ln(e+T).
\eeq
\noi where step \step{1} uses $\ln(e+T)\ge 1$.

\paragraph{Bounding the Term $\EEE[\WWW_T^2]$.} Squaring \eqref{eq:MuonI:pathwise:W} and using
$(a+b+c)^2\le 3(a^2+b^2+c^2)$, we obtain
\beq
\ts  \WWW_T^2 \le \tfrac{3}{c_0^2} \left( F_1^2 + c_1^2\UUU_T^2 + c_2^2 \left(\sum_{t=1}^T\gamma_t^2m_t^2\right)^2\right).
\eeq
Taking expectations, using Lemma \ref{lemma:U:square:MuonI} and
Lemma \ref{lemma:bcd:MuonI}(c), yields
\beq \EEE[\WWW_T^2] &\le& \tfrac{3}{c_0^2}
\left( F_1^2 + c_1^2 C_u\ln^2(e+T) + c_2^2 C_d^2\ln^2(e+T)
\right) \nn\\                                                   &\le& \underbrace{\tfrac{3}{c_0^2} \left( F_1^2+c_1^2C_u+c_2^2C_d^2 \right)}_{:=C_x}
\ln^2(e+T).
\eeq

\paragraph{Bounding the Term $\alpha_T^{1/2}\MMM_T$.} We derive:
\beq
&& \ts   \ts {\alpha_T}^{1/2}\MMM_T \nn\\
&\leq & \ts  \sqrt{\alpha_T} \cdot \sum_{t=1}^T m_t^2 ~~ \xxxx{1}{\leq}~~ \ts   \sqrt{\varpi}\tfrac{\sqrt{\alpha_T}}{\gamma_T} \sum_{t=1}^T \gamma_t m_t^2 ~~=~~ \sqrt{\varpi} \WWW_T \cdot \tfrac{\sqrt{\alpha_T}}{\gamma_T}  \nn\\
&\xxxx{2}{\leq}& \ts \sqrt{\varpi} \WWW_T \max\left\{1, \sqrt{\alpha_T}\sqrt{1+\sum_{j=1}^T\alpha_j^{-1/2}m_j^2}\right\} ~~\xxxx{3}{\leq}~~ \ts {\varpi} \WWW_T      \max\left\{1, \sqrt{\alpha_T}\sqrt{1+ \alpha_T^{-1/2} \MMM_T }\right\} \nn\\
&\xxxx{}{=}& \ts {\varpi} \WWW_T      \max\left\{1,  \sqrt{\alpha_T + \alpha_T^{1/2} \MMM_T }\right\} ~~\xxxx{}{\leq}~~ \ts 2 {\varpi} \WWW_T \max\big( 2 , \sqrt{\alpha_T^{1/2} \MMM_T } \big) \nn\\
&{\leq}& \ts \max \big( 4 \varpi \WWW_T, 4\varpi^2\WWW_T^2 \big) \leq  4 \varpi^2 (\WWW_T +\WWW_T^2 ),\nn
\eeq
\noi where step \step{1} uses the fact that $\gamma_T\leq \sqrt{\varpi} \gamma_t$; step \step{2} uses $\tfrac{1}{\gamma_t} = \max(\alpha_t^{-1/2},(1+\sum_{j=1}^t \alpha_j^{-1/2}m_j^2)^{1/2})$; step \step{3} uses $\alpha_t^{-1/2}\leq \sqrt{\varpi}\,\alpha_T^{-1/2}$ for $t\le T$, with the additional factor absorbed into the leading constant. Taking expectations, we conclude that
\beq
  \ts \EEE[\alpha_T^{1/2}\MMM_T] &\le& \ts 4\varpi^2 \big( \EEE[\WWW_T] + \EEE[\WWW_T^2] \big) ~~\leq ~~4 \varpi^2 \left( C_w\ln(e+T) + C_x\ln^2(e+T)\right)  \nn \\
&\le& \underbrace{ 4 \varpi^2  (C_w+C_x)}_{:=C_v}\ln^2(e+T).\nn
\eeq
\end{proof}

\subsection{Proof of Theorem \ref{the:it:MuonI}}
\label{app:the:it:MuonI}

\begin{proof} Define $m_t=\|\M_t\|_{\fro}$, $s_t=\|\S_t\|_{\fro}$, and $g_t=\|\G_t\|_{\fro}$.

\noi Define $\MMM_T:=\sum_{t=1}^T m_t^2$, $\SSS_T:=\sum_{t=1}^T s_t^2$, and $\GGG_T:=\sum_{t=1}^T g_t^2$.

\noi Let $\epsilons_t:=\G_t-\nabla f(\X_t)$, and $\ZZ_T:=\EEE\left[(1+\GGG_T)^{1/3}\right]$.

\paragraph{Bounding the Term $\EEE[\frac1{\sqrt{\alpha_T}}]$.} Since \(q=2/3\), the definition of \(\alpha_T\) gives
\beq
\ts \frac1{\sqrt{\alpha_T}} ~~=~~
\left(
\frac{1+\sum_{i=1}^{T-1} g_i^2}{1+\max_{1\le i\le {T-1}}g_i^2}
\right)^{1/3}~~\le~~ \left(1+\sum_{i=1}^Tg_i^2\right)^{1/3}
~~=~~ (1+\GGG_T)^{1/3}.\nn
\eeq
\noi This leads to $\EEE\big[\frac1{\sqrt{\alpha_T}}\big]\leq\ZZ_T$.

\paragraph{Bounding the Momentum Sum.} By Lemma \ref{lemma:weighted:bridge:MuonI} and Cauchy--Schwarz,
\beq \label{eq:MuonI:sqrt:M}
\ts \EEE[\sqrt{\MMM_T}] &=& \ts  \EEE\left[\sqrt{\sqrt{\alpha_T}\MMM_T\cdot\frac1{\sqrt{\alpha_T}}}\right] ~~\le~~ \ts  \sqrt{\EEE[\sqrt{\alpha_T}\MMM_T] \cdot \EEE\left[\frac1{\sqrt{\alpha_T}}\right]} \nn \\
&\le & \ts \sqrt{C_v\ln^2(e+T)\ZZ_T}=\sqrt{C_v}\ln(e+T)\sqrt{\ZZ_T}.
\eeq

\paragraph{Bounding the Tracking-Error Sum.} Because $\alpha_T\leq \varpi \alpha_t$ for $t\le T$, $\SSS_T=\sum_{t=1}^T s_t^2 \le \frac{\sqrt{\varpi}}{\sqrt{\alpha_T}} \sum_{t=1}^T\sqrt{\alpha_t}s_t^2$. Using Cauchy--Schwarz, Lemma \ref{lemma:bcd:MuonI}(b), and $\EEE\left[\frac1{\sqrt{\alpha_T}}\right]\leq\ZZ_T$, we obtain
\beq \label{eq:MuonI:sqrt:S}
\ts \EEE[\sqrt{\SSS_T}] ~~\le~~ \sqrt{ \sqrt{\varpi} \EEE\left[\sum_{t=1}^T\sqrt{\alpha_t}s_t^2\right] \cdot \EEE\left[\frac1{\sqrt{\alpha_T}}\right]}    ~~\le~~ \sqrt{ \sqrt{ \varpi} C_c\ln(e+T)\ZZ_T}.
\eeq

\paragraph{Bounding the Gradient-Energy Term.} Since $\nabla f(\X_t)=\M_t-\S_t$, we have $\|\nabla f(\X_t)\|_{\fro}^2\le 2m_t^2+2s_t^2$. Therefore, by \eqref{eq:MuonI:sqrt:M} and \eqref{eq:MuonI:sqrt:S},
\beq \label{eq:MuonI:A:Z}
\ts \AAA_T :=\EEE\left[\sqrt{\sum_{t=1}^T\|\nabla f(\X_t)\|_{\fro}^2}
\right] ~~\le~~ \sqrt{2}\EEE[\sqrt{\MMM_T}] +  \sqrt{2}\EEE[\sqrt{\SSS_T}]     ~~\le ~~ {C_{z}\ln(e+T)\sqrt{\ZZ_T}},
\eeq
where $C_{z}:=\sqrt{2C_v}+ \sqrt{2 C_c} \varpi^{1/4} \geq 1$.

\paragraph{Self-Consistent Bound on \(\ZZ_T\).} We derive:
\beq
\ZZ_T & \xxxx{1}{:=}&  \ts \EEE[(1+\GGG_T)^{1/3}] \nn \\
&\ts\xxxx{2}{\le}& \ts 1 + 2^{1/3}
\EEE\big[ \big(\sum_{t=1}^T\|\nabla f(\X_t)\|_{\fro}^2\big)^{1/3}\big] + 2^{1/3} \EEE\big[\big(\sum_{t=1}^T\|\epsilons_t\|_{\fro}^2\big)^{1/3}\big] \nn \\
&\xxxx{3}{\le}& \ts 1 + 2\AAA_T^{2/3} + 2 (T\sigma^2)^{1/3} ~~\leq ~~ 1 + 2 C_z^{2/3} \ln^{2/3} (e+T) \ZZ_T^{1/3} + 2 (T\sigma^2)^{1/3}  \nn\\
&\le& \ts 2 \max\left( 1 + 2 (T\sigma^2)^{1/3}, 2 C_z^{2/3} \ln^{2/3} (e+T) \ZZ_T^{1/3}   \right) \nn\\
&\le& \ts \max\left( 2 + 4 (T\sigma^2)^{1/3}, (4 C_z^{2/3} \ln^{2/3} (e+T) )^{3/2}  \right) \nn\\
&\le& \ts 2 + 4 (T\sigma^2)^{1/3} +  4^{3/2} C_z \ln(e+T)  \nn\\
&\xxxx{4}{\le}& \ts 4 (T\sigma^2)^{1/3} +  10 C_z \ln(e+T),  \label{eq:final:final}
\eeq
\noi where step \step{1} uses the definition of $\ZZ_T$; step \step{2} uses $\GGG_T \le 2\sum_{t=1}^T\|\nabla f(\X_t)\|_{\fro}^2 + 2\sum_{t=1}^T\|\epsilons_t\|_{\fro}^2$ (which is implied by $\G_t=\nabla f(\X_t)+\epsilons_t$), and $(a+b+c)^{1/3}\le a^{1/3}+b^{1/3}+c^{1/3}$; step \step{3} uses Jensen's inequality, and Assumption \ref{ass:sigma}; step \step{4} uses $C_z\geq 1$.

\paragraph{Final Bound.} By Cauchy--Schwarz,
\beq
\ts \EEE\left[ \frac1T\sum_{t=1}^T\|\nabla f(\X_t)\|_{\fro}\right] \le \frac1{\sqrt{T}}
\EEE\left[\sqrt{\sum_{t=1}^T\|\nabla f(\X_t)\|_{\fro}^2}\right].\nn
\eeq
Using Inequalities (\ref{eq:MuonI:A:Z}) and (\ref{eq:final:final}), we get
\beq
\ts  \EEE\left[
\frac1T\sum_{t=1}^T\|\nabla f(\X_t)\|_{\fro}
\right] &\le& \frac{ C_{z}\ln(e+T)}{\sqrt{T}} \sqrt{ 10 C_{z} \ln(e+T)+ 4 \sigma^{2/3} T^{1/3} } \nn \\
&\le& 4 C_z^{2} \frac{\ln^{3/2}(e+T)}{T^{1/2}} + 2 C_z \frac{\sigma^{1/3}\ln(e+T)}{T^{1/3}}.\nn
\eeq
Finally, if \(R\sim{\rm Unif}\{1,\ldots,T\}\) is independent of the algorithmic
randomness and \(\X_{\rm out}:=\X_R\), then
\beq
\ts \EEE[\|\nabla f(\X_{\rm out})\|_{\fro}] = \EEE\left[
\frac1T\sum_{t=1}^T\|\nabla f(\X_t)\|_{\fro}\right],\nn
\eeq
which proves the output guarantee.
\end{proof}

\end{document}